\documentclass[final,3p,times]{elsarticle}
\usepackage{lineno}

\usepackage{mathtools}
\usepackage{amsmath,amssymb,amsfonts,amsthm}
\newtheorem{theorem}{Theorem}
\newtheorem{proposition}{Proposition}
\newtheorem{corollary}{Corollary}
\newtheorem{remark}{Remark}
\newtheorem{definition}{Definition}
\usepackage{bm}
\usepackage{hyperref}
\usepackage[capitalize]{cleveref}
\crefname{equation}{}{}
\Crefname{equation}{}{}

\usepackage{subcaption}
\usepackage{algorithm}
\usepackage{algpseudocode}
\algrenewcommand\algorithmicrequire{\textbf{Input:}}
\algrenewcommand\algorithmicensure{\textbf{Output:}}
\usepackage{tikz}
\usetikzlibrary{spy,shapes,arrows,positioning,calc,decorations.markings}
\usepackage{pgfplots}
\pgfplotsset{compat=1.18}
\usepgfplotslibrary{colorbrewer} 
\usepackage{pgfplotstable}
\usepackage{booktabs}
\usepackage{colortbl}
\usepackage{makecell}

\newcommand*\circled[1]{\tikz[baseline=(char.base)]{\node(char)[shape=rounded rectangle,draw,inner sep=0.6pt,minimum height=1.5ex]{#1};}} %
\newcommand\kronF[2]{{#1}^{\circled{\tiny{\ensuremath{#2}}}}} 
\renewcommand\Vec[1][]{\textnormal{vec}\ensuremath{\if$#1$ \else \left[#1\right]\fi}}

\newcommand{\real}{\mathbb{R}}

\newcommand{\rd}{\text{\upshape d}} 

\newcommand{\bzero}{\ensuremath{\mathbf{0}}} 

\newcommand{\bA}{\ensuremath{\mathbf{A}}}
\newcommand{\bB}{\ensuremath{\mathbf{B}}}
\newcommand{\bC}{\ensuremath{\mathbf{C}}}
\newcommand{\bD}{\ensuremath{\mathbf{D}}}
\newcommand{\bE}{\ensuremath{\mathbf{E}}}
\newcommand{\bF}{\ensuremath{\mathbf{F}}}
\newcommand{\bG}{\ensuremath{\mathbf{G}}}

\newcommand{\bI}{\ensuremath{\mathbf{I}}}
\newcommand{\bJ}{\ensuremath{\mathbf{J}}}
\newcommand{\bK}{\ensuremath{\mathbf{K}}}

\newcommand{\bM}{\ensuremath{\mathbf{M}}}

\newcommand{\bQ}{\ensuremath{\mathbf{Q}}}
\newcommand{\bR}{\ensuremath{\mathbf{R}}}

\newcommand{\bT}{\ensuremath{\mathbf{T}}}

\newcommand{\bV}{\ensuremath{\mathbf{V}}}

\newcommand{\bZ}{\ensuremath{\mathbf{Z}}}

\newcommand{\bb}{\ensuremath{\mathbf{b}}}

\renewcommand{\bf}{\ensuremath{\mathbf{f}}}
\newcommand{\bg}{\ensuremath{\mathbf{g}}}

\newcommand{\bq}{\ensuremath{\mathbf{q}}}
\newcommand{\br}{\ensuremath{\mathbf{r}}}

\newcommand{\bu}{\ensuremath{\mathbf{u}}}
\newcommand{\bv}{\ensuremath{\mathbf{v}}}

\newcommand{\bx}{\ensuremath{\mathbf{x}}}
\newcommand{\by}{\ensuremath{\mathbf{y}}}
\newcommand{\bz}{\ensuremath{\mathbf{z}}}

\newcommand{\cG}{\ensuremath{\mathcal{G}}}

\newcommand{\cJ}{\ensuremath{\mathcal{J}}}

\newcommand{\cL}{\ensuremath{\mathcal{L}}}

\newcommand{\cV}{\ensuremath{\mathcal{V}}}

\newcommand{\cf}{\ensuremath{\bm{f}}}

\journal{Journal of Computational Physics}

\begin{document}

\begin{frontmatter}
    \title{Nonlinear Feedback Control in High Dimensions using the Polynomial-Polynomial Regulator} 
    \author{Nicholas A. Corbin} 
    \author{Boris Kramer} 
    \affiliation[1]{organization={Department of Mechanical and Aerospace Engineering, University of California San Diego},
        addressline={9500 Gilman Drive},
        city={La Jolla},
        postcode={92093},
        state={CA},
        country={USA}}

    \begin{abstract}
        We consider the problem of computing Taylor approximations to the optimal feedback law and value function for infinite-horizon finite-dimensional nonlinear optimal control problem---also known as Al'brekht's method.
        The scalability of existing Al'brekht-based approaches to computing approximate solutions to the Hamilton-Jacobi-Bellman equation associated with the optimal control problem has been severely limited by computational resources.
        The computational time issue has been partially mitigated by recent works leveraging tensor-structure to develop efficient numerical methods.
        However, RAM requirements still have remained a major hindrance to scaling beyond state dimensions of a few hundred or low thousands.
        We propose an acceleration to compute the expensive higher-order terms in a reduced space, mitigating the curse of dimensionality.
        A fast and scalable \textsc{Matlab} implementation is provided open-source for general purpose use, with the ability to accommodate arbitrary polynomial nonlinearities in the control-affine state equation, as well as polynomial state-dependence in the cost function, and the ability to handle systems in generalized form with a non-identity mass matrix, such as finite element models.
        Combined with the exploitation of sparsity and modern low-rank Riccati solvers, this permits computing approximations to polynomial feedback laws based on Al'brekht's method in hundreds of thousands of states.
    \end{abstract}




\end{frontmatter}


\section{Introduction}\label{ch3:sec:introduction}
Hamilton-Jacobi-Bellman partial differential equations (HJB PDEs) are ubiquitous in many active areas of research, including optimal control, safety-critical control, and model order reduction.
HJB PDEs have been studied thoroughly for decades (see \cite{Bardi2013} for an overview).
Entire fields have been devoted to the theoretical study of solutions to HJB PDEs, even for highly nonlinear systems for which the HJB PDE may not admit a solution in the classical sense, leading to the notion of viscosity solutions \cite{Crandall1983}.
And yet, commonly available solvers for HJB PDEs are far from satisfactory, often only being suitable for examples with a few degrees of freedom.
This is due to the \textit{curse of dimensionality}, which refers to the exponential increase in the computational complexity of solving the HJB PDE with respect to increasing the state dimension.

It is true that for the most general class of nonlinear systems,
as well as for systems featuring state and control constraints,
the solutions to HJB PDEs are very irregular, featuring kinks and other strongly nonlinear, non-smooth characteristics that make computing solutions very challenging.
For such systems, the study of viscosity solutions is necessary, and sophisticated and mature solvers exist to approximate solutions in dimensions on the order of $n=5$ to $n=14$, including grid-based and semi-Lagrangian schemes \cite{Falcone2014,Falcone2016}, sparse grid adaptations \cite{Kang2017}, iterative schemes based on value/policy iteration and Galerkin approximation \cite{Beard1997,Kalise2018}, and low-rank methods such as those based on tensor-train representations \cite{Gorodetsky2018}.
There are also a variety of data-driven \cite{Dolgov2023,Ehring2024} and machine learning methods \cite{NakamuraZimmerer2021}, along with of course the open-loop methods \cite{Ito2001} and other related approaches which devise ways to avoid explicitly solving the HJB PDE, such as model predictive control \cite{Rawlings2020}.

On the other hand, there exist many applications in which the nonlinearities are mild and the HJB solutions are at least locally well-posed and smooth.
For such systems, polynomial expansions have long been known to be suitable for computing local nonlinear approximations to the HJB solutions with improved tractability.
Commonly known as Al'brekht's method \cite{Albrekht1961}, the approach of computing a Taylor expansion to the solutions to the HJB PDE for sufficiently smooth nonlinear systems provides a straightforward improvement over linearization (see \cite{Krener2013} for a survey).
Krener's Nonlinear Systems Toolbox \cite{Krener2019}, originally introduced in 1997, provided an automated symbolic implementation of Al'brekht's method, although its scalability to high-dimensional systems was still limited.
Scalability began to improve with the development of specialized linear solvers \cite{Grasedyck2004,Chen2019}, bringing renewed interest to Al'brekht's method in a series of recent works \cite{Breiten2018,Breiten2019a,Almubarak2019,Borggaard2018,Borggaard2020,Borggaard2021,Kramer2024}.
These enabled scaling to state dimensions as high as $n=20$ for systems with polynomial drift nonlinearities and quadratic cost functions in \cite{Borggaard2021}, and up to $n=64$ for the special case of quadratic drift nonlinearities in \cite{Borggaard2020}.
In \cite{Breiten2018}, using a related approach, the authors demonstrate scalability as high as $n=100$ for bilinear systems.
Extending that approach from bilinear systems to quadratic systems, the authors in \cite{Breiten2019a} compute a polynomial control law to stabilize the two-dimensional Navier-Stokes equations.
The full-order model considered therein is of dimension $n=O(10{,}000)$, but the authors circumvent the dimensionality via a specialized Leray projection that is possible for the Navier-Stokes equation \cite{Benner2013b}.
For the specific problem considered,
this permits them to directly compute an approximation of the feedback law
without ever explicitly forming the value function or feedback law coefficients.
Our recent work \cite{Corbin2024e} built on the general tensor algebra approaches of \cite{Borggaard2020,Borggaard2021} to generalize to polynomial drift \textit{and} polynomial input nonlinearities, along with polynomial state penalties in the cost function, as is permitted in the theory provided in \cite{Albrekht1961,Lukes1969}.
Scalability was also increased to $n=129$ for a degree 4 value function approximation.
With these improvements, our previous works permitted the design of nonlinear feedback controllers for fairly general control-affine nonlinear systems with identity mass matrices with a high degree of scalability relative to other HJB approaches.
Scalability beyond the low hundreds of states remained an issue though.

More-so than computational time, memory usage is the main hindrance to further scalability of Al'brekht's method.
The number of unknowns in each value function coefficient scales exponentially with respect to the state dimension and polynomial degree, so while recent research efforts \cite{Borggaard2020,Borggaard2021,Kramer2024,Corbin2024e,Corbin2025a} have accelerated the computations using specialized linear solvers that leverage the tensor structure of the equations, memory usage has remained an open challenge.
For example, a degree~4 coefficient for a state-dimension of $n=1000$ requires solving a linear system with $1000^4$ unknowns.
Using IEEE double-precision floating-point arithmetic, each number requires 64 bits or 8 bytes;
hence even just the right-hand-side vector $\bb$ in $\bA\bx=\bb$ requires $8 \times1000^4$ bytes,
or 8 TB of RAM.
As a result, all of the Al'brekht-based works we are aware of hit limits on the order of $n=O(100)$ for a degree~4 approximation without additional approximations---significantly higher than $n=5$ to $n=10$, but still nowhere near the scale of typical PDE semi-discretizations such as finite element models.

A common approach that has been used to mitigate
memory and other computational obstacles when dealing with large models
is model order reduction, typically by balanced truncation or proper orthogonal decomposition (POD) when solving HJB PDEs \cite{Kunisch2004,Alla2017,Breiten2018}.
Rather than computing the value function and optimal control for a full-order model of dimension $n$, a reduced-order model of dimension $r \ll n$
can be computed and used to approximate the value function and optimal feedback law on the reduced model.
This is sometimes called the \textit{reduce-then-design} approach.
Depending on the method to compute HJB solutions on the reduced-order model, more or less reduction may be necessary.
In \cite{Kunisch2004}, a reduced-order model of dimension $r=4$ is computed using POD, after which the HJB PDE is solved on a discretized grid.
In \cite{Alla2017}, the full-order models are reduced to dimension $r=4$ using balanced truncation and POD, and
the HJB PDE is solved using a semi-Lagrangian scheme and value iteration.
The authors also propose a reduction based on the solution to the algebraic Riccati equation, which is closely related to a part of the approach proposed in this work.
In \cite{Breiten2018}, bilinear balanced truncation \cite{Benner2011} is used to reduced the dimension of the model to a range $r=O(100)$, after which Al'brekht's method is used to compute a Taylor approximation of the value function and feedback laws.

A criticism of the POD approach is the requirement for simulation data to perform the reduction, as these simulations may be difficult to perform for expensive models and are not guaranteed to capture the relevant dynamics, especially in the case of unstable plant models for which a stabilizing control is to be designed.
Balanced truncation similarly poses issues for unstable plant models, as standard balancing is defined for asymptotically stable systems.
Furthermore, balanced truncation is defined for input-output systems with a measurement equation, whereas the optimal state feedback problem does not have a measurement equation defined.
There are closed-loop balancing formulations that extend to unstable plant models \cite{Jonckheere1983,Mustafa1991}, but to the authors' knowledge their usage in the context of attempting to compute solutions to HJB PDEs has never been explored and likely introduces additional challenges, both theoretical and practical.
Regardless of the approach used, the existing works feature a major drawback: in the pursuit of making the computations of the \textit{nonlinear} components of the feedback law solution feasible and more efficient, significant accuracy in the \textit{linear} component of the solution may be lost due to the reduction, especially in the case of a severe reduction.

In this work, we build on our previous works on Al'brekht's method \cite{Corbin2024e,Corbin2025a} to enable scalability to high-dimensional, sparse models with non-trivial mass matrices, such as those obtained by finite element discretization.
Compared to our existing works \cite{Corbin2024e,Corbin2025a}, we provide both theoretical and algorithmic improvements to the implementation of Al'brekht's method to achieve better scalability.
Our first contribution is the adaptation of the theory of Al'brekht to models with generalized dynamics featuring a nonsingular mass matrix, which permits exploiting sparsity and low-rank structure often found in semi-discretized PDE models.
The second major contribution of this work is a novel dimensionality reduction approach tailored specifically to the Taylor series method of Al'brekht.
The proposed reduction approach retains the accuracy and stability properties of the full-order optimal control linearization locally and only applies reduction to the higher-order nonlinear components of the feedback law.
The third contribution is an algorithm that also integrates with existing low-rank generalized Riccati solvers \cite{Benner2017b,Saak2022} to greatly accelerate the computations of the nonlinear feedback laws. The presented results
demonstrate the capability to design nonlinear feedback laws in state dimensions on the order of $O(100{,}000)$, which is several orders of magnitude higher than existing approaches.
The fast, scalable, and user-friendly \textsc{Matlab} implementation is available in the open-source Polynomial-Polynomial Regulator (PPR) repository:
\begin{center}
    \url{https://github.com/cnick1/PPR}
\end{center}
The exact version of the code used in this paper can be found under the \texttt{v1.0.1} tag for reproducibility, along with documentation in the \textsc{readme} file and all of the numerical examples.

The paper is organized as follows.
After defining some preliminary notation in \cref{ch3:sec:preliminaries},
we describe Al'brekht's method for computing Taylor series approximations to the HJB PDEs arising in the optimal control problem in \cref{ch3:sec:PPR-Poly}.
This section contains our first major contributions, which include the generalization to systems with non-identity mass matrices and full polynomial state nonlinearities in the state equation and cost function.
The associated computational approach, summarized in \cref{ch3:alg:full-PPR}, permits computing for example
a cubic Taylor approximation to the optimal feedback law for models with state dimensions on the order of $n=O(100)$ before reaching RAM limitations.
In \cref{ch3:sec:reduced-Albrekht}, we present our second main result, which is the proposed method for overcoming these memory limitations using model order reduction, enabling scalability to models with state dimensions on the order of $n=O(100{,}000)$.
We illustrate the method first on some low-dimensional examples in \cref{ch3:sec:examples}, after which in \cref{ch3:sec:allen-cahn} we demonstrate the ability to scale to high-dimensional semi-discretized PDE models.
Finally, we summarize our conclusions and future directions in \cref{ch3:sec:conclusions}.

\section{Preliminaries}\label{ch3:sec:preliminaries}
Throughout this work,
matrices are denoted by capital bold letters, vectors by lowercase bold letters, and scalars by lowercase letters.
For a square matrix $\bA$,
we use the notation $\bA \succ 0$ to indicate that it is positive definite, and $\bA \succeq 0$ to indicate that it is positive semidefinite.
As a matter of convention, we consider the \emph{gradient} $\partial V(\bx)/\partial \bx$ of a scalar-valued differentiable function
$V:\real^n \to \real$ to be a \emph{column} vector.
The \textit{Jacobian} $\partial \cf(\bu)/\partial \bu$ of a vector-valued differentiable function $\cf:\real^m \to \real^n$ is an $n \times m$ matrix-valued function; when $n=1$, the Jacobian corresponds to the \emph{transpose} of the gradient.
The Kronecker product of two matrices $\bA \in \real^{p \times q}$ and $\bB \in \real^{s \times t}$ is the $ps \times qt$ block matrix
\begin{align*}
    \bA \otimes \bB \coloneqq \begin{bmatrix} a_{11}\bB & \cdots & a_{1q}\bB \\
                \vdots    & \ddots & \vdots    \\
                a_{p1}\bB & \cdots & a_{pq}\bB
                              \end{bmatrix},
\end{align*}
where $a_{ij}$ denotes the $(i,j)$th entry of $\bA$.
We write repeated Kronecker products compactly as
\begin{equation*}
    \kronF{\bx}{k} \coloneqq \underbrace{\bx \otimes \dots \otimes \bx}_{k \ \text{times}}
    \in \real^{n^k}.
\end{equation*}

Two common operations that aid with Kronecker product algebra are the $\Vec(\cdot)$ and $\text{reshape}(\cdot)$ operations.
The $\Vec(\cdot)$ operation stacks the columns of a matrix into one tall column vector.
As defined in \cite[pg. 28]{Golub2013}, for $p_1 q_1 = p_2 q_2$, the $\text{reshape}(\cdot)$ operation reshapes a matrix $\bA \in \real^{p_1 \times q_1}$ into a new matrix
$\text{reshape}(\bA,p_2,q_2) \in \real^{p_2 \times q_2}$
such that
$\Vec(\bA) = \Vec(\text{reshape}(\bA,p_2,q_2))$.

The vector $\kronF{\bx}{k}$ exhibits a type of symmetry due to repeated entries in the Kronecker product,
which may be used to define a symmetrization operation on vectors.

\begin{definition}\label{def:sym}
    Let $\bv_k \in \real^{n^k}$ be a vector corresponding to a $k$-way tensor $\bV_k \in \real^{n \times n \times \dots \times n}$, such that $\Vec(\bV_k) = \bv_k$.
    We say that the vector $\bv_k$ is \emph{symmetric} if $\bV_k$ is a symmetric tensor, i.e., for any permutation $\pi$ of the indices $\{1,2,\dots,k\}$, the following holds:
    \begin{equation*}
        \bV_k(i_1,i_2,\dots,i_k) = \bV_k(i_{\pi(1)},i_{\pi(2)},\dots,i_{\pi(k)}), \quad \forall i_1,i_2,\dots,i_k \in \{1,2,\dots,n\},
    \end{equation*}
    where $\bV_k(i_1,i_2,\dots,i_k)$ denotes the entry of $\bV_k$ corresponding to the multi-index $(i_1,i_2,\dots,i_k)$.
\end{definition}
\Cref{def:sym} can be used to define a vector symmetrization operation.
For any given vector $\bv_k \in \real^{n^k}$, we define
$\rm{symmetrize}(\bv_k)$
to be the vector $\tilde{\bv}_k \in \real^{n^k}$ such that $\tilde{\bv}_k = \Vec(\tilde\bV_k)$, where $\tilde\bV_k$ is the symmetric $k$-way tensor obtained by reshaping $\bv_k$ as a $k$-way tensor and averaging over all permutations of its indices.

Finally, we define the generalized $k$-way Lyapunov matrix, which plays an important role in the Kronecker product formulation of Al'brekht's method.
For $\bA \in \real^{p \times q}$ and $\bE \in \real^{q \times q}$, the generalized
\textit{$k$-way Lyapunov matrix} is
\begin{equation*}
    \cL^\bE_k(\bA) \coloneqq \sum_{i=1}^k\underbrace{\bE \otimes \bA \otimes \bE \otimes \dots \otimes \bE}_{\text{$k$ factors, $\bA$ in the $i$th position}} \quad \in \real^{p^k \times p^{k-1}q}.
\end{equation*}

\section{Taylor series approximations to optimal control problems} \label{ch3:sec:PPR-Poly}
In this section, we describe the method by which a Taylor approximation of the value function and feedback law can be computed for the infinite horizon nonlinear optimal control problem.
In \cref{ch3:sec:theory} we review the relevant theory describing the way by which the solution to the optimal control problem, which can be obtained as the solution to the Hamilton-Jacobi-Bellman equation, can be approximated by a Taylor series expansion.
This approach is commonly known as Al'brekht's method.
In \cref{ch3:sec:full-Albrekht}, we describe our specific computational approach for implementing Al'brekht's method based on the Kronecker product, and we expand on our previous work \cite{Corbin2024e} by (i) addressing generalized dynamics with non-identity mass matrices and (ii) allowing for polynomial state dependence in the control penalty portion of the cost.
Subsequently, we describe the computational complexity and memory requirements of the full-order approach in \cref{ch3:sec:complexity}, which motivates the dimensionality reduction approach proposed in \cref{ch3:sec:reduced-Albrekht}.

\subsection{The theory of Al'brekht's method}\label{ch3:sec:theory}
We are concerned with the optimal regulation problem for nonlinear control-affine systems
\begin{align}
    \dot{\bx}(t)
    = \bf(\bx(t))  + \bg(\bx(t)) \bu(t), \qquad \bx(0)  = \bx_0, \label{ch3:eq:FOM-NL}
\end{align}
where $t$ is time,
$\bx(t) \in \real^n$ is the state,
$\bu(t) \in \real^m$ is the input,
$\bf:\real^n \to \real^n$ is the nonlinear drift, and
$\bg:\real^n \to \real^{n\times m}$ is the nonlinear input map.
We assume $\bf(\bzero) = \bzero$, and that there exists an equilibrium point at the origin $\bx=\bzero$ which we seek to stabilize.
The optimal control problem is to find the solution
\begin{align}
    \bu^*(\cdot) = {\arg \min}_{\bu(t) \in L_2(0,\infty)}
    \left\{
    J(\bx_0,\bu)  \coloneqq  \frac{1}{2}  \int_{0}^{\infty} \bx^\top \bm{Q}(\bx) \bx + \bu^\top \bm{R}(\bx) \bu \,\rd t
    \right\},
    \label{ch3:eq:opt-contr-problem}
\end{align}
where $\bJ(\cdot,\cdot)$ is the scalar performance index.
The functions $\bf(\bx)$, $\bg(\bx)$, $\bm{Q}(\bx) \succeq 0$, and $\bm{R}(\bx) \succ 0$ are all assumed to be analytic, and the variables $\bx$ in the integrand of the cost are the solution $\bx(t)$ subject to the initial condition $\bx(0) = \bx_0$ and the control input $\bu(t)$.

A foundational result in control theory is that the solution to the optimal control problem \cref{ch3:eq:opt-contr-problem} can be obtained by solving the Hamilton-Jacobi-Bellman partial differential equations (HJB PDEs) for the \textit{value function} $V(\bx) \in \real_{\geq 0}$,
where $V(\bx) = J(\bx_0,\bu^*)$, and the optimal state feedback law $\bu(\bx) \in \real^m$:
\begin{align}
    0      & = \frac{\partial V^\top(\bx)}{\partial \bx}  \bf(\bx) + \frac{\partial V^\top(\bx)}{\partial \bx}\bg(\bx) \bu(\bx) + \frac{1}{2} \bx^\top \bm{Q}(\bx) \bx  + \frac{1}{2} \bu(\bx)^\top \bm{R}(\bx) \bu(\bx), \label{ch3:eq:HJB-PDE} \\
    \bzero & =\bm{R}(\bx) \bu(\bx) +  \bg^\top(\bx)\frac{\partial V(\bx)}{\partial \bx} . \label{ch3:eq:HJB-PDE-u}
\end{align}
The control-affine dynamics, the structure of the cost function being quadratic with respect to the input, and the lack of control constraints are key to obtaining this result; for more details on the derivation of the HJB PDE, we refer to \citep{Kalman1963}.
While solving the HJB PDEs \labelcref{ch3:eq:HJB-PDE,ch3:eq:HJB-PDE-u} analytically is not feasible in general,
there exist special cases where additional assumptions can be leveraged to obtain solution approximations.
A famous example is the method of Al'brekht \cite{Albrekht1961}, which gives a formal procedure for approximating HJB PDE solutions via power series expansions under particular structural and local stabilizability and detectability assumptions.
\begin{theorem}[Al'brekht's method \citep{Albrekht1961,Lukes1969}]\label{ch3:thm:albrekht}
    Assume that the functions $\bf(\bx)$ and $\bg(\bx)$ in the dynamics \cref{ch3:eq:FOM-NL}, along with the functions $\bm{Q}(\bx)$ and $\bm{R}(\bx)$ in the cost function \cref{ch3:eq:opt-contr-problem}, are analytic about $\bx=\bzero$.
    Define the matrices $\bA = \partial \bf/\partial \bx(\bzero)$, $\bB = \bg(\bzero)$, $\bQ_2 = \bm{Q}(\bzero)$, and $\bR_0 = \bm{R}_0(\bzero)$.
    Further assume that the pair $(\bA, \bB)$ is stabilizable and the pair $(\bA, \bQ_2)$ is detectable.
    These conditions guarantee that the associated linear-quadratic regulator (LQR) problem admits a unique stabilizing solution and the corresponding closed-loop linearization is Hurwitz.
    Then the following statements hold for the nonlinear optimal control problem \eqref{ch3:eq:opt-contr-problem}:
    \begin{enumerate}
        \item the value function and the optimal feedback law are
              analytic about $\bx=\bzero$ and can be locally approximated by polynomials;
        \item the lowest degree polynomial term in the value function is degree~$2$, whose polynomial coefficient
              is given by the unique stabilizing solution to the algebraic Riccati equation (ARE) associated with the LQR problem on the linearized dynamics;
        \item the remaining higher degree polynomial coefficients of the value function are solutions to
              linear algebraic equations that are entirely determined by the already-computed polynomial coefficients.
    \end{enumerate}
\end{theorem}
Al'brekht's results showed existence and uniqueness of
power series expansions for the value function and the optimal feedback law in the form
\begin{align*}
    V(\bx)   & = {V}^{(2)}(\bx) + {V}^{(3)}(\bx) + \dots + {V}^{(d)}(\bx) + \dots         \\
    \bu(\bx) & = {\bu}^{(1)}(\bx) + {\bu}^{(2)}(\bx) + \dots + {\bu}^{(d-1)}(\bx) + \dots
\end{align*}
where $^{(1)}$ denotes a linear term, $^{(2)}$ denotes a quadratic term, and so on.
However, while he abstractly proved that
algebraic constraints for these polynomial expansion coefficients could be derived from the HJB PDEs,
his presentation was theoretical in nature.
He therefore did not provide explicit equations to solve for the coefficients, nor did he analyze numerical scalability or computational complexity.
Likewise, his work did not focus on scalable implementation details, such as exploiting sparsity or handling the generalized case of a non-identity mass matrix.
In this work, we adopt the Kronecker product representation to write polynomials explicitly and derive precise formulas for the value function and optimal feedback law.
Moreover, we generalize the method to work with an invertible mass matrix $\bE$ in pursuit of efficient computations that leverage sparsity.

\subsection{The solution to Al'brekht's method explicitly in terms of the Kronecker product}\label{ch3:sec:full-Albrekht}
In the interest of deriving explicit formulas for the coefficients of the value function and feedback law expansions,
consider
a nonlinear control-affine dynamical system with polynomial structure
in generalized form:
\begin{align}\label{ch3:eq:FOM-Poly}
    \bE\dot{\bx} & = \underbrace{\bA \bx + \sum_{p=2}^\infty \bF_p \kronF{\bx}{p}}_{\bf(\bx)} + \underbrace{\left(\sum_{p=1}^\infty \bG_p \left(\kronF{\bx}{p} \otimes \bI_m\right) + \bB\right)}_{\bg(\bx)} \bu,
\end{align}
where $\bA \in \real^{n\times n}$, $\bF_p \in \real^{n\times n^p}$, $\bB \in \real^{n\times m}$, and $\bG_p \in \real^{n \times mn^p}$.
Any model for which $\bf(\bx)$ and $\bg(\bx)$ are analytic can be brought into this form locally via Taylor expansion, and many models already feature polynomial nonlinearities.
The mass matrix $\bE \in \real^{n\times n}$ is assumed to be nonsingular, so in theory it is possible to invert $\bE$ to obtain a standard control-affine model, but for computational purposes it is advantageous to consider the generalized form for high dimensional models.
The cost function is written similarly
in terms of the Kronecker product as
\begin{align}
    J(\bx_0,\bu)  \coloneqq  \frac{1}{2}  \int_{0}^{\infty}
    \bx^\top \bQ_2 \bx +   \bu^\top \bR_0 \bu + \sum_{p=3}^{\infty} \bq_p^\top \kronF{\bx}{p}
    +  \sum_{p=1}^{\infty}
        {\kronF{\bx}{p}}^\top \br_p^\top \kronF{\bu}{2}
    \,\rd t, \label{ch3:eq:performanceIndex-Poly}
\end{align}
where $\bq_p \in \real^{n^p}$ and $\br_p \in \real^{m^2 \times n^p}$, and
we assume $\bQ_2 \succeq \bzero$ and $\bR_0 \succ \bzero$.

Following the result of \cref{ch3:thm:albrekht}, the value function and optimal feedback law are analytic, so they both yield locally convergent Taylor expansions.
The feedback law can be expanded as
\begin{align}
    \bu(\bx) & = \bK_1 \bx + \bK_2 \kronF{\bx}{2} + \bK_3 \kronF{\bx}{3} + \dots
    = \sum_{i=1}^{\infty} \bK_i \kronF{\bx}{i} \label{ch3:eq:Ki-coeffs}
\end{align}
with the coefficients $\bK_i \in \real^{m \times n^i}$ for $i=1,2,3,\dots$.
In order to perform computations without explicitly inverting the mass matrix $\bE$, it is preferable to compute the value function in a transformed coordinate $\bz \coloneqq \bE \bx$, hence we expand the value function as
\begin{align}
    \overline V(\bz) & = \frac{1}{2} \bz^\top \overline\bV_2 \bz +
    \frac{1}{2}\bar\bv_3^\top \kronF{\bz}{3} +
    \frac{1}{2}\bar\bv_4^\top \kronF{\bz}{4} + \dots
    =  \frac{1}{2} \sum_{i=2}^\infty \bar\bv_i^\top \kronF{\bz}{i}, \label{ch3:eq:vi-coeffs}
\end{align}
with the coefficients $\bar\bv_i \in \real^{n^i}$ for $i=2,3,4, \dots$.
For convenience, we define $\bar\bv_2 \coloneqq \Vec[\overline\bV_2] \in \real^{n^2}$.
Without loss of generality, the value function coefficients may be assumed to be symmetric; this allows us to write the gradient of the value function as
\begin{align}
    \frac{\partial^\top \overline V(\bz)}{\partial \bz}
     & = \frac{1}{2}\left (2\bar\bv_2^\top (\bI_n \otimes \bz)  + 3\bar\bv_3^\top (\bI_n \otimes \bz \otimes \bz)
    + \cdots \right )
    = \frac{1}{2}\sum_{i=2}^{\infty} i \bar\bv_i^\top \left(\bI_n \otimes \kronF{\bz}{i-1}\right).\label{ch3:eq:Past-value-deriv-compact}
\end{align}

\begin{remark}\label{ch3:rem:vbar-v-relation}
    The value function in the original coordinates is related to the transformed value function by the relation $V(\bx) \coloneqq \overline V(\bz)\Big|_{\bz=\bE\bx}$.
    For the purposes of computing the optimal feedback law, it is possible to use the transformed value function $\overline V(\bz)$ without explicitly forming the value function $V(\bx)$ in the original coordinates.
\end{remark}

In practice, we compute a truncated degree~$d$ Taylor approximation to the value function, which corresponds to a degree~$d-1$ approximation to the feedback law.
The next theorem provides our first main result, which is the explicit equations for
the feedback law coefficients $\bK_i$ for $i=1, 2, \dots, d-1$ in Kronecker product form.

\begin{theorem}[Feedback law polynomial coefficients]\label{ch3:thm:kiPoly}
    Consider the special case of the optimal control problem \cref{ch3:eq:FOM-Poly,ch3:eq:opt-contr-problem} given by the polynomial performance index \cref{ch3:eq:performanceIndex-Poly} and polynomial dynamics \cref{ch3:eq:FOM-Poly}.
    Assume that
    $\bQ_2 \succeq \bzero$, $\bR_0 \succ \bzero$,
    $\bE$ is nonsingular,
    and assume that the the pair $(\bE^{-1}\bA,\bE^{-1}\bB)$ is stabilizable and $(\bE^{-1}\bA,\bQ)$ is detectable.
    Let the optimal feedback law $\bu(\bx)$ from~\cref{ch3:eq:Ki-coeffs} be expanded and truncated to $d-1$ terms, with the coefficients $\bK_i \in \real^{m \times n^i}$ for $i=1,2,\dots,d-1$.
    Then $\overline\bV_2$ and $\bK_1$ are the unique stabilizing solution to the generalized algebraic Riccati equation (GARE) and the linear gain matrix, respectively, associated with the LQR problem:
    \begin{align}
         & \bA^\top  \overline\bV_2 \bE +   \bE^\top  \overline\bV_2 \bA- \bE^\top  \overline\bV_2 \bB \bR_0^{-1} \bB^\top \overline\bV_2 \bE
        + \bQ_2 = \bzero, \label{ch3:eq:ARE}                                                                                                  \\
         & \bK_1      = - \bR_0^{-1} \bB^\top \overline\bV_2 \bE.\label{ch3:eq:K1}
    \end{align}
    For $3 \leq k \leq d$, the linear systems
    \begin{equation}
        \begin{split}
            \cL^\bE_{k}(\bA + \bB\bK_1 )^\top \bar{\bv}_k & =
            -\rm{symmetrize}\left(\bb\right)  ,\label{ch3:eq:value-function-linear-system} \\
            \bb                                           & \coloneqq
            \sum_{\mathclap{i+p=k+1}} \cL^\bE_{i} (\bF_p)^\top \bv_i
            + \sum_{\mathclap{i+p+q=k+1}} i
            \Vec\left[\bK_q^\top\cG\cV^\bE_{p,i}\right]
            +  \bq_k
            + \sum_{\mathclap{i+p+q=k}} \Vec\left[\br_i^\top \left(\bK_p \otimes \bK_q \right)\right] ,
        \end{split}
    \end{equation}
    where $\cG\cV^\bE_{p,i} = {\text{reshape}\left(\bG_p^\top\overline\bV_i \kronF{\bE}{i-1},m,n^{k-q}\right)}$,
    are uniquely solvable, and their solutions $\bar{\bv}_k \in \real^{n^k}$ are used to determine the gain coefficients
    $\bK_{k-1}$ for $3 \leq k \leq d$ as
    \begin{align}
        \bK_{k-1} & =  \bR_0^{-1}  \left[- \sum_{\mathclap{\substack{p+q=k-1 \\q\neq0}}}  \bR_q^\top \left(\bK_p \otimes \bI_{n^q}\right)   -\sum_{\mathclap{i+p=k}} \frac{1}{2} i
            \cG\cV^\bE_{p,i}
            \right]. \label{ch3:eq:control-gain-linear-system}
    \end{align}
\end{theorem}
The proof of \cref{ch3:thm:kiPoly} is somewhat lengthy, and the interested reader can find it in \labelcref{ch3:sec:thm-proof}.
Furthermore, the existence and uniqueness of a stabilizing solution to the GARE \cref{ch3:eq:ARE} are guaranteed by the stabilizability and detectability assumptions on the linearization of the dynamics at the origin, as discussed in \labelcref{sec:are-solution}.
Similarly, the invertibility of $\bA + \bB \bK_1$ (guaranteed since it is Hurwitz) translates to invertibility of the larger $k$-way Lyapunov matrix $\cL^\bE_{k}(\bA + \bB\bK_1 )$ to guarantee uniqueness and existence of solutions to \eqref{ch3:eq:value-function-linear-system} \cite{Grasedyck2004}, a property which traces back even to Al'brekht's original work \cite{Albrekht1961}.
\\

For the purposes of computing the optimal feedback coefficients $\bK_{k-1}$, it is sufficient to deal with the transformed value function coefficients $\bar \bv_k$; the following corollary describes how to obtain the value function coefficients in the original coordinates if they are are needed explicitly.
\begin{corollary}[Value function polynomial coefficients]\label{ch3:cor:viPoly}
    The value function polynomial coefficients in the original coordinates $\bv_k$ are given in terms of the value function coefficients in the transformed coordinates $\bar \bv_k$ as
    \begin{align*}
        \bv_k = {\kronF{\bE}{k}}^\top \bar \bv_k.
    \end{align*}
\end{corollary}
\begin{proof}
    As mentioned in \cref{ch3:rem:vbar-v-relation},
    $V(\bx) \coloneqq \overline V(\bz)$ at $\bz=\bE\bx$, so
    for each coefficient we have $\bv_k^\top \kronF{\bx}{k} = \bar\bv_k^\top \kronF{\bz}{k}$.
    Inserting $\bz=\bE\bx$ and rearranging proves the result.
\end{proof}

The solution procedure described by \cref{ch3:thm:kiPoly} is summarized in \cref{ch3:alg:full-PPR}.
\begin{algorithm}[!ht]
    \caption{Taylor approximation to the optimal feedback law $\bu(\bx)$}\label{ch3:alg:full-PPR}
    \begin{algorithmic}[1]
        \Require System matrices $\bE$, $\bA$, $\bB$, $\{\bF_p\}_{p=2}^{d-1}$, $\{\bG_p\}_{p=1}^{d-2}$ from \cref{ch3:eq:FOM-Poly}; cost function weights $\{\bq_p\}_{p=2}^d$, $\{\br_p\}_{p=0}^{d-2}$ from \cref{ch3:eq:performanceIndex-Poly}; desired value function approximation degree $d$
        \Ensure Polynomial gain coefficients $\{\bK_i\}_{i=1}^{d-1}$
        \State Solve the ARE \cref{ch3:eq:ARE} for $\overline\bV_2$.
        \State Form the linear gain coefficient $\bK_1$ according to \cref{ch3:eq:K1}.
        \For{$k=3$ to $d$}
        \State Solve \cref{ch3:eq:value-function-linear-system} for $\bar\bv_k$.
        \State Form the next nonlinear gain coefficient $\bK_{k-1}$ according to \cref{ch3:eq:control-gain-linear-system}.
        \EndFor
    \end{algorithmic}
\end{algorithm}

\subsection{Generalized Bartels-Stewart algorithm to efficiently solve structured linear systems}\label{ch3:sec:linear-solver}
The linear systems \labelcref{ch3:eq:value-function-linear-system} represent the main computational hurdle for Al'brekht-based approaches.
These linear systems feature special Lyapunov-type structure; for $k=2$ they are equivalent to matrix Lyapunov equations, whereas for arbitrary $k$ they can be considered tensor Lyapunov equations.
Solving these linear systems for $\bar\bv_k$ using standard methods such as Gaussian Elimination is prohibitively expensive, with a cost of $O(n^{3k})$ due to the vector $\bar\bv_k$ having $n^k$ unknowns.
Instead, we leverage structured solvers that take advantage of the $k$-way Lyapunov matrix structure of the coefficient matrix $\cL_k^\bE(\bA+\bB\bK_1)^\top$.
For the case when $\bE = \bI$, existing solvers can be used \cite{Chen2019,Borggaard2021}.
Both of these solvers are Bartels-Stewart type algorithms: the quantity $(\bA+\bB\bK_1)^\top$ is first triangularized using a Schur decomposition: $(\bA+\bB\bK_1)^\top = \bQ \bT\bQ^*$.
The unitary transformation $\bQ$ can be used to then triangularize the $k$-way Lyapunov matrix $\cL_k(\bA+\bB\bK_1)^\top$, leading to a triangular system with coefficient matrix $\cL_k(\bT)$.
The method in \citep{Borggaard2021} solves the triangular system via block back substitution by solving $n^{k-1}$ smaller triangular systems of dimension $n$, leading to an overall cost of $O(n^{k+1})$, whereas \citep{Chen2019} uses a recursive blocked approach
that is more efficient but overall still scales as $O(n^{k+1})$.

We have developed an additional \textit{generalized} Bartels-Stewart type solver for the case where $\bE$ is nonsingular but not necessarily identity, as is often the case with models arising from semi-discretization techniques such as the finite element method.
In this case, the pair
\begin{equation*}
    \left({(\bA + \bB\bK_1)}^\top, \bE^\top\right)
\end{equation*}
is decomposed using the QZ decomposition \citep{Moler1973}, also known as the generalized Schur decomposition.
This again leads to a triangular linear system that can be solved via back substitution with $O(n^{k+1})$ operations.
However, a more complicated back substitution procedure is necessary, similar to the approach described in \citep{Penzl1998} for generalized Lyapunov matrix equations.

\subsection{Computational complexity and memory requirements}\label{ch3:sec:complexity}
In this section, we analyze the computational complexity and associated memory requirements of the
solution procedure described in \cref{ch3:thm:kiPoly} and \cref{ch3:alg:full-PPR}.
\begin{proposition}
    The computational complexity of \cref{ch3:alg:full-PPR} is $O(n^{d+1})$, and the memory complexity is $O(n^d)$.
\end{proposition}
\begin{proof}
    The complexity of \cref{ch3:alg:full-PPR} can be assessed using a floating point operation (flop) count.
    Standard solvers for the ARE in line 1 require $O(n^3)$ operations (low-rank solvers are discussed in \cref{ch3:sec:lr-adaptation}).
    Once $\overline\bV_2$ is known, $\bK_1$ can be formed with $O(m^3n)$ operations by solving $n$ linear systems with respect to inverting $\bR_0$.
    Often $\bR_0^{-1}$ is known explicitly, in which case $\bK_1$ can be formed with $O(m^2n)$ operations.
    The operations of forming and appropriately symmetrizing the quantities in lines 4-7 can be performed efficiently by reshaping Kronecker products as matrix multiplication operations.
    The main workhorse is the Kronecker-Vec identity $\left(\bB^\top \otimes \bA \right) \Vec[\bD] = \Vec[\bA\bD\bB]$, as described in \cite[Sec. 1.3.7]{Golub2013}.
    In the case of repeated Kronecker products, the Kronecker-Vec relation can be applied recursively.
    This allows the linear systems for $\bar\bv_k$ and $\bK_{k-1}$ to be formed with $O(n^{k+1})$ complexity.

    Solving the linear system \cref{ch3:eq:value-function-linear-system} for $\bar\bv_k$ in line 6 using standard methods prohibitively expensive, with a cost of $O(n^{3k})$ due to the vector $\bar\bv_k$ having $n^k$ unknowns.
    Instead, as described in \cref{ch3:sec:linear-solver}, we leverage structured Lyapunov solvers that scale as $O(n^{k+1})$.
    Once $\bar\bv_k$ is known, $\bK_{k-1}$ can be formed with $O(m^3n^{k-1})$ operations by solving the $n^{k-1}$ columns as linear systems with respect to inverting $\bR_0$.
    Similar to the case for $\bK_{1}$, which corresponds to $k=2$, when $\bR_0^{-1}$ is known explicitly, $\bK_{k-1}$ can be formed with $O(m^2n^{k-1})$ operations.

    Overall, the complexity of \cref{ch3:alg:full-PPR} for computing a degree $d$ value function approximation and degree $d-1$ approximation of the optimal feedback law is $O(n^{d+1})$.
    The memory requirements for this algorithm are $O(n^d)$, primarily due to the need to store the right-hands-side vector $\bb \in \real^{n^d}$ in \cref{ch3:eq:value-function-linear-system} and the resulting value function coefficient $\bar\bv_k \in \real^{n^d}$
    (the structured $k$-way Lyapunov solvers never form the matrix $\cL_k^E(\bA + \bB\bK_1)^\top$).
    The highest-order gain coefficient $\bK_{d-1}$ requires $O(mn^{d-1})$ storage, so the memory usage is dominated by the value function computations.
\end{proof}

\begin{remark}
    Throughout our complexity analysis, we assume that the polynomial degree $d$ is a small fixed constant (typically $d \leq 6$, with the highest being $d=12$ in one of our examples).
    The implied constant factor in the big-O notation can depend linearly on $d$, but the dominant contributor to the increase in computational complexity with respect to polynomial degree is the exponential factor $n^{d+1}$.
\end{remark}

\section{Improving scalability of Al'brekht's method via higher-order coefficient reduction}\label{ch3:sec:reduced-Albrekht}
The approach in \cref{ch3:sec:PPR-Poly} can be used to compute the Taylor coefficients of the value function and feedback law for models of moderate dimensions, e.g. a degree~4 value function and corresponding degree~3 feedback law can be computed for models of dimension $n=O(100)$, and a degree~3 value function and corresponding degree~2 feedback law can be computed for models of dimension $n=O(1000)$ on a computer with 16 GB of RAM.
For accommodating models with larger state dimensions, we propose a novel dimensionality reduction approach that can reduce the memory requirements of Al'brekht's method.
In \cref{ch3:sec:reduction-projection} we describe the proposed approach for dimensionality reduction to circumvent the memory limitations of the full-order approach.
The modified version of \cref{ch3:alg:full-PPR} featuring the proposed reduction step is then summarized in \cref{ch3:alg:reduced-PPR}, with the computational and memory complexity analysis revisited in \cref{ch3:sec:reduced-complexity}.
We then present a specialized version of the algorithm for the case of sparse, low-rank problems that often arise due to the semi-discretization of PDEs in \cref{ch3:sec:lr-adaptation}, in which \cref{ch3:alg:reduced-PPR} is modified to use low-rank factors $\bV_2 \approx \bZ \bZ^\top$ instead of the full GARE solution, as described in \cref{ch3:alg:lr-PPR}.

\subsection{Dimensionality reduction of the higher-order coefficients by projection}\label{ch3:sec:reduction-projection}
The method proposed in this section involves computing value function and feedback law approximations of the form
\begin{align}
    V(\bx)   & \approx \frac{1}{2} \left(\bx^\top \bV_2 \bx + {\bv^r_3}^\top \kronF{\bx_r}{3} + {\bv^r_4}^\top \kronF{\bx_r}{4}   + \dots\right), \label{ch3:eq:vi-coeffs-reduced} \\
    \bu(\bx) & \approx  \bK_1 \bx + \bK^r_2 \kronF{\bx_r}{2} + \bK^r_3 \kronF{\bx_r}{3}   + \dots, \label{ch3:eq:Ki-coeffs-reduced}
\end{align}
where $\bx_r \in \real^r$ is a (projected) reduced-order approximation of the state.
Compared with \cref{ch3:eq:Ki-coeffs,ch3:eq:vi-coeffs}, in \cref{ch3:eq:vi-coeffs-reduced,ch3:eq:Ki-coeffs-reduced}, the higher-order terms with coefficients $\bv^r_k \in \real^{r^k}$ and $\bK^r_k \in \real^{m \times r^k}$ are instead written in terms of the reduced state $\bx_r$.
Meanwhile, the first terms in the value function and feedback law approximations
are still written in terms of the full state $\bx$.
Their coefficients , $\bV_2 \in \real^{n \times n}$ and $\bK_1 \in \real^{m \times n}$,
are computed by solving the LQR problem on the full-order linearized dynamics, precisely as in
\cref{ch3:eq:ARE,ch3:eq:K1}.
This is a key distinction with other existing papers that propose to use model order reduction for computing HJB solution approximations \cite{Kunisch2004,Alla2017,Breiten2018}, as it permits using the full-order solutions for the first terms in the expansions, which dominate and contribute significantly to local optimality.
Critically, it is the higher-order coefficients whose storage requirements of $O(n^k)$ are obstructive, so the approximations \cref{ch3:eq:vi-coeffs-reduced,ch3:eq:Ki-coeffs-reduced} retain the full-order accuracy in the linear terms while also reducing the storage requirements to $O(n^2 + r^k)$.

The reduced state $\bx_r$ can be computed with any model order reduction method.
Of particular interest are linear approximations of the form  $\bx \approx \bT \bx_r$ with reduction basis $\bT \in \real^{n \times r}$.
We denote by $\bT^\dagger$ the left inverse, $\bT^\dagger \bT = \bI_r$.
Common ways to construct $\bT$ include POD and balanced truncation.
In this work, we propose to use the dominant eigenvectors of the quadratic value function coefficient $\bV_2$ from \cref{ch3:eq:ARE} to construct $\bT$, as it adds negligible additional computational cost due to already being computed during the first step of the algorithm.
Additionally, $\bV_2$ is symmetric positive semidefinite, and its eigenvectors are orthogonal, so the left inverse is known analytically $\bT^\dagger~=~\bT^\top$.

This choice of reduction basis has an intuitive justification.
For a positive semidefinite matrix $\bM$, the quadratic form $m(\bx) = \bx^\top \bM \bx$ defines a paraboloid whose the level sets are ellipses, as depicted in \cref{ch3:fig:paraboloid} for the two dimensional case.
The geometry of these ellipses is determined by the eigenvalue decomposition of the weighting matrix $\bM$: the directions of the semi-axes of the ellipse are given by the eigenvectors, whereas the inverse of the eigenvalues give the lengths of the semi-axes.

\begin{figure}[htb]
    \centering
    \includegraphics[page=1,width=0.25\textwidth]{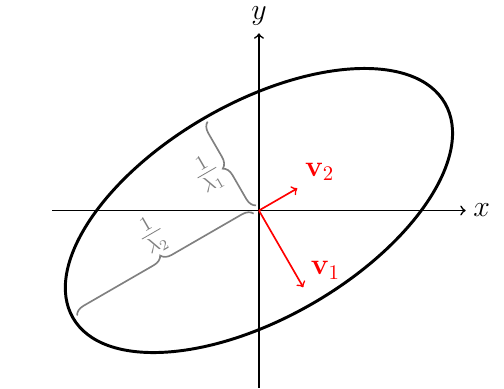}\\
    \includegraphics[page=2,width=0.25\textwidth]{paraboloid.pdf}
    \caption{Depiction of an elliptic paraboloid defined by a quadratic form $m(\bx) = \bx^\top \bM \bx$. Level sets of $q(\bx$) are ellipses; the eigenvectors of $\bM$ determine the directions of the ellipse's semi-axes, and the inverse of the eigenvalues of $\bM$ determine the length of the semi-axes.}
    \label{ch3:fig:paraboloid}
\end{figure}

In the context of optimal control, $V(x) = \frac{1}{2}\bx^\top \bV_2 \bx$ represents a quadratic local approximation to the value function, its gradient being $\partial V(x)/\partial x = \bV_2 \bx$.
Importantly, the optimal control depends on the \textit{gradient} of the value function.
According to the eigenvalue decomposition then,
eigenvectors corresponding to small eigenvalues are directions in which the gradient of the value function is small, and therefore the optimal control signal is small, i.e. little control effort is needed\footnote{This line of reasoning is of course only valid in a neighborhood of the equilibrium point.}.
Conversely, the leading eigenvectors of $\bV_2$ indicate directions in which significant control action is needed.
The hypothesis of this approach is therefore that these are the directions which will benefit the most from the inclusion of additional degrees of approximation.
This reasoning is of course only valid locally, and if sufficiently large reduced-order dimension $r$ is selected.
In practice, the controller performance needs to be evaluated to verify if the desired performance is acquired with the computational resources available.
The LQR controller provides a helpful baseline, and it should always be remembered that if the trajectory or initial condition departs too far from the regulation point, the convergence of the Taylor approximation may be lost.

Finally, we again note that even if a different reduction basis $\bT$ is chosen, our proposed approach permits retaining the full-order accuracy in the linear term and only applying the reduction to the computation of higher-order terms.
It is worth mentioning that it is also possible to apply different amounts of reduction to each degree, or even a nonlinear reduced-order approximation, although these ideas have not been investigated.

The reduced-order coefficients $\bv_k^r$ and $\bK_k^r$ can be computed by
applying Al'brekht's method to the HJB PDEs associated with the dynamics of the following projected reduced-order operators\footnote{The transformation $\bT$ does not in general preserve the sparsity of the full-order model, so we opt to put the reduced-order model in standard form in order to use more efficient standard $k$-way Lyapunov solvers.}:
\begin{equation}\label{ch3:eq:ROM-operators}
    \begin{aligned}
        \bE_r     & = \bT^\dagger\bE\bT \in \real^{r \times r},
        \\
        \bA_r     & = \bE_r^{-1}\bT^\dagger\bA\bT \in \real^{r \times r},                                             &
        \bF_{p,r} & = \bE_r^{-1}\bT^\dagger\bF_{p}\kronF{\bT}{p} \in \real^{r \times r^p},
        \\
        \bB_r     & = \bE_r^{-1}\bT^\dagger\bB \in \real^{r \times m},                                                &
        \bG_{p,r} & = \bE_r^{-1}\bT^\dagger\bG_p \left(\kronF{\bT}{p}\otimes \bI_m\right) \in \real^{r \times m r^p},   \\
        \bq_{p,r} & = {\kronF{\bT}{p}}^\top \bq_p \in \real^{r^p},                                                    &
        \br_{p,r} & = \br_p  \kronF{\bT}{p} \in \real^{m^2 \times r^p}.
    \end{aligned}
\end{equation}

The next theorem corresponds to \cref{ch3:thm:kiPoly} for the reduced-order dynamics computed using the eigenvectors of $\bV_2$ as a reduction basis.
\begin{theorem}[Reduced feedback law polynomial coefficients]\label{ch3:thm:kiPoly-reduced}
    Consider the optimal control problem \cref{ch3:eq:opt-contr-problem} with dynamics \labelcref{ch3:eq:FOM-NL} where the functions the functions $\bf(\bx)$, $\bg(\bx)$, $\bm{Q}(\bx) \succeq 0$, and $\bm{R}(\bx) \succ 0$ are analytic and $\bE$ is invertible.
    Let the optimal feedback law $\bu(\bx)$ be expanded as in~\cref{ch3:eq:Ki-coeffs-reduced} with the coefficients $\bK_1 \in \real^{m \times n}$ and $\bK^r_i \in \real^{m \times r^i}$ for $i=2,\dots,d-1$.
    Then $\bK_1$ is the linear gain associated with the full-order LQR problem, and $\overline\bV_2 \in \real^{n \times n}$ is the corresponding ARE solution.

    Now, let $\bT \in \real^{n\times r}$ contain $r$ dominant eigenvectors of $\bV_2 = \bE^\top \overline\bV_2 \bE$, and construct the reduced-order model coefficients as in
    \cref{ch3:eq:ROM-operators}.
    Furthermore, define $\bV_2^r = \bT^{\top} \bV_2 \bT$ and $\bK_1^r = \bK_1 \bT$ for use in the computations of the higher-order gain coefficients.
    For $3 \leq k \leq d$,
    the linear systems
    \begin{equation}
        \begin{split}
            \cL_{k}(\bA_r + \bB_r\bK^r_1 )^\top  {\bv}^r_k & = -
            \rm{symmetrize}\left(\bb\right)  ,\label{ch3:eq:reduced-value-function-linear-system} \\
            \bb                                            & \coloneqq
            \sum_{\mathclap{i+p=k+1}} \cL_{i} (\bF^r_p)^\top \bv^r_i
            + \sum_{\mathclap{i+p+q=k+1}} i
            \Vec\left[{\bK^r_q}^\top\cG\cV^r_{p,i}\right]
            +  \bq^r_k
            + \sum_{\mathclap{i+p+q=k}} \Vec\left[{\br^r_i}^\top \left(\bK^r_p \otimes \bK^r_q \right)\right] ,
        \end{split}
    \end{equation}
    where $\cG\cV^r_{p,i} = {\text{reshape}\left({\bG^r_p}^\top \bV^r_i, m,r^{k-q}\right)}$, are uniquely solvable, and their solutions ${\bv}^r_k \in \real^{r^k}$ are used to determine
    the  gain coefficients $\bK^r_{k-1}$ for $3 \leq k \leq d$ as
    \begin{align}
        \bK^r_{k-1} & =  \bR_0^{-1}  \left[- \sum_{\mathclap{\substack{p+q=k-1 \\q\neq0}}}  {\bR^r_q}^\top \left(\bK^r_p \otimes \bI_{r^q}\right)   -\sum_{\mathclap{i+p=k}} \frac{1}{2} i
            \cG\cV_{p,i}
            \right]. \label{ch3:eq:reduced-control-gain-linear-system}
    \end{align}
\end{theorem}

\begin{proof}
    The proof follows the proof for \cref{ch3:thm:kiPoly}, but replacing the dynamics for $\bx$ with those in the reduced coordinates $\bx_r$.
\end{proof}

\begin{corollary}[Value function and feedback law coefficient approximations]\label{ch3:cor:vi-Ki-Poly_approxs}
    The value function polynomial coefficients in the full-order coordinates, $\bv_k$, are approximated in terms of the value function coefficients in the reduced coordinates, $\bv_k^r$, as
    \begin{align*}
        \bv_k \approx {\kronF{\bT}{k}}  \bv_k^r.
    \end{align*}
    The feedback law polynomial coefficients in the full-order coordinates $\bK_k$ are approximated in terms of the feedback law coefficients in the reduced coordinates $\bK_k^r$ as
    \begin{align*}
        \bK_k \approx   \bK_k^r {\kronF{{\bT^\top}}{k}}.
    \end{align*}
\end{corollary}
\begin{proof}
    The proof follows similarly to \cref{ch3:cor:viPoly}.
    According to our approximations,
    $\bv_k^\top \kronF{\bx}{k} \approx {\bv_k^r}^\top \kronF{\bx_r}{k}$
    and
    $\bK_k \kronF{\bx}{k} \approx \bK_k^r \kronF{\bx_r}{k}$.
    Inserting
    $\bx_r \approx \bT^\top \bx$
    and rearranging proves the results.
\end{proof}

\begin{remark}
    In practice, \cref{ch3:cor:vi-Ki-Poly_approxs} should not be invoked to obtain an approximation to the value function $V(\bx)$ and feedback law $\bu(\bx)$, as it involves scaling the computed reduced-order coefficients back up to the full dimension.
    This is expensive both in terms of flops and memory.
    Rather, the reduced-order state approximation $\bx_r = \bT^\top \bx$ should be used in the evaluation of the higher-order terms.
    For example, to compute a cubic approximation to $\bu(\bx)$ for a given state $\bx$, instead of evaluating $\bu(\bx) \approx \bK_1 \bx + \bK_2^r {\kronF{{\bT^\top}}{2}} \kronF{\bx}{2} + \bK_2^r {\kronF{{\bT^\top}}{3}} \kronF{\bx}{3}$, one should
    first form $\bx_r = \bT^\top \bx$ and then evaluate $\bu(\bx) \approx \bK_1 \bx + \bK_2^r \kronF{\bx_r}{2} + \bK_2^r  \kronF{\bx_r}{3}$.
\end{remark}
\begin{remark}
    As described in the documentation for our \textsc{Matlab} function \texttt{ppr()}, a polynomial feedback law can be constructed and evaluated efficiently using our package and two lines of code:
    \begin{align*}
         & \texttt{[v, K] = ppr(f, g, q, r, degree, options);} \\
         & \texttt{uPPR = @(x) kronPolyEval(K, x);}
    \end{align*}
    where $\texttt{f}$, $\texttt{g}$, $\texttt{q}$, $\texttt{r}$, and $\texttt{degree}$ are the inputs described in \cref{ch3:alg:full-PPR}, namely the system dynamics, cost function parameters, and desired value function approximation degree.
    If the dimensionality reduction acceleration is desired, the reduced-order dimension can be passed as an option using $\texttt{options.reducedDimension}$; similarly, a non-identity mass matrix can be passed using $\texttt{options.E}$.
    For more information on the usage of the \textsc{Matlab} package, we refer to the \textsc{readme} file supplied in \cite{Corbin2024}.
\end{remark}

\subsection{Computational complexity and memory requirements}\label{ch3:sec:reduced-complexity}
The efficient procedure described in \cref{ch3:thm:kiPoly-reduced} is summarized in \cref{ch3:alg:reduced-PPR}.
The main difference between \cref{ch3:alg:reduced-PPR} and \cref{ch3:alg:full-PPR} occurs in lines 3-5, where $\bV_2$ is formed and its dominant eigenvectors are used to reduce the dimensionality of the model.
\begin{algorithm}[!ht]
    \caption{Reduced approximation to the optimal feedback law $\bu(\bx)$}\label{ch3:alg:reduced-PPR}
    \begin{algorithmic}[1]
        \Require System matrices $\bE$, $\bA$, $\bB$, $\{\bF_p\}_{p=2}^{d-1}$, $\{\bG_p\}_{p=1}^{d-2}$; cost function weights $\{\bq_p\}_{p=2}^d$, $\{\br_p\}_{p=0}^{d-2}$; desired approximation degree~$d$; desired reduced-order $r$
        \Ensure Approximation of the polynomial gain coefficients $\{\bK_i\}_{i=1}^{d-1}$
        \State Solve the ARE \cref{ch3:eq:ARE} for $\overline\bV_2$.
        \State Form the linear gain coefficient $\bK_1$ according to \cref{ch3:eq:K1}.
        \State Form $\bV_2 = \bE^\top \overline\bV_2 \bE$.
        \State Compute and define $\bT$ to be the $r$ dominant eigenvectors of $\bV_2$.
        \State Form the coefficients of the reduced-order dynamics and cost by projection, as in \cref{ch3:eq:ROM-operators}.
        \For{$k=3$ to $d$}
        \State Solve \cref{ch3:eq:reduced-value-function-linear-system} for $\bv_k^r$.
        \State Form the next reduced gain coefficient $\bK_{k-1}^r$ according to \cref{ch3:eq:reduced-control-gain-linear-system}.
        \EndFor
    \end{algorithmic}
\end{algorithm}
Computing the quantities $\overline\bV_2$ and $\bK_1$ in lines 1 and 2 associated with the full-order LQR problem require $O(n^3)$ operations, identical to \cref{ch3:alg:full-PPR}.
Forming $\bV_2$ and subsequently computing its eigenvectors in lines 3 and 4 also generally require $O(n^3)$ operations.
Forming the reduced-order dynamics and cost via projection with $\bT$ in line 5 has a computational complexity of $O(rn^{d})$, where $d$ is the desired degree of value function approximation.
The remainder of the operations in lines 6 through 11 are identical to those in \cref{ch3:alg:full-PPR}, except they are performed with dynamics of reduced dimension $r$, so the remaining operations have a complexity of $O(r^{d+1})$.
The overall computational complexity is thus dominated by the $O(rn^d)$ step during which the full-order model is projected down to a reduced dimension.
We remark that in practice, exploiting sparsity and the generalized form with the mass matrix $\bE$, the projection step can be performed much more efficiently,
and the complexity scales more as $O(r^{d+1})$.

Due to the requirement to store a model in memory before any operations can take place, the overall memory footprint of \cref{ch3:alg:reduced-PPR}
is $O(n^d)$ in the case of dense models.
Again, however, the model operators for large models such as finite element models are rarely dense and typically exhibit significant sparsity.
Thus, assuming the model operators that are inputs to the algorithm are sufficiently sparse to be stored in memory and that their contribution can be neglected, the remaining memory requirements of \cref{ch3:alg:reduced-PPR} once the reduced-order model has been computed diminish to $O(r^d)$.

\subsection{Low-rank adaptation}\label{ch3:sec:lr-adaptation}
Under certain conditions, such as when the quadratic state penalty $\bQ_2$ in \cref{ch3:eq:performanceIndex-Poly,ch3:eq:ARE} exhibits low-rank structure and when the number of inputs $m$ is much smaller than the number of states $n$, it has been shown that the Riccati solution exhibits low-rank structure and can be approximated as $\bV_2 \approx \bZ \bZ^\top$.
This has spurred the development of many advanced Riccati solvers that leverage sparsity and low-rank structure to enable scalability to much higher dimensions than traditional Schur-type and invariant subspace Riccati solvers \cite{Penzl1999,Penzl2000,Opmeer2010,Baker2015,Benner2013}.
These methods work especially well when the model operators are sparse, as is often the case for semi-discretized PDEs.
For an introduction to the topic of low-rank Riccati solvers, we recommend the survey \cite{Benner2013}.

For the nonlinear control problem considered herein,
\cref{ch3:alg:full-PPR,ch3:alg:reduced-PPR} can be adapted to leverage these modern low-rank Riccati solvers.
We have implemented the option in our software package to use the
low-rank alternating directions implicit (LR-ADI) solver \cite{Benner2017b} provided in the M-M.E.S.S. package \cite{Saak2022}.
This can be done by setting \texttt{options.lrradi = true} and supplying the low-rank factor $\bQ_2 = \bC^\top \bC$ via \texttt{options.C}.
Step 1 of \cref{ch3:alg:reduced-PPR} is replaced with solving the ARE for the low-rank factors $\overline\bV_2 \approx \overline\bZ \overline\bZ^\top$.
Subsequently, anywhere that $\overline\bV_2$ is required, we avoid forming $\overline\bV_2$ explicitly and deal instead directly with the low-rank factor $\overline\bZ$.
Most notably, in steps 3 and 4, we form $\bZ = \bE^\top \overline\bZ$, and the dominant eigenvectors of $\bV_2$ can be computed as the left singular vectors of $\bZ$ without needing to explicitly form $\bV_2$.
This procedure is summarized in \cref{ch3:alg:lr-PPR}.
\begin{algorithm}[!ht]
    \caption{Reduced approximation to the optimal feedback law $\bu(\bx)$ with low-rank adaptation}\label{ch3:alg:lr-PPR}
    \begin{algorithmic}[1]
        \Require Sparse system matrices $\bE$, $\bA$, $\bB$, $\{\bF_p\}_{p=2}^{d-1}$, $\{\bG_p\}_{p=1}^{d-2}$; cost function weights $\{\bq_p\}_{p=2}^d$, $\{\br_p\}_{p=0}^{d-2}$; desired approximation degree $d$; desired reduced-order $r$
        \Ensure Approximation of the polynomial gain coefficients $\{\bK_i\}_{i=1}^{d-1}$
        \State Solve the ARE \cref{ch3:eq:ARE} for the low-rank factors $\overline\bV_2\approx \overline\bZ \overline\bZ^\top$
        and the linear gain coefficient $\bK_1$.
        \State Compute and define $\bT$ to be the first $r$ left singular vectors of $\bZ = \bE^\top \overline\bZ$.
        \State Form the coefficients of the reduced-order dynamics and cost by projection, as in \cref{ch3:eq:ROM-operators}.
        \For{$k=3$ to $d$}
        \State Solve \cref{ch3:eq:reduced-value-function-linear-system} for $\bv_k^r$.
        \State Form the next reduced gain coefficient $\bK_{k-1}^r$ according to \cref{ch3:eq:reduced-control-gain-linear-system}.
        \EndFor
    \end{algorithmic}
\end{algorithm}

\begin{remark}
    Truncated polynomial feedback law approximations are not guaranteed to preserve global, and in some cases even local, stability of the closed-loop system.
    However, unlike other polynomial approximation techniques such as Galerkin approximations,
    Al'brekht's method \emph{does} guarantee local stabilization, which it inherits from the assumption that the linearized dynamics are stabilized by the solution to the Riccati equation.
    As such, the controllers computed using \cref{ch3:alg:full-PPR,ch3:alg:reduced-PPR} provide local stability guarantees, since they retain the full-order Riccati solution.
    Conversely, controllers computed using \cref{ch3:alg:lr-PPR} are only guaranteed to be locally stabilizing if the low-rank Riccati solver converges to a stabilizing solution.
\end{remark}

\section{Illustrative examples}\label{ch3:sec:examples}
To demonstrate some of the benefits of using polynomial feedback laws over linear approximations, we initially consider a 1D ODE model exhibiting a transcritical bifurcation in \cref{ch3:sec:example1}.
Next, we apply the PPR method to the inverted pendulum in \cref{ch3:sec:pendulum}, since most readers are familiar with the example and its simple two-dimensional dynamics can provide meaningful insight.
Finally, we consider a comparison with another Taylor-series based approach \cite{Breiten2018} on a 1D semi-discretized Fokker-Planck equation in \cref{ch3:sec:fokker-planck}.

The results shown in \cref{ch3:sec:examples,ch3:sec:allen-cahn} are computed on a Windows 11 laptop with an AMD Ryzen 9 HX 370 CPU with 32 GB of RAM.
The examples and numerical codes are available open-source in the \href{ch3:https://github.com/cnick1/PPR}{PPR Github repository}, with the specific version of the code used in this paper captured in the \texttt{v1.0.0} tag.
The results in the tables and figures can be reproduced by running the script \texttt{examplesForJCP.m}.

\subsection{1D example with transcritical bifurcation}\label{ch3:sec:example1}
Consider a 1D model with initial condition $x(0) = x_0$ and state equation featuring a cubic drift nonlinearity
\begin{align}
    \dot{x}(t) & = \varepsilon x(t) -  2x(t)^2 + x(t)^3 + u(t).
\end{align}
The open-loop system exhibits a transcritical bifurcation at $\varepsilon = 0$.
For $\varepsilon < 0$, the origin is \emph{stable}, and there is an \emph{unstable} equilibrium to the \emph{left} of the origin, whereas
for $\varepsilon > 0$, the origin is \emph{unstable}, and there is a \emph{stable} equilibrium to the \emph{right} of the origin.
For $\varepsilon = 0$, the origin is ``half-stable''.
\Cref{ch3:fig:1DPhasePortrait-openLoop} depicts the phase portrait for the unstable open-loop system with $\varepsilon = 1/10$.
We wish to globally asymptotically stabilize the model; to that end, we formulate an optimal control problem to find the control $u_*$ that minimizes the cost
\begin{equation}
    J(x_0,u) =
    \frac{1}{2} \int_{0}^\infty x(t)^2 + \frac{2}{5}x(t)^3 + u(t)^2 \,\,{\rm{d}} t.
\end{equation}
Using the method proposed in \cref{ch3:thm:kiPoly}, we compute a degree 4 approximation to the value function (rounded to four decimal places)
\begin{align*}
    V(x) =   0.5525 x^2   -0.6667 x^3  +  0.7724 x^4, 
\end{align*}
leading to a cubic control law
\begin{align*}
    u(x)
     & = - \frac{\partial^\top V(x)}{\partial x} = - 1.1050  x +  2  x^2  - 3.0896 x^3 .
\end{align*}
The closed-loop dynamics under the action of the cubic controller are
\begin{align*}
    \dot{x}(t) & =  - 1.0050 x(t) - 2.0896 x(t)^3.
\end{align*}
We compare this with the linear-quadratic regulator obtained by retaining only the linear term in the feedback law;
the closed-loop dynamics under the action of the LQR controller are
\begin{align*}
    \dot{x}(t) & =  - 1.0050 x(t) - 2x(t)^2 + x(t)^3.
\end{align*}

In \cref{ch3:fig:1DPhasePortrait-linearControl,ch3:fig:1DPhasePortrait-cubicControl}, we compare the model's closed-loop dynamics using the linear and cubic feedback laws with the open-loop phase portrait shown in \cref{ch3:fig:1DPhasePortrait-openLoop}.
While the linear feedback locally asymptotically stabilizes the origin, the overall behavior is qualitatively similar to the open-loop behavior.
This is because the linear controller has no effect on the quadratic and cubic components in the dynamics; hence a second unstable equilibrium is still present to the left of the origin which undesirably repels trajectories.

Conversely, the cubic controller eliminates the quadratic component of the drift due to the cubic term in the cost function.
Furthermore, it inverts the cubic dynamics so that they have a stabilizing rather than destabilizing effect.
As a result, the cubic controller is able to globally asymptotically stabilize the origin, as seen in the phase portrait in \cref{ch3:fig:1DPhasePortrait-cubicControl}.

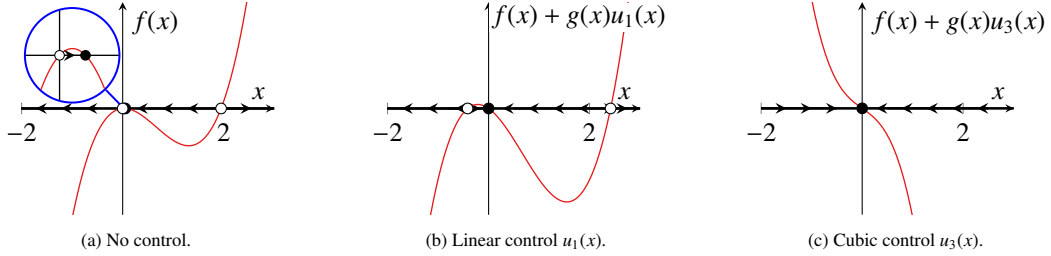
\begin{figure}[htb]
    \centering
    \begin{subfigure}[]{0.3\textwidth}
        \centering

        \def\spyfactor{10}
        \def\spyviewersize{1.25cm}

        \def\spyonclipreduce{0.25pt}
        \def\spywindowlinewidth{.75pt}
        \def\spywindowlinecolor{blue}

        \def\spypoint1{axis cs:.025,0}
        \def\spyviewer1{axis cs:-1,1.5}

        \begin{tikzpicture}[
                spy/.style={%
                        draw,blue,
                        line width=0.75pt,
                        circle,
                        inner sep=0pt,
                    },
            ]
            \begin{axis}[
                    axis lines=middle,
                    xmin=-2,xmax=3,
                    ymin=-3,ymax=3,
                    ymajorticks=false,
                    xlabel=$x$,
                    ylabel=$f(x)$,
                    width = 140pt,
                    height = 125pt,
                ]
                \addplot[thick,variable=\t, domain=-2:3, postaction=decorate] ({t}, {0});
                \addplot[thick,variable=\t, domain=-2.5:-.5, decoration={markings,
                            mark=between positions 0 and 1 step .25 with {\arrow [scale=1]{stealth reversed}}}, postaction=decorate] ({t}, {0});
                \addplot[thick,variable=\t, domain=0.5:2.5, decoration={markings,
                            mark=between positions 0 and .5 step .25 with {\arrow [scale=1]{stealth reversed}}
                        }, postaction=decorate] ({t}, {0});
                \addplot[thick,variable=\t, domain=2:3, decoration={markings,
                            mark=between positions .55 and .9 step .5 with {\arrow [scale=1]{stealth}}
                        }, postaction=decorate] ({t}, {0});

                \addplot[domain=-2:3, samples=100, color=Set1-A, line width = 0.5 ] {1/10*x - 2*x^2 + x^3};

                \addplot [only marks, mark=*, mark options={scale=1, draw=black, fill=black}] coordinates {(1-3/sqrt(10),0)};
                \addplot [only marks, mark=*, mark options={scale=1, draw=black, fill=white}] coordinates {(0,0)};
                \addplot [only marks, mark=*, mark options={scale=1, draw=black, fill=white}] coordinates {(1+3/sqrt(10),0)};

                \node[\spywindowlinecolor, line width=\spywindowlinewidth, circle, draw, minimum size={\spyviewersize/\spyfactor}, inner sep=0pt] (spypoint1) at (\spypoint1) {};
                \node[\spywindowlinecolor, line width=\spywindowlinewidth, circle, draw, minimum size=\spyviewersize, inner sep=0pt,fill=white] (spyviewer1) at (\spyviewer1) {};
                \draw[\spywindowlinecolor, line width=\spywindowlinewidth] (spypoint1) edge (spyviewer1);

            \end{axis}
            \begin{scope}
                \clip (spyviewer1) circle (0.5*\spyviewersize-\spyonclipreduce);
                \begin{scope}[
                        shift={($1/\spyfactor*(spyviewer1)-1/\spyfactor*(spypoint1)$)},
                        scale around={\spyfactor:(spypoint1)}
                    ]
                    \begin{axis}[
                            axis lines=middle,
                            axis line style = {line width = 0.05},
                            xmin=-2,xmax=3,
                            ymin=-.2,ymax=.2,
                            ymajorticks=false,
                            xlabel=$x$,
                            ylabel=$f(x)$,
                            width = 140pt,
                            height = 125pt,
                        ]
                        \addplot[domain=-.1:.1, samples=100, color=Set1-A, line width = 0.05] {1/10*x - 2*x^2 + x^3};

                        \addplot[line width = 0.015,variable=\t, domain=0:.0513, decoration={markings,
                                    mark=between positions .6 and .75 step .5 with {\arrow [scale=.19]{stealth}}
                                }, postaction=decorate] ({t}, {0});

                        \fill[black] (0.0513167,0) circle (.2pt);
                        \fill[black] (0,0) circle (.2pt);
                        \fill[white] (0,0) circle (.16pt);

                    \end{axis}
                \end{scope}
            \end{scope}

        \end{tikzpicture}

        \caption{No control.}
        \label{ch3:fig:1DPhasePortrait-openLoop}
    \end{subfigure}
    \begin{subfigure}[]{0.3\textwidth}
        \centering
        \def\spyfactor{5}
        \def\spyviewersize{1.75cm}

        \def\spyonclipreduce{0.25pt}
        \def\spywindowlinewidth{.75pt}
        \def\spywindowlinecolor{blue}

        \def\spypoint1{axis cs:-.115,0}
        \def\spyviewer1{axis cs:2,5}

        \begin{tikzpicture}[
                spy/.style={%
                        draw,blue,
                        line width=0.75pt,
                        circle,
                        inner sep=0pt,
                    },
            ]
            \begin{axis}[
                    axis lines=middle,
                    xmin=-2,xmax=3,
                    ymin=-3,ymax=3,
                    ymajorticks=false,
                    xlabel=$x$,
                    ylabel=$f(x) + g(x) u_{1}(x)$,
                    ylabel style={fill=white, inner sep=1pt},
                    width = 140pt,
                    height = 125pt,
                ]
                \addplot[thick,variable=\t, domain=-2:3, postaction=decorate] ({t}, {0});
                \addplot[thick,variable=\t, domain=-2.5:-.5, decoration={markings,
                            mark=between positions 0 and .9 step .25 with {\arrow [scale=1]{stealth reversed}}}, postaction=decorate] ({t}, {0});
                \addplot[thick,variable=\t, domain=0.5:2.5, decoration={markings,
                            mark=between positions 0 and .75 step .25 with {\arrow [scale=1]{stealth reversed}}
                        }, postaction=decorate] ({t}, {0});
                \addplot[thick,variable=\t, domain=2.5:3, decoration={markings,
                            mark=between positions 0 and .5 step .5 with {\arrow [scale=1]{stealth}}
                        }, postaction=decorate] ({t}, {0});
                \addplot[thick,variable=\t, domain=-.4:0, decoration={markings,
                            mark=between positions 0.6 and .9 step .5 with {\arrow [scale=1]{stealth}}
                        }, postaction=decorate] ({t}, {0});
                \addplot[thick,variable=\t, domain=-.4:-.8, decoration={markings,
                            mark=between positions 0.65 and .9 step .5 with {\arrow [scale=1]{stealth}}
                        }, postaction=decorate] ({t}, {0});

                \addplot[domain=-2:3, samples=100, color=Set1-A, line width = 0.5 ] {- 1.0050*x  - 2*x^2 + x^3};

                \addplot [only marks, mark=*, mark options={scale=1, draw=black, fill=black}] coordinates {(0,0)};
                \addplot [only marks, mark=*, mark options={scale=1, draw=black, fill=white}] coordinates {(-.41598,0)};
                \addplot [only marks, mark=*, mark options={scale=1, draw=black, fill=white}] coordinates {(2.416,0)};

            \end{axis}

        \end{tikzpicture}        \caption{Linear control $u_1(x)$.}
        \label{ch3:fig:1DPhasePortrait-linearControl}
    \end{subfigure}
    \begin{subfigure}[]{0.3\textwidth}
        \centering
        \begin{tikzpicture}
            \begin{axis}[
                    axis lines=middle,
                    xmin=-2,xmax=3,
                    ymin=-3,ymax=3,
                    ymajorticks=false,
                    xlabel=$x$,
                    ylabel=$f(x) + g(x) u_{3}(x)$,
                    width = 140pt,
                    height = 125pt,
                ]
                \addplot[thick,variable=\t, domain=-2:3, postaction=decorate] ({t}, {0});
                \addplot[thick,variable=\t, domain=-2.5:-.5, decoration={markings,
                            mark=between positions 0 and 1 step .25 with {\arrow [scale=1]{stealth}}}, postaction=decorate] ({t}, {0});
                \addplot[thick,variable=\t, domain=0.5:2.5, decoration={markings,
                            mark=between positions 0 and 1 step .25 with {\arrow [scale=1]{stealth reversed}}
                        }, postaction=decorate] ({t}, {0});
                \addplot[domain=-2:5, samples=100, color=Set1-A, line width = 0.5 ] { - 1.0050 *x - 2.0699 *x^3  };
                \addplot [only marks, mark=*, mark options={scale=1, draw=black, fill=black}] coordinates {(0,0)};
            \end{axis}
        \end{tikzpicture}        \caption{Cubic control $u_3(x)$.}
        \label{ch3:fig:1DPhasePortrait-cubicControl}
    \end{subfigure}
    \caption{Phase portrait for the 1D example for $\varepsilon = 1/10$ with (a) no control, (b) linear control, and (c) cubic control laws.}
    \label{ch3:fig:1DPhasePortraits}
\end{figure}

In general it is not guaranteed that higher-order controllers succeed in global stabilization, nor does the size of the region of stabilization necessarily increase with respect to the degree of the polynomial control law.
However, for this problem, we see that polynomial feedback provides additional flexibility for the control law to influence higher-order nonlinear terms in the dynamics.

\subsection{2D inverted pendulum example}\label{ch3:sec:pendulum}
The governing equations for an inverted pendulum with torque control $u$ applied at the pivot point, are
\begin{equation}
    \begin{split}
        \dot{x}_1(t) & = x_2(t)  ,                        \\
        \dot{x}_2(t) & = \sin \left(x_1(t)\right) + u(t),
    \end{split}
\end{equation}
where the state is given by the angular position $x_1$ and the angular velocity $x_2$ of the pendulum relative to the unstable equilibrium in the ``up'' position.
\Cref{ch3:fig:pendulumPhasePortrait} depicts the phase portrait for the inverted pendulum with no control action, with a particular trajectory corresponding to the initial condition $x_1(0) = -\pi$ and $x_2(0) = 2$ highlighted.
The highlighted trajectory corresponds to a trajectory for which no control input is required to bring the system to the equilibrium, hence the optimal control is known to be $u_* = 0$.
\begin{figure}[htb]
    \centering
    \includegraphics[page=2]{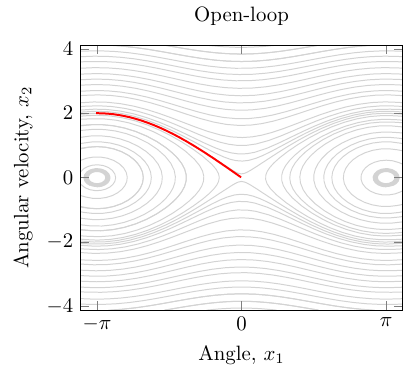}
    \caption{Inverted pendulum. Open-loop phase portrait. }
    \label{ch3:fig:pendulumPhasePortrait}
\end{figure}

To stabilize the unstable vertical position with minimum energy control, we formulate an optimal control problem to minimize the cost
\begin{equation}
    J(\bx_0,u) =
    \frac{1}{2} \int_{0}^\infty  u(t)^2 \,\,{\rm{d}} t. \label{ch3:eq:pendulum-cost}
\end{equation}
Using a Taylor expansion for $\sin (x_1)$, the dynamics are rewritten as
\begin{align*}
    \dot{x}_1(t) & = x_2(t)                                                           \\
    \dot{x}_2(t) & = \sum_{j=0}^\infty\frac{ (-1)^j x_1(t)^{2j+1}}{ (2j+1)!} + u(t) ,
\end{align*}
which can alternatively be written in matrix form as:
\begin{align*}
    \dot{\bx}(t) & = \bA \bx(t) + \sum_{i=2}^\infty \bF_i \kronF{\bx(t)}{i} + \bB u(t),
\end{align*}
where the linear system matrix is $\bA = \begin{bmatrix}0 & 1 \\ 1 & 0\end{bmatrix}$, the input matrix is $\bB = \begin{bmatrix}0 \\ 1 \end{bmatrix}$,
and the polynomial drift coefficients $\bF_i \in \real^{n \times n^i}$ are matrices of all zeros if $i$ is even, and all zero except for the (2,1) entry given by
$(-1)^{(i-1)/2}/i!$
if $i$ is odd.

We compute controllers of degree $1$, $3$, $5$, $7$, and $9$.
The closed-loop phase portraits under the action of these controllers depicted in \cref{ch3:sfig:pendulumPhasePortrait_closedLoop_a}.
\Cref{ch3:sfig:pendulumPhasePortrait_closedLoop_b} shows contour plots of the corresponding value function approximations of degree $2$, $4$, $6$, $8$, and $10$, and \cref{ch3:sfig:pendulumPhasePortrait_closedLoop_c} illustrates the residuals of the HJB PDE for each of these approximations.
For this problem, only the linear controller achieves global stabilization, with the polynomial controllers only stabilizing the system locally.
However, the higher-order control laws are closer to achieving optimality locally; compare the highlighted trajectories in \cref{ch3:sfig:pendulumPhasePortrait_closedLoop_a}, which correspond to the trajectory in \cref{ch3:fig:pendulumPhasePortrait} which asymptotically approaches the origin in the open-loop.
The higher-order controllers are able to replicate the open-loop behavior, i.e. no control, leading to a lower cost.
The linear feedback initially accelerates the pendulum, after which it has to overcompensate to negate the effect, leading to a relatively high control cost.
The degree 9 controller in particular uses very little energy and behaves similarly to the uncontrolled behavior, which is optimal for this initial condition.

\begin{figure}[htb]
    \centering
    \begin{subfigure}[]{\textwidth}
        \centering
        \includegraphics[page=3,width=\textwidth]{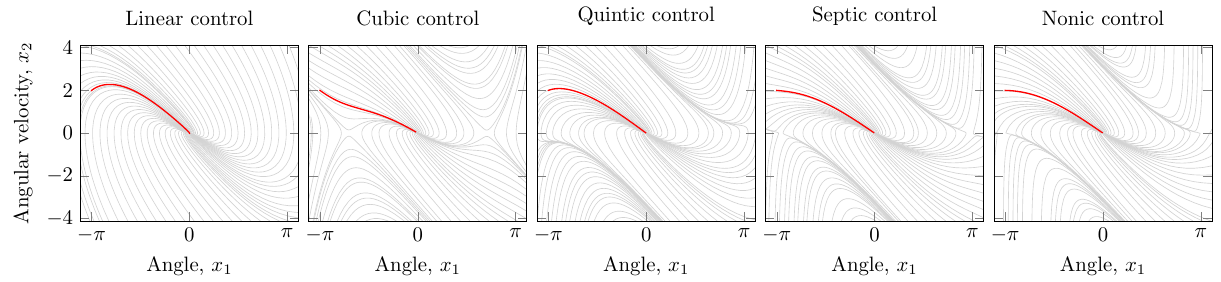}
        \caption{Closed-loop phase portraits; the highlighted trajectories have initial condition $x_1(0) = \theta_0$, $x_2(0) = \sqrt{2 (1-\cos (\theta_0))}$. }
        \label{ch3:sfig:pendulumPhasePortrait_closedLoop_a}
    \end{subfigure}

    \begin{subfigure}[]{\textwidth}
        \centering
        \includegraphics[page=4,width=\textwidth]{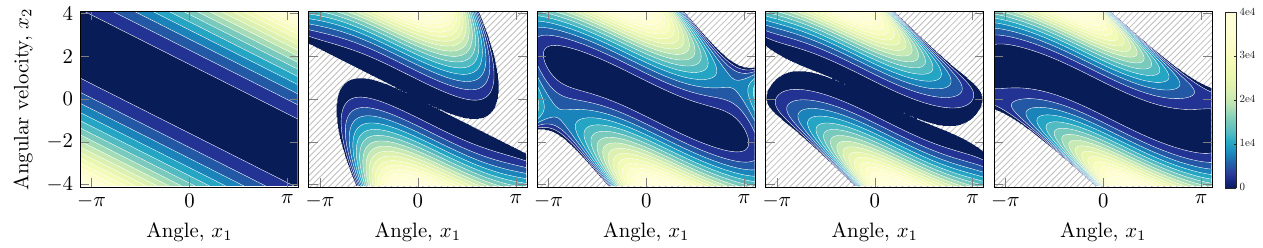}
        \caption{Value function approximations; patterned areas are where the value function approximation is negative, which is prohibited by the value function definition. }
        \label{ch3:sfig:pendulumPhasePortrait_closedLoop_b}
    \end{subfigure}

    \begin{subfigure}[]{\textwidth}
        \centering
        \includegraphics[page=5,width=\textwidth]{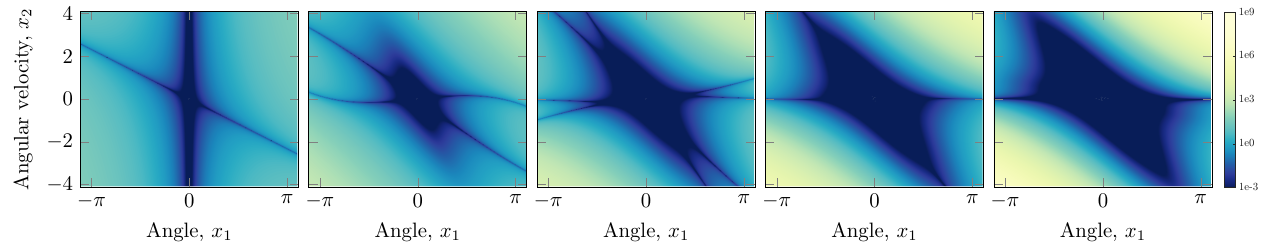}
        \caption{HJB PDE residuals; higher degree approximations are more accurate over a larger region.}
        \label{ch3:sfig:pendulumPhasePortrait_closedLoop_c}
    \end{subfigure}
    \caption{Inverted pendulum. Closed-loop phase portraits (a) with controls given by degree $2$, $4$, $6$, $8$, and $10$ value function approximations (b).
        The hatched regions in (b) correspond to regions where the value function is negative.
        The HJB PDE residuals in (c) can be interpreted loosely as a pointwise error metric, indicating the area where the approximation to the value function is good. Dark blue areas have errors of $10^{-3}$ or smaller, whereas the bright yellow areas have large errors; these do not necessarily correlate with the region of stabilization.
    }
    \label{ch3:fig:pendulumPhasePortrait_closedLoop}
\end{figure}

There is a tradeoff for this problem; as shown in the phase portraits, value function approximations, and in the HJB residuals, despite there being a region in which the PPR controllers perform better, outside of this region, the errors can be very high.
This is a hallmark feature of polynomial approximation, and it is well explained by the fact that the computed solutions are Taylor approximations.
Regardless of the radius of convergence,
once truncated,
higher-order polynomials grow more quickly than low-order or quadratic polynomials.
This often makes the higher-order approximations worse far from the equilibrium point.
This can be seen in the HJB residual plots, for which the values for the highest errors grow simultaneously as the region in which the error is low grows.
Still, the dark blue regions where the approximations are very good grow with increasing degrees of Taylor approximation.
Note that the HJB residual plots should be only be interpreted as an indication of the local approximation quality rather than a guarantee of performance or closed-loop stability.
For example, the thin tendrils in \cref{ch3:sfig:pendulumPhasePortrait_closedLoop_c} emanating from the origin, despite being darkly colored, do not always correlate with the closed-loop region of stability.

In \cref{ch3:tab:valueFunApproxs}, we compare the value function approximations and the integrated control costs \cref{ch3:eq:pendulum-cost} for each of the controllers until the final time $T=10$.
Any initial condition of the form $x_1(0) = \theta_0$ and $x_2(0) = \sqrt{2 (1-\cos (\theta_0))}$ lies along an open-loop trajectory that asymptotically approaches the origin,
meaning that the optimal control is known to be $u_* = 0$ and the optimal cost is likewise zero.
The first column shows the value function approximations, which approach zero as $d \to \infty$.
The true value function should be nonnegative definite, but some of the approximations are negative; this is a downside of the polynomial approximation, as they are only guaranteed to be valid locally, and there is no inherent constraint guaranteeing positivity outside of some local region.
In the second column, we see that the integrated costs tend towards zero as the control law is refined, confirming the result shown in \cref{ch3:fig:pendulumPhasePortrait_closedLoop} that the degree 9 controller approximates the optimal control very well.
The value function is defined as the value of the performance index under the action of the optimal control, so we can see that as $d$ tends to infinity, the left and right column values both are tending towards being equal to zero.

\begin{table}[htb]
    \centering
    \caption{Inverted pendulum. Value function approximations and control costs ($\bx_0 = [-\pi  \quad 2]^\top$) computed up to $T = 10$.}
    \begin{tabular}{ccc}
        \toprule
        $d$ & $\frac{1}{2} \sum_{i=2}^{d} \bv_i^\top \kronF{\bx_0}{i}$ & $ \frac{1}{2} \int_0^T u(t)^2 \rd t$ \\
        \midrule
        2   & 1.30323                                                  & 0.44946                              \\
        4   & -2.60971                                                 & 0.48623                              \\
        6   & 1.98305                                                  & 0.06537                              \\
        8   & -0.86663                                                 & 0.00287                              \\
        10  & 0.19815                                                  & 0.00016                              \\
        12  & -0.00135                                                 & 0.00001                              \\
        \bottomrule
    \end{tabular}
    \label{ch3:tab:valueFunApproxs}
\end{table}

\subsection{1D Fokker-Planck equation}\label{ch3:sec:fokker-planck}
Next, we consider the Fokker-Planck equation, which describes the time evolution of the probability density function of a stochastic process.
We follow the formulation in \cite{Breiten2018} in order to make a direct comparison.
For brevity, we defer the discussion of the governing PDE and its discretization to \cite{Breiten2018}.
Upon discretization, a finite-dimensional \textit{bilinear} system of dimension $n=99$ is obtained, featuring a linear drift vector field and an input map that is linear in the state:
\begin{align*}
    \dot{\bx} & = \bA \bx + \bB \bu + \bG_1 (\bx \otimes \bu).
\end{align*}
This is a special case of \cref{ch3:eq:FOM-Poly}.
The cost function considered in \cite{Breiten2018} is a standard quadratic cost.
We apply the proposed PPR method to test case 2 from \cite{Breiten2018} and compare with the results therein.

The value function approximation used in this work is identical to the one used in \cite{Breiten2018},
which is also based on Al'brekht's method.
The only difference in the value function computation is in the implementation: the authors in \cite{Breiten2018} use a quadrature-based $k$-way Lyapunov solver that is approximately 4 times slower than the Bartels-Stewart solver used in this work.

However, the polynomial control laws in \cite{Breiten2018} are constructed slightly different than those we propose.
The control laws we propose, if the reduction technique of \cref{ch3:sec:reduced-Albrekht} is not applied, are consistent with Al'brekht's method and are precisely Taylor approximants of the optimal control law.
For a degree $d$ value function approximation, these correspond to degree $d-1$ approximations to the optimal control law.
For example, a quadratic value function approximation $V(\bx) \approx \frac{1}{2} \bx^\top \bV_2 \bx$ obtained by solving the Riccati equation \cref{ch3:eq:ARE} leads to a linear PPR (LQR) controller $\bu_1(\bx) \approx -\bR^{-1} \bB^\top \bV_2 \bx$.
The authors in \cite{Breiten2018} instead compute a degree $d$ approximation to the optimal control law corresponding to each degree $d$ approximation to the value function.
For example, with the quadratic value function approximation $V(\bx) \approx \frac{1}{2} \bx^\top \bV_2 \bx$,
the authors in \cite{Breiten2018} instead compute a quadratic controller $\bu_{b2}(\bx) \approx -\bR^{-1} (\bB + \bG_1\bx)^\top \bV_2 \bx$ by retaining every term contributed by the nonlinear input map.
While this may appear sensible, as it represents \cref{ch3:eq:HJB-PDE-u} with the quadratic value function approximation, this type of polynomial control law corresponds to an incomplete Taylor approximant, with all but one of the degree $d$ terms in the control law omitted.
Al'brekht's method, and the PPR package, instead only retains terms contributed by the nonlinear input map \textit{if the resulting term remains degree $d-1$}, which is consistent with the Taylor approximation to the optimal control law.
As we shall see, this has ramifications on closed-loop performance.

In \cref{fig:example31_evolution}, we show the evolution of the state, which represents the distance of the probability density function from stationary density, which represents the steady-state solution.
On the left, we have the uncontrolled trajectory, whereas on the right we have the optimal trajectory under the action of the open-loop optimal control signal.
The optimal control $u_{\text{opt}}(t)$ is computed using a gradient-descent method in \cite{Breiten2018}, which is expensive and does not produce the control law in feedback form, but can be used as a reference to compare other controllers against.

\begin{figure}[htb]
    \centering
    \includegraphics[width=0.75\textwidth]{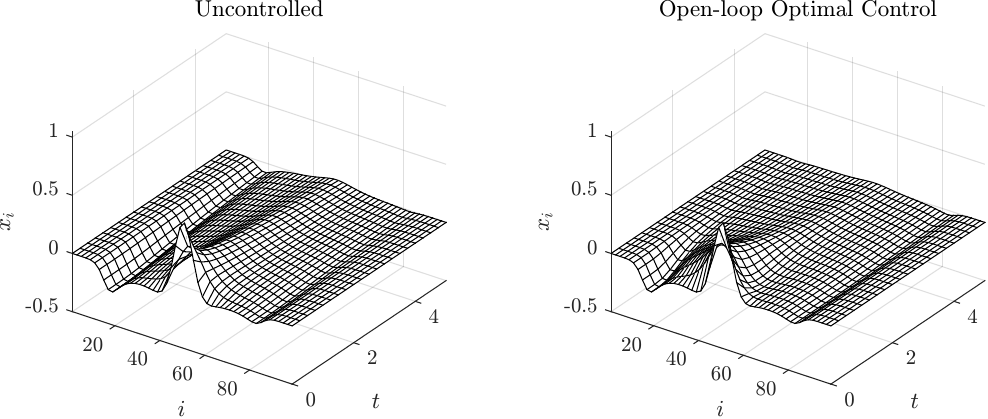}
    \caption{1D Fokker-Planck equation. Simulations of the uncontrolled and optimally controlled system to illustrate the system dynamics.
        The axes represent the individual state components $x_i$, the index $i$ for each state component, and time $t$.}
    \label{fig:example31_evolution}
\end{figure}

\Cref{fig:example31_controlSignals} compares the control signals computed in \cite{Breiten2018} with the Taylor approximants computed with the PPR method proposed in this work.
In the left plot, the signals $u_{b2}(t)$, $u_{b3}(t)$, $u_{b4}(t)$ are produced by the quadratic, cubic, and quartic control laws from \cite{Breiten2018} corresponding to quadratic, cubic, and quartic value function approximations.
On the right plot, $u_{1}(t)$, $u_{2}(t)$, $u_{3}(t)$ are linear\footnote{Note that a linear PPR controller is precisely the LQR controller.}, quadratic, and cubic PPR control laws corresponding to the same quadratic, cubic, and quartic value function approximations.
We also show on the left plot the controllers $u^{\text{PPR}}_{b2}(t)$, $u^{\text{PPR}}_{b3}(t)$, $u^{\text{PPR}}_{b4}(t)$ to illustrate that the PPR package can also compute the same control laws from \cite{Breiten2018} by simply adding the additional degree $d$ term contributed by the bilinear input term.
The blue, red, and green curves on the left therefore correspond exactly to the blue, red, and green curves on the right (representing Taylor approximants) with one additional term added to the control law.

\begin{figure}[htb]
    \centering
    \includegraphics[width=0.75\textwidth]{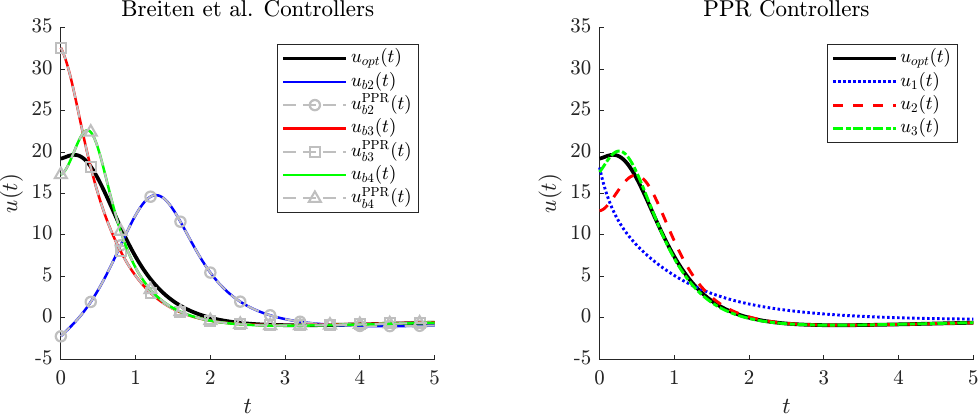}
    \caption{1D Fokker-Planck equation. Optimal control approximations using the method of \cite{Breiten2018} (cf. Fig. 3b in \cite{Breiten2018}) (left) versus the Taylor approximants computed using the PPR method proposed in this work (right). The gray lines on the left illustrate that the PPR package can be modified to reproduce the controllers from \cite{Breiten2018} by adding the additional term contributed by the bilinear input.}
    \label{fig:example31_controlSignals}
\end{figure}

It appears that the addition of the isolated degree $d$ term in each control law corresponding to the bilinear input term, \textit{without also adding the other degree $d$ terms from the Taylor approximation}, is not a straightforward improvement over a degree $d-1$ control law.
Indeed, $u_{b2}(t)$, which only requires the solution of the LQR Riccati equation, behaves dramatically differently than the related LQR controller; it in fact completely inverts the sign of the initial control response (incorrectly, as seen by the sign of the true optimal control response).
While the higher-order controllers from \cite{Breiten2018} improve, they still are less accurate than the simpler PPR controllers that simply omit the \textit{additional} bilinear input term.
We emphasize that the PPR control laws \textit{do include} contributions from the bilinear terms; they however only include them up to degree of the Taylor approximation, i.e. ensuring that all of the terms of a particular degree either are, or are not, included in the control law.
Comparing the blue quadratic PPR control $u_2(t)$ on the right with the blue $u_{b2}(t)$ on the left, we see that the additional quadratic terms corresponding to the next value function coefficient are necessary to ensure that the resulting control law produces the control signal with the correct sign, so while they do impose a significant computational complexity penalty to compute, they are nonetheless necessary if a quadratic feedback law is desired.

In \cref{ch3:tab:valueFunApproxs-fp1}, we compare the integrated costs for each controller with the uncontrolled cost and the optimal cost.
We see that in each case, the degree $d$ feedback law from \cite{Breiten2018} performs worse than both the degree $d$ and $d-1$ PPR controllers.
We conclude that, from the perspective of both computational complexity and control performance, it is ill-advised to augment an optimal control Taylor approximant with additional polynomial terms without careful consideration.
\begin{table}[htb]
    \centering
    \caption{1D Fokker-Planck equation. Integrated control costs computed up to $T =  20$.}
    \begin{tabular}{lc}
        \toprule
        Controller          & Integrated Cost \\
        \midrule
        $0$                 & 0.159878        \\
        $u_1(t)$            & 0.120474        \\
        $u_{b2}(t)$         & 0.129229        \\
        $u_2(t)$            & 0.115332        \\
        $u_{b3}(t)$         & 0.115957        \\
        $u_3(t)$            & 0.114724        \\
        $u_{b4}(t)$         & 0.115002        \\
        $u_{\text{opt}}(t)$ & 0.114710        \\
        \bottomrule
    \end{tabular}
    \label{ch3:tab:valueFunApproxs-fp1}
\end{table}

Next, we perform a convergence study to illustrate the behavior of the PPR feedback laws as the initial condition is scaled to get further from the origin.
\Cref{fig:example31_regionOfAccuracy} depicts the closed-loop cost gap, $J\left(\bx_0,\cdot\right) - J\left(\bx_0,u_{\text{opt}}(t)\right)$, as a measure of sub-optimality for each controller, as a function of the norm of the initial condition.

\pgfplotstableread[col sep=&, row sep=\\]{
    0.000000e+00      & 0.000000e+00     & 0.000000e+00     & 0.000000e+00     &        NaN       &        NaN       &        NaN       & 0.000000e+00     & 0.000000e+00     & 0.000000e+00     & 0.000000e+00\\
    1.650066e-01      & -1.967021e-06    & -2.046902e-06    & -2.047623e-06    & -1.570438e-03    & -1.634213e-03    & -1.634789e-03    & 1.252531e-03     & 1.250564e-03     & 1.250484e-03     & 1.250483e-03\\
    3.300133e-01      & -3.377168e-07    & -2.003685e-06    & -2.057088e-06    & -6.736440e-05    & -3.996752e-04    & -4.103276e-04    & 5.013283e-03     & 5.012945e-03     & 5.011279e-03     & 5.011226e-03\\
    4.950199e-01      & 1.030580e-05     & -1.083903e-06    & -1.726708e-06    & 9.142928e-04     & -9.615994e-05    & -1.531872e-04    & 1.127188e-02     & 1.128219e-02     & 1.127080e-02     & 1.127015e-02 \\
    6.600266e-01      & 5.050948e-05     & 2.529070e-06     & -1.217192e-06    & 2.528434e-03     & 1.266017e-04     & -6.093092e-05    & 1.997659e-02     & 2.002710e-02     & 1.997911e-02     & 1.997537e-02  \\
    8.250332e-01      & 1.656353e-04     & 1.371260e-05     & -8.373300e-07    & 5.338129e-03     & 4.419325e-04     & -2.698564e-05    & 3.102872e-02     & 3.119436e-02     & 3.104244e-02     & 3.102789e-02  \\
    9.900399e-01      & 4.384961e-04     & 4.326558e-05     & 3.814877e-07     & 9.904112e-03     & 9.772199e-04     & 8.616489e-06     & 4.427415e-02     & 4.471265e-02     & 4.431741e-02     & 4.427453e-02 \\
    1.155047e+00      & 9.912071e-04     & 1.054740e-04     & 2.741683e-06     & 1.665733e-02     & 1.772500e-03     & 4.607425e-05     & 5.950577e-02     & 6.049698e-02     & 5.961124e-02     & 5.950851e-02 \\
    1.320053e+00      & 1.972415e-03     & 2.137487e-04     & 5.496409e-06     & 2.578984e-02     & 2.794819e-03     & 7.186696e-05     & 7.648034e-02     & 7.845276e-02     & 7.669409e-02     & 7.648584e-02 \\
    1.485060e+00      & 3.526865e-03     & 3.796039e-04     & 7.516448e-06     & 3.714337e-02     & 3.997819e-03     & 7.915988e-05     & 9.495275e-02     & 9.847961e-02     & 9.533235e-02     & 9.496026e-02 \\
    1.650066e+00      & 5.764622e-03     & 6.221375e-04     & 1.452900e-05     & 5.025395e-02     & 5.423577e-03     & 1.266588e-04     & 1.147098e-01     & 1.204744e-01     & 1.153319e-01     & 1.147243e-01 \\
    1.815073e+00      & 8.748502e-03     & 9.870973e-04     & 5.083895e-05     & 6.452371e-02     & 7.280239e-03     & 3.749577e-04     & 1.355858e-01     & 1.443343e-01     & 1.365729e-01     & 1.356367e-01 \\
    1.980080e+00      & 1.249475e-02     & 1.548909e-03     & 1.616323e-04     & 7.935011e-02     & 9.836618e-03     & 1.026474e-03     & 1.574635e-01     & 1.699583e-01     & 1.590124e-01     & 1.576252e-01 \\
    2.145086e+00      & 1.698659e-02     & 2.396161e-03     & 4.093781e-04     & 9.423249e-02     & 1.329261e-02     & 2.271010e-03     & 1.802626e-01     & 1.972492e-01     & 1.826588e-01     & 1.806720e-01 \\
    2.310093e+00      & 2.218498e-02     & 3.616894e-03     & 8.632951e-04     & 1.087873e-01     & 1.773596e-02     & 4.233292e-03     & 2.039300e-01     & 2.261149e-01     & 2.075469e-01     & 2.047933e-01 \\
    2.475100e+00      & 2.803813e-02     & 5.285314e-03     & 1.598212e-03     & 1.227421e-01     & 2.313744e-02     & 6.996468e-03     & 2.284312e-01     & 2.564694e-01     & 2.337165e-01     & 2.300294e-01 \\
    2.640106e+00      & 3.448493e-02     & 7.462359e-03     & 2.687280e-03     & 1.359040e-01     & 2.940892e-02     & 1.059048e-02     & 2.537447e-01     & 2.882297e-01     & 2.612071e-01     & 2.564320e-01 \\
    2.805113e+00      & 4.147015e-02     & 1.020170e-02     & 4.204783e-03     & 1.481854e-01     & 3.645376e-02     & 1.502496e-02     & 2.798532e-01     & 3.213233e-01     & 2.900549e-01     & 2.840580e-01 \\
    2.970120e+00      & 4.893268e-02     & 1.354191e-02     & 6.213483e-03     & 1.595204e-01     & 4.414661e-02     & 2.025594e-02     & 3.067487e-01     & 3.556814e-01     & 3.202906e-01     & 3.129622e-01 \\
    3.135126e+00      & 5.681778e-02     & 1.752206e-02     & 8.763185e-03     & 1.698977e-01     & 5.239484e-02     & 2.620386e-02     & 3.344235e-01     & 3.912412e-01     & 3.519455e-01     & 3.431867e-01 \\
    3.300133e+00      & 6.507363e-02     & 2.217402e-02     & 1.187477e-02     & 1.793292e-01     & 6.110692e-02     & 3.272434e-02     & 3.628725e-01     & 4.279462e-01     & 3.850466e-01     & 3.747473e-01 \\
    3.465140e+00      & 7.365317e-02     & 2.752648e-02     & 1.553617e-02     & 1.878469e-01     & 7.020423e-02     & 3.962385e-02     & 3.920914e-01     & 4.657446e-01     & 4.196179e-01     & 4.076276e-01 \\
    3.630146e+00      & 8.251370e-02     & 3.360602e-02     & 1.970500e-02     & 1.954945e-01     & 7.962061e-02     & 4.668580e-02     & 4.220769e-01     & 5.045906e-01     & 4.556829e-01     & 4.417819e-01 \\
    3.795153e+00      & 9.161636e-02     & 4.043563e-02     & 2.432394e-02     & 2.023214e-01     & 8.929620e-02     & 5.371587e-02     & 4.528259e-01     & 5.444423e-01     & 4.932615e-01     & 4.771498e-01 \\
    3.960159e+00      & 1.008673e-01     & 4.798271e-02     & 2.928083e-02     & 2.082336e-01     & 9.905699e-02     & 6.044824e-02     & 4.843950e-01     & 5.852623e-01     & 5.323777e-01     & 5.136758e-01 \\
    4.125166e+00      & 1.104164e-01     & 5.645521e-02     & 3.472564e-02     & 2.137365e-01     & 1.092822e-01     & 6.721954e-02     & 5.166004e-01     & 6.270168e-01     & 5.730556e-01     & 5.513261e-01 \\
    4.290173e+00      & 1.200560e-01     & 6.569792e-02     & 4.048252e-02     & 2.184346e-01     & 1.195333e-01     & 7.365545e-02     & 5.496201e-01     & 6.696762e-01     & 6.153181e-01     & 5.901027e-01 \\
    4.455179e+00      & 1.298257e-01     & 7.581048e-02     & 4.665325e-02     & 2.225374e-01     & 1.299486e-01     & 7.996949e-02     & 5.833881e-01     & 7.132138e-01     & 6.591986e-01     & 6.300414e-01 \\
    4.620186e+00      & 1.397044e-01     & 8.683807e-02     & 5.333001e-02     & 2.260951e-01     & 1.405371e-01     & 8.630830e-02     & 6.179013e-01     & 7.576057e-01     & 7.047394e-01     & 6.712313e-01 \\
    4.785193e+00      & 1.496679e-01     & 9.882382e-02     & 6.063698e-02     & 2.291433e-01     & 1.513004e-01     & 9.283591e-02     & 6.531630e-01     & 8.028309e-01     & 7.519868e-01     & 7.138000e-01 \\
    4.950199e+00      & 1.597070e-01     & 1.118510e-01     & 6.876770e-02     & 2.317401e-01     & 1.622996e-01     & 9.978424e-02     & 6.891640e-01     & 8.488709e-01     & 8.010150e-01     & 7.579317e-01 \\
    5.115206e+00      & 1.698080e-01     & 1.260259e-01     & 7.795758e-02     & 2.339271e-01     & 1.736130e-01     & 1.073942e-01     & 7.259012e-01     & 8.957092e-01     & 8.519271e-01     & 8.038588e-01 \\
    5.280213e+00      & 1.799610e-01     & 1.415057e-01     & 8.848998e-02     & 2.357454e-01     & 1.853697e-01     & 1.159202e-01     & 7.633700e-01     & 9.433310e-01     & 9.048757e-01     & 8.518600e-01 \\
    5.445219e+00      & 1.901572e-01     & 1.585590e-01     & 1.006674e-01     & 2.372318e-01     & 1.978113e-01     & 1.255882e-01     & 8.015672e-01     & 9.917245e-01     & 9.601263e-01     & 9.022347e-01 \\
    5.610226e+00      & 2.003898e-01     & 1.776611e-01     & 1.147801e-01     & 2.384207e-01     & 2.113785e-01     & 1.365637e-01     & 8.404882e-01     & 1.040878e+00     & 1.018149e+00     & 9.552683e-01 \\
    5.775233e+00      & 2.106513e-01     & 1.998534e-01     & 1.310583e-01     & 2.393410e-01     & 2.270725e-01     & 1.489078e-01     & 8.801304e-01     & 1.090782e+00     & 1.079984e+00     & 1.011189e+00 \\
    5.940239e+00      & 2.209396e-01     & 2.283191e-01     & 1.496520e-01     & 2.400244e-01     & 2.480412e-01     & 1.625789e-01     & 9.204883e-01     & 1.141428e+00     & 1.148807e+00     & 1.070140e+00 \\
    6.105246e+00      & 2.312505e-01     & 3.006908e-01     & 1.706276e-01     & 2.404954e-01     & 3.127118e-01     & 1.774489e-01     & 9.615589e-01     & 1.192809e+00     & 1.262250e+00     & 1.132187e+00 \\
    6.270252e+00      & 2.415826e-01     & 4.185280e-01     & 1.939959e-01     & 2.407790e-01     & 4.171359e-01     & 1.933506e-01     & 1.003338e+00     & 1.244920e+00     & 1.421866e+00     & 1.197333e+00 \\
    6.435259e+00      & 2.519358e-01     & 4.248316e-01     & 2.197376e-01     & 2.408981e-01     & 4.062190e-01     & 2.101105e-01     & 1.045819e+00     & 1.297755e+00     & 1.470651e+00     & 1.265557e+00 \\
    6.600266e+00      & 2.623087e-01     & 4.381097e-01     & 2.478215e-01     & 2.408710e-01     & 4.023042e-01     & 2.275678e-01     & 1.089001e+00     & 1.351310e+00     & 1.527111e+00     & 1.336823e+00 \\
    6.765272e+00      & 2.727017e-01     & 4.546778e-01     & 2.782214e-01     & 2.407155e-01     & 4.013470e-01     & 2.455878e-01     & 1.132879e+00     & 1.405581e+00     & 1.587557e+00     & 1.411101e+00 \\
    6.930279e+00      & 2.831165e-01     & 4.734114e-01     & 3.109190e-01     & 2.404489e-01     & 4.020649e-01     & 2.640613e-01     & 1.177450e+00     & 1.460567e+00     & 1.650861e+00     & 1.488369e+00 \\
    7.095286e+00      & 2.935587e-01     & 4.938158e-01     & 3.458721e-01     & 2.400896e-01     & 4.038716e-01     & 2.828746e-01     & 1.222705e+00     & 1.516264e+00     & 1.716521e+00     & 1.568577e+00 \\
    7.260292e+00      & 3.040267e-01     & 5.156260e-01     & 3.829884e-01     & 2.396469e-01     & 4.064385e-01     & 3.018879e-01     & 1.268644e+00     & 1.572671e+00     & 1.784271e+00     & 1.651633e+00 \\
    7.425299e+00      & 3.145255e-01     & 5.386908e-01     & 4.221380e-01     & 2.391351e-01     & 4.095690e-01     & 3.209534e-01     & 1.315263e+00     & 1.629788e+00     & 1.853954e+00     & 1.737401e+00 \\
    7.590306e+00      & 3.250578e-01     & 5.629136e-01     & 4.632222e-01     & 2.385647e-01     & 4.131306e-01     & 3.399657e-01     & 1.362556e+00     & 1.687614e+00     & 1.925470e+00     & 1.825778e+00 \\
    7.755312e+00      & 3.356291e-01     & 5.882282e-01     & 5.062689e-01     & 2.379472e-01     & 4.170296e-01     & 3.589239e-01     & 1.410519e+00     & 1.746148e+00     & 1.998747e+00     & 1.916788e+00 \\
    7.920319e+00      & 3.462426e-01     & 6.145918e-01     & 5.514979e-01     & 2.372908e-01     & 4.211989e-01     & 3.779587e-01     & 1.459149e+00     & 1.805391e+00     & 2.073741e+00     & 2.010647e+00 \\
    8.085326e+00      & 3.569034e-01     & 6.419720e-01     & 5.992915e-01     & 2.366042e-01     & 4.255866e-01     & 3.972922e-01     & 1.508440e+00     & 1.865344e+00     & 2.150412e+00     & 2.107732e+00 \\
    8.250332e+00      & 3.676169e-01     & 6.703455e-01     & 6.500709e-01     & 2.358954e-01     & 4.301528e-01     & 4.171428e-01     & 1.558389e+00     & 1.926006e+00     & 2.228735e+00     & 2.208460e+00 \\
}\exampleThirtyOneRegionOfAccuracyData

\pgfdeclarelayer{background}
\pgfsetlayers{background,main}
\begin{figure}[htb]
    \centering
    \begin{tikzpicture}[]
        \begin{axis}[
                width = 0.65\columnwidth, height = 5cm,
                grid,
                xmin=0, xmax=8.5,
                ymin=1e-6, ymax=0.75,
                xlabel={Initial condition norm, $\vert\vert \bx_0 \vert\vert$},
                ylabel={Cost gap,\\ $J\left(\bx_0,\cdot\right) - J\left(\bx_0,u_{\text{opt}}\right)$},
                clip mode=individual, 
                legend pos=north west,
                legend style={draw=none,cells={align=center},font=\scriptsize\sffamily},
                xlabel style={font=\small\sffamily},
                ylabel style={font=\small\sffamily,align=center},
                ylabel shift = -3 pt,
            ]
            \addplot
            [color=Set1-A, thick, dashed, dash pattern=on 4pt off 1pt on 1pt off 1pt]
            table [x index=0, y index=1, col sep=&] {\exampleThirtyOneRegionOfAccuracyData};
            \addlegendentry{$u_1(t)$}
            \addplot
            [color=Set1-E, thick, dashed, dash pattern=on 4pt off 1pt on 1pt off 1pt on 1pt off 1pt]
            table [x index=0, y index=2, col sep=&] {\exampleThirtyOneRegionOfAccuracyData};
            \addlegendentry{$u_2(t)$}
            \addplot
            [color=Set1-D, thick, dashed, dash pattern=on 4pt off 1pt on 1pt off 1pt on 1pt off 1pt on 1pt off 1pt]
            table [x index=0, y index=3, col sep=&] {\exampleThirtyOneRegionOfAccuracyData};
            \addlegendentry{$u_3(t)$}
        \end{axis}
    \end{tikzpicture}
    \caption{1D Fokker-Planck equation. Growth of sub-optimality of the PPR controllers as initial condition distance to the origin increases. The kink in the $u_2(t)$ curve corresponds to a point where the control response to the initial condition inverts. The polynomial nature of the control laws causes higher-order variants to accumulate error at a faster rate further from the origin.}
    \label{fig:example31_regionOfAccuracy}
\end{figure}
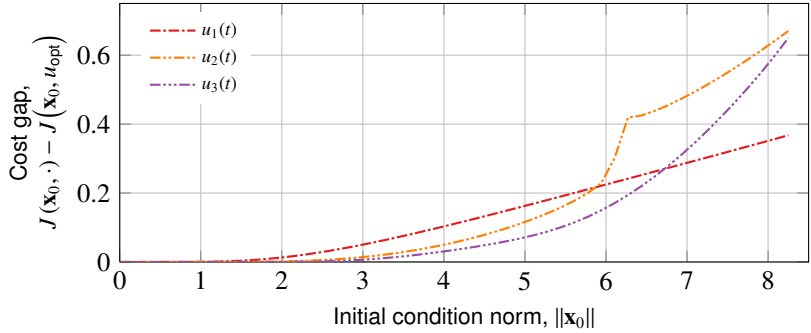

The results of the convergence plot are consistent with what is expected of a Taylor approximation.
Near the origin, all of the controllers perform well.
For slightly further initial conditions, the linear feedback law begins to show sub-optimal performance, and adding additional polynomial degrees brings the feedback law closer to the true optimal solution.
However, the reponse of the polynomial controllers grows quickly outside of the local region near the equilibrium point once they begin to dominate.
For sufficiently large initial conditions, we exit the region of convergence of the Taylor approximation, and the polynomial behavior of the higher-order control laws causes them to diverge even more quickly than the linear feedback law.
This confirms that the proposed method, despite offering some improvements over linear feedback, is still primarily a local method, and outside of a local region, it is not even guaranteed that the polynomial feedback laws are stabilizing.

Finally, we perform convergence and scalability studies with respect to the number of discretization points $n$ in the Fokker-Planck mesh.
In \cref{sfig:example31_convergence_n}, we show the convergence of the value function at an initial condition $\bx_0$ as $n$ is increased, and in \cref{sfig:example31_cpuScaling_n}, we show the CPU time scaling for both the PPR method and the method of \cite{Breiten2018}.
In \cref{sfig:example31_convergence_n}, we see that, as expected, refining the mesh leads to convergence of the value function evaluation.
For this problem, the mesh does not converge until $n \approx 1000$, and for this initial condition, at least a degree 4 approximation would be required.
In \cref{sfig:example31_cpuScaling_n}, the scaling of both methods is shown to be approximately $O(n^{d+1})$, which is consistent with the complexity analysis in \cref{ch3:sec:complexity}.
For both the degree $d=3$ and $d=4$ cases, the Bartels-Stewart solver used in PPR is shown to be faster than the quadrature-based solver used in \cite{Breiten2018}.
For the $n=1023$ degree 3 case, the method in \cite{Breiten2018} took roughly 54 minutes to compute, whereas the PPR method completed in about 9 minutes.
The exception occurs at small $n$, where the quadrature-based solver appears to have lower overhead and is able to extend to lower computational times than the Bartels-Stewart solver.
Since the value function approximations computed by both methods are identical, either solver could be used, but in this case the Bartels-Stewart solver is generally preferred.

\pgfplotstableread[col sep=&, row sep=\\]{
    n &        n^3 &   CPU-sec (PPR)   &   CPU-sec (Breiten)   &      V(x0)    \\
    15 &       3375 &    2.72373e-02    &    6.07914e-03    &    2.75018e-02    \\
    31 &      29791 &    3.21858e-02    &    2.93176e-02    &    2.92917e-02    \\
    63 &     250047 &    8.66296e-02    &    3.35754e-01    &    3.10647e-02    \\
    127 &    2048383 &    4.91268e-01    &    2.11790e+00    &    3.23060e-02    \\
    255 &   16581375 &    4.90645e+00    &    2.45793e+01    &    3.29795e-02    \\
    511 &  133432831 &    3.92980e+01    &    2.45743e+02    &    3.33338e-02    \\
    1023 & 1070599167 &    5.21083e+02    &    3.23936e+03    &    3.35151e-02    \\
}\exampleThirtyOneConvergenceDatadthreeData
\pgfplotstableread[col sep=&, row sep=\\]{
    n &        n^4 &   CPU-sec (PPR)   &   CPU-sec (Breiten)   &      V(x0)      \\
    15 &      50625 &    1.23061e-01    &    7.15473e-02        &    2.71064e-02   \\
    31 &     923521 &    2.30783e-01    &    1.19097e+00        &    2.87038e-02   \\
    63 &   15752961 &    3.16591e+00    &    1.73396e+01        &    3.03023e-02   \\
    127 &  260144641 &    5.82740e+01    &    3.09185e+02        &    3.14139e-02   \\
}\exampleThirtyOneConvergenceDatadfourData

\begin{figure}[htb]
    \centering
    \begin{subfigure}[h]{0.45\textwidth}
        \centering
        \begin{tikzpicture}
            \begin{semilogxaxis}[xlabel=$n$,
                    ylabel=$V(\mathbf{x}_0)$,
                    width=\columnwidth, height=4.5cm,
                    xmin=1e1, xmax=1.2e3,
                    ylabel near ticks,
                    legend pos=north west,
                    legend style={draw=none,font=\scriptsize\sffamily},
                    xlabel style={font=\small\sffamily},
                    ylabel style={font=\small\sffamily},
                ]

                \addplot[color=Paired-C,mark=*,mark size=2] table[x={n}, y={V(x0) },col sep=&] {\exampleThirtyOneConvergenceDatadthreeData};
                \addlegendentry{degree 3}

                \addplot[color=Paired-G,mark=star,mark size=2] table[x={n}, y={V(x0)},col sep=&] {\exampleThirtyOneConvergenceDatadfourData};
                \addlegendentry{degree 4}
            \end{semilogxaxis}
        \end{tikzpicture}
        \caption{Solution convergence.}
        \label{sfig:example31_convergence_n}
    \end{subfigure}
    \hfill
    \begin{subfigure}[h]{0.525\textwidth}
        \centering
        \begin{tikzpicture}
            \begin{loglogaxis}[xlabel=$n$,
                    ylabel=CPU sec,
                    width=\columnwidth, height=4.5cm,
                    xmin=5, xmax=1.2e3,
                    ymin=1e-3, ymax=1e4,
                    legend pos=north west,
                    legend style={draw=none,font=\scriptsize\sffamily},
                    xlabel style={font=\small\sffamily},
                    ylabel style={font=\small\sffamily},
                ]
                \addplot[domain=5:2e3, samples=10, color=lightgray, thick, forget plot] {6*10^(-8) * x^3} node [pos=.8, fill=white, sloped] {$n^{3}$};
                \addplot[domain=5:6e2, samples=10, color=lightgray, thick, forget plot] {6*10^(-8) * x^4} node [pos=.8, fill=white, sloped] {$n^{4}$};
                \addplot[domain=5:6e2, samples=10, color=lightgray, thick, forget plot] {10^(-7) * x^5} node [pos=.6, fill=white, sloped] {$n^{5}$};

                \addplot[color=Paired-A,mark=o,mark size=2] table[x={n}, y={CPU-sec (PPR)},col sep=&] {\exampleThirtyOneConvergenceDatadthreeData};
                \addlegendentry{$V_3(\bx)$ PPR}

                \addplot[color=Paired-B,mark=*,mark size=2, dashed] table[x={n}, y={CPU-sec (Breiten)},col sep=&] {\exampleThirtyOneConvergenceDatadthreeData};
                \addlegendentry{$V_3(\bx)$ Breiten}

                \addplot[color=Paired-I,mark=square,mark size=2] table[x={n}, y={CPU-sec (PPR)},col sep=&] {\exampleThirtyOneConvergenceDatadfourData};
                \addlegendentry{$V_4(\bx)$ PPR}

                \addplot[color=Paired-J,mark=square*,mark size=2, dashed] table[x={n}, y={CPU-sec (Breiten)},col sep=&] {\exampleThirtyOneConvergenceDatadfourData};
                \addlegendentry{$V_4(\bx)$ Breiten}
            \end{loglogaxis}
        \end{tikzpicture}
        \caption{Computational scalability.}
        \label{sfig:example31_cpuScaling_n}
    \end{subfigure}
    \caption{Convergence of the value function (left) and CPU time scaling (right) w.r.t. $n$ as the Fokker-Planck mesh is refined.}
    \label{fig:example31}
\end{figure}
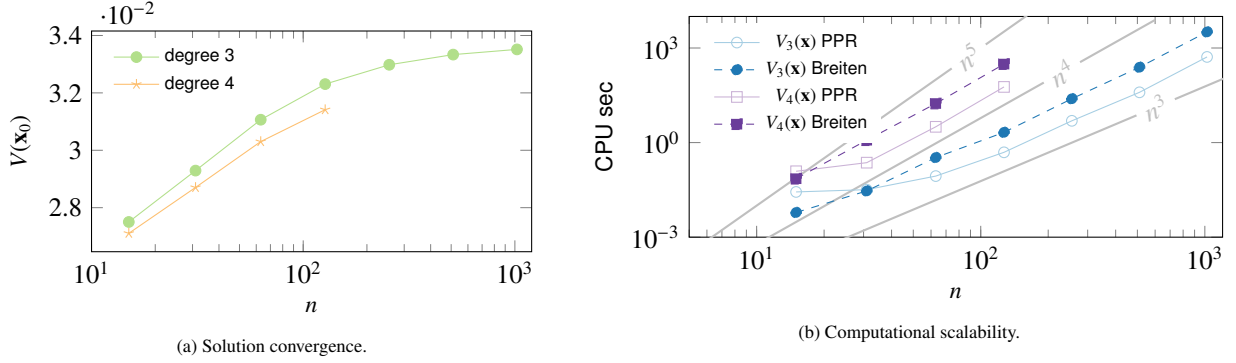

\section{High-dimensional semi-discretized Allen-Cahn equation}\label{ch3:sec:allen-cahn}
The Allen-Cahn equation \cite{Allen1979}, which models phase separation in multiphase materials, such as those used in powder-bed fusion additive manufacturing \cite{Yang2019,Li2018,Ali2021,Darabi2023},
is a common benchmark for evaluating HJB-based nonlinear control techniques \cite{Alla2020,Dolgov2021}.
The model takes the form of a nonlinear heat equation with a cubic nonlinearity.
Also known as a bistable equation, it features three constant steady-state solutions, typically at 1, 0, and -1, with the solutions at 1 and -1 being stable and 0 being unstable.
In \cref{ch3:sec:allen-cahn-1}, we compare the performance of the proposed Polynomial-Polynomial Regulator (PPR) controllers with several state-of-the-art methods, including LQR, State-Dependent Riccati Equation (SDRE), and Tensor-Train HJB (TT-HJB), for the Allen-Cahn equation on a one-dimensional spatial domain.
To compare with the existing literature \cite{Dolgov2021}, we take the same Chebychev pseudospectral discretization with $n=14$.
Then in \cref{ch3:sec:allen-cahn-2D}, we consider finite element discretization of the Allen-Cahn equation in two spatial dimensions, leading to high-dimensional but sparse discretizations on the order of $n=O(100{,}000)$ that illustrate the scalability of Al'brekht's method when paired with the proposed dimensionality reduction technique.

\subsection{Distributed control of the 1D Allen-Cahn equation}\label{ch3:sec:allen-cahn-1}
\noindent \paragraph{Model description} The control problem defined in \citep{Dolgov2021} concerns the Allen-Cahn equation
\begin{align}
    \frac{\partial w(z,t)}{\partial t} & = \epsilon \frac{\partial^2 w(z,t)}{\partial z^2} + w(z,t) - w(z,t)^3
    + \chi_\omega(z) u(t)
\end{align}
for $t > 0$ and $z \in \Omega = [-1,1]$, subject to the Neumann boundary conditions $\left. \partial w(z,t)/\partial z\right\vert_{z=-1} = 0$ and $\left. \partial w(z,t)/\partial z\right\vert_{z=1} = 0$
and initial condition $w(z,0) = w_0(z)$.
A single input $u(t)$ acts on the subdomain $\omega = \left[-0.5, 0.2\right]$ via the characteristic function $\chi_\omega(z)$, as depicted in \cref{ch3:fig:cheb-nodes-0}.
\begin{figure}[htb]
    \centering
    \includegraphics[page=5,width=0.9\textwidth]{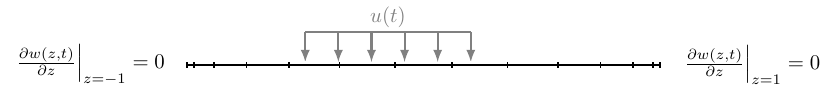}
    \caption{1D Allen-Cahn equation. Spatial domain discretized with 14 Chebychev nodes, which are denser near the boundaries.}
    \label{ch3:fig:cheb-nodes-0}
\end{figure}

This system has three constant steady-state solutions:
$w(z,t) = -1$ and $w(z,t) = 1$, which are stable, and
$w(z,t) = 0$, which is unstable.
The objective defined in \citep{Dolgov2021} is to stabilize the unstable equilibrium $w(z,t) = 0$ subject to the performance index
\begin{align}
    \cJ(w_0,u) = \frac{1}{2}  \int_0^\infty \left\Vert w(z,t)\right\Vert^2_{L_2(-1,1)} + \gamma u(t)^2 \,\rd t.
\end{align}
\noindent \paragraph{Discretization} Using a Chebychev pseudospectral spatial discretization of dimension $n=14$ as in \cite{Dolgov2021}, the model takes the form
\begin{align}
    \dot{\bx}(t) & = \bA \bx(t) + \bF_3 \kronF{\bx(t)}{3} + \bB u(t), \label{ch3:eq:allen-cahn-1}
\end{align}
and the associated discretized performance index is
\begin{align}
    J(\bx_0, \bu) = \frac{1}{2} \int_0^\infty \bx(t)^\top \bQ \bx(t)  +  \gamma u(t)^2 \,\rd t, \label{ch3:eq:allen-cahn-neumann-cost}
\end{align}
where $\bQ$ is a diagonal weighting matrix arising from the Chebychev discretization
and $\gamma = 0.1$.
The values for $\bQ$ and $\gamma$ are taken from \cite{Dolgov2021}.

In \citep{Dolgov2021}, the chosen diffusion coefficient is $\epsilon = 0.2$.
However, the closed-loop simulation code supplied with \citep{Dolgov2021} features a typo where the input matrix $\bB$ is missing for the TT-HJB simulation; as a result, the control input for TT-HJB was applied over the entire domain rather than the subdomain $\chi_\omega(z)$ depicted in \cref{ch3:fig:cheb-nodes-0}.
Upon applying the control only to that subdomain, the method presented therein fails to stabilize the system for $\epsilon = 0.2$.
We have adopted a higher diffusion coefficient $\epsilon = 0.5$ in this work.
We construct an initial condition based on that used in \citep{Dolgov2021} as
\begin{align*}
    w_0(z)
    =\cos(2\pi z)\cos(\pi z)  + w_\text{offset},
\end{align*}
with $w_\text{offset}=0.5,1,1.5,$ or $2$.
The offset parameter allows us to move the initial condition closer or further from the equilibrium at the origin.
The initial condition $\bx_0$ of the finite-dimensional dynamics is then given by the discretization of the PDE initial condition $w_0(z)$.

\noindent \paragraph{Control design} We solve the optimal control problem with performance index \labelcref{ch3:eq:allen-cahn-neumann-cost} for system \labelcref{ch3:eq:allen-cahn-1} and compare with three prevalent methods: the Linear-Quadratic Regulator (LQR), State-Dependent Riccati Equation (SDRE), and Tensor-Train HJB (TT-HJB) of \citep{Dolgov2021}.
Then we compute four variants of our PPR controller.
The first is a plain PPR controller of degree 5, corresponding to a degree 6 value function approximation, with no dimensionality reduction applied.
The degree 6 value function coefficient requires solving a linear system with over $14^6 \approx 7.5$ million unknowns, so our efficient implementation is key for enabling the computation of this type of control law.
The PPR controller of degree 3 was found to not stabilize initial conditions with offsets larger than 1, likely because the associated degree 4 value function approximation becomes negative.
The second controller is a degree~5 PPR controller featuring the acceleration due to dimensionality reduction proposed in \cref{ch3:sec:reduced-Albrekht}, illustrating that close to the full-order performance can be retained with significant speed-up in computations.
The third controller is a ``tuned'' degree 3 PPR controller, in which we manipulate the higher-order cost parameters $\bq_p$ and $\br_p$ in \cref{ch3:eq:performanceIndex-Poly} in order to improve the performance of the controllers far from the origin, i.e. for larger offsets $w_\text{offset}$.
For the fourth and final controller,
dimensionality reduction is applied to the tuned degree 3 controller, again illustrating that our acceleration step can boost computational performance while retaining accuracy.
For the latter two controllers, the artificial cost for computing the controllers is chosen by trial and error as
\begin{align*}
    J_\text{tuned}(\bx_0, \bu) = \frac{1}{2} \int_0^\infty \bx(t)^\top \bQ \bx(t) + \gamma u(t)^2 + 0.25 \sum_{i=1}^n x_i(t)^4 + 0.009 \sum_{i=1}^n x_i(t) u(t)^2\,\rd t.
\end{align*}
The additional quartic terms in the state penalty are introduced in order to manipulate and make the quartic term in the value function positive.
The additional control penalty term with linear state dependence is introduced to moderate the control action: it was observed that while the standard PPR controller performs better than LQR locally, it overcompensates for initial conditions further from the origin, resulting in larger control costs.
Hence control penalty terms of the form $x_i(t) u(t)^2$, which increase as the state deviates from the origin, are added to lower the control input applied by the tuned PPR controller far from the origin.
We emphasize that this artificial cost is only used to provide the $\bq_p$ and $\br_p$ parameters as inputs to the PPR algorithm in \cref{ch3:alg:full-PPR};
the performance for all of the controllers is evaluated with the same cost given in \cref{ch3:eq:allen-cahn-neumann-cost}.
These extra parameters can be interpreted as additional design parameters to tune the behavior of the closed-loop system.
It is standard practice in control design to choose the cost function parameters to yield the desired performance.

We briefly remark that
since the value function should be a positive semidefinite function, and since odd polynomials are guaranteed to become negative,
we generally compute even-degree value function approximations corresponding to odd-degree feedback law approximations.
Hence we tend to compute linear, cubic, quintic, etc. feedback laws, although
in principle nothing prohibits the computation of even-degree feedback laws and odd-degree value function approximations.
Even-degree polynomials can (and often do) also become negative, but at least they are not guaranteed to do so.

\noindent \paragraph{Results} The integrated control cost results are presented in \cref{ch3:tab:valueFunApproxs-ac1}.
For all of the controllers except the TT-HJB controller, the cost integrals converge in about 10 seconds; we extended the simulations up to $T=10^7$ seconds to verify convergence.
Even with the larger value of the diffusion coefficient $\epsilon = 0.5$, we found that the controller computed using the TT-HJB method fails to fully stabilize the solution of interest.
Therefore, the cost function integral fails to converge to a finite value, and continuing to integrate for longer horizons results in a cost that continues to increase, as shown in \cref{ch3:fig:example9_runningCosts}.
Next, we observe the performance of the LQR and SDRE methods.
These both stabilize the equilibrium, and by re-linearizing and recomputing the solution to the Riccati equation at each time step, the SDRE approach achieves even lower control costs at the expense of increased computational cost.
Disregarding its high online computational cost, the SDRE approach produces the best control actions for every initial condition considered for this example.

\begin{table}[htb]
    \centering
    \caption{1D Allen-Cahn equation. Integrated control costs computed up to $T =  10^7$ for different initial condition offsets.}
    \begin{tabular}{lcccc}
        \toprule
        Controller                     & $w_\text{offset}=0.5$ & $w_\text{offset}=1.0$ & $w_\text{offset}=1.5$ & $w_\text{offset}=2.0$ \\
        \midrule
        Uncontrolled                   & $2\cdot 10^7$         & $2\cdot 10^7$         & $2\cdot 10^7$         & $2\cdot 10^7$         \\[1pt]
        TT-HJB                         & 422.806               & 424.408               & 425.522               & 426.270               \\[1pt]
        LQR                            & 0.678                 & 2.137                 & 3.787                 & 5.400                 \\[1pt]
        SDRE                           & 0.667                 & 1.928                 & 2.995                 & 3.741                 \\[1pt]
        Degree 5 PPR                   & 0.667                 & 2.026                 & 6.569                 & 35.376                \\[1pt]
        Degree 5 PPR, reduced          & 0.667                 & 2.018                 & 6.720                 & 37.394                \\[2pt]
        Degree 3 PPR, tuned            & 0.668                 & 1.932                 & 3.010                 & 3.806                 \\[2pt]
        Degree 3 PPR, tuned \& reduced & 0.667                 & 1.942                 & 3.114                 & 4.645                 \\
        \bottomrule
    \end{tabular}
    \label{ch3:tab:valueFunApproxs-ac1}
\end{table}

\begin{figure}[htb]
    \centering
    \includegraphics[width=0.65\textwidth]{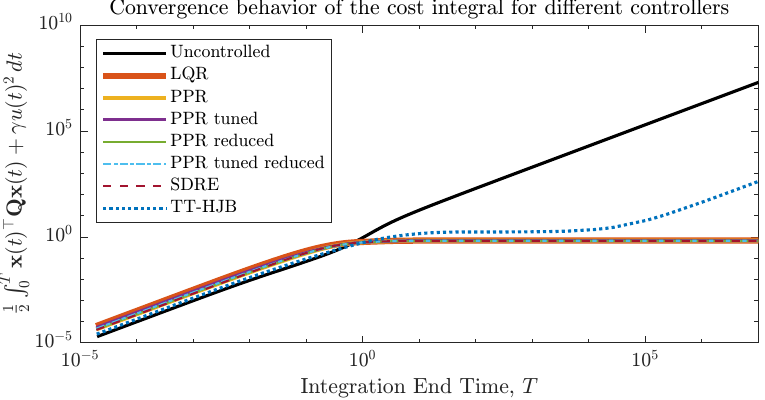}
    \caption{1D Allen-Cahn equation. Finite-time approximations of the cost integral $J(\bx_0, \bu) \approx \frac{1}{2} \int_0^T \bx(t)^\top \bQ \bx(t)  +  \gamma u(t)^2 \,\rd t$ as the simulation time $T$ increases for simulations with the initial condition closest to the origin, corresponding to $w_\text{offset} = 0.5$. For the infinite horizon stabilization problem, a successful control law should drive the cost integral to a finite value, with the optimal controller having the smallest cost. The LQR, SDRE, and PPR controllers are all able to asymptotically stabilize the equilibrium and achieve a finite control cost that converges to a finite value in about 10 seconds, whereas the TT-HJB diverges and is not able to asymptotically stabilize the equilibrium.}
    \label{ch3:fig:example9_runningCosts}
\end{figure}

Continuing with the analysis of the control costs presented in \cref{ch3:tab:valueFunApproxs-ac1},
the first of the PPR controllers presented is a full-order PPR controller of degree 5 (i.e. a polynomial controller featuring terms up to degree 5).
We see that the controller performs slightly better than the LQR controller for initial conditions close to the equilibrium.
This comports with the theory underlying the method, as it is guaranteed that adding higher-order terms to the Taylor approximation of the optimal feedback law will improve the accuracy, and therefore optimality, locally near the equilibrium point.
However, as we depart from the equilibrium point and consider initial conditions further away, we expect the accuracy to degrade, and if we depart the region of convergence of the Taylor approximation, we can even expect the polynomial approximation to be worse than the LQR solution.
For an initial condition with $w_\text{offset}=1.5$, the degree 5 PPR controller performs worse than LQR by a factor of 2, and extending the initial condition slightly further to $w_\text{offset}=2.0$, the performance rapidly degrades further to a factor of 7.
The next PPR controller presented is a reduced-order approximation of the degree 5 controller.
The first term in the feedback law is still the full-order LQR gain, but the higher-order terms are approximated in a reduced dimension.
By trial and error, the reduced-order dimension is chosen to be $r=3$ to give a satisfactory balance between performance and computational cost.
This is sufficient to obtain similar performance to the full-order solution computed in $n=14$ dimensions for this problem.

The final two controllers are designed using additional higher-order penalties in the cost function to improve the controller performance further from the equilibrium point at the expense of local optimality.
This is an additional design freedom that is not awarded to the other methods; the comparison here is included to illustrate that this design freedom is readily available in the PPR method, whereas other methods may or may not not easily permit such modification of the nonlinear behavior of the controllers.
The tuned PPR controllers presented are of degree 3, and whereas the true degree 3 Taylor expansion of the optimal feedback law is not stabilizing for all of the initial conditions of interest,
the additional penalties in the cost function produce controllers that are.
Consequently, it is important to note that these modified controllers are no longer Taylor approximants of the optimal feedback law.
With the additional design flexibility afforded by the higher-order cost terms,
we are able to
reduce the control costs for $w_\text{offset}=1.5$ and $w_\text{offset}=2.0$ and improve the controller performance.
The resulting controllers are superior to the LQR controller, but still not as good as the SDRE controller.

In \cref{ch3:fig:example9_neumann_v02}, the first 5 seconds of the closed-loop simulations are shown for each of the controllers starting from the initial condition corresponding to $w_\text{offset}=2.0$.
In the simulations, it is clear that the degree 5 PPR and degree 5 PPR reduced controllers both act too aggressively, with a major dip in the response that drives it below 0, leading to their poor performance costs.
Meanwhile, the tuned PPR controllers behave more closely to the SDRE controller, which offers the best performance of those tested.
In the TT-HJB simulation plot, the higher residual state error due to the solution failing to converge to the desired steady-state quickly enough is also visible.
This residual error persists even over long time horizons, leading to the continued accumulation of the cost function as seen in \cref{ch3:tab:valueFunApproxs-ac1}.
\begin{figure}[htb]
    \centering
    \includegraphics[width=\textwidth]{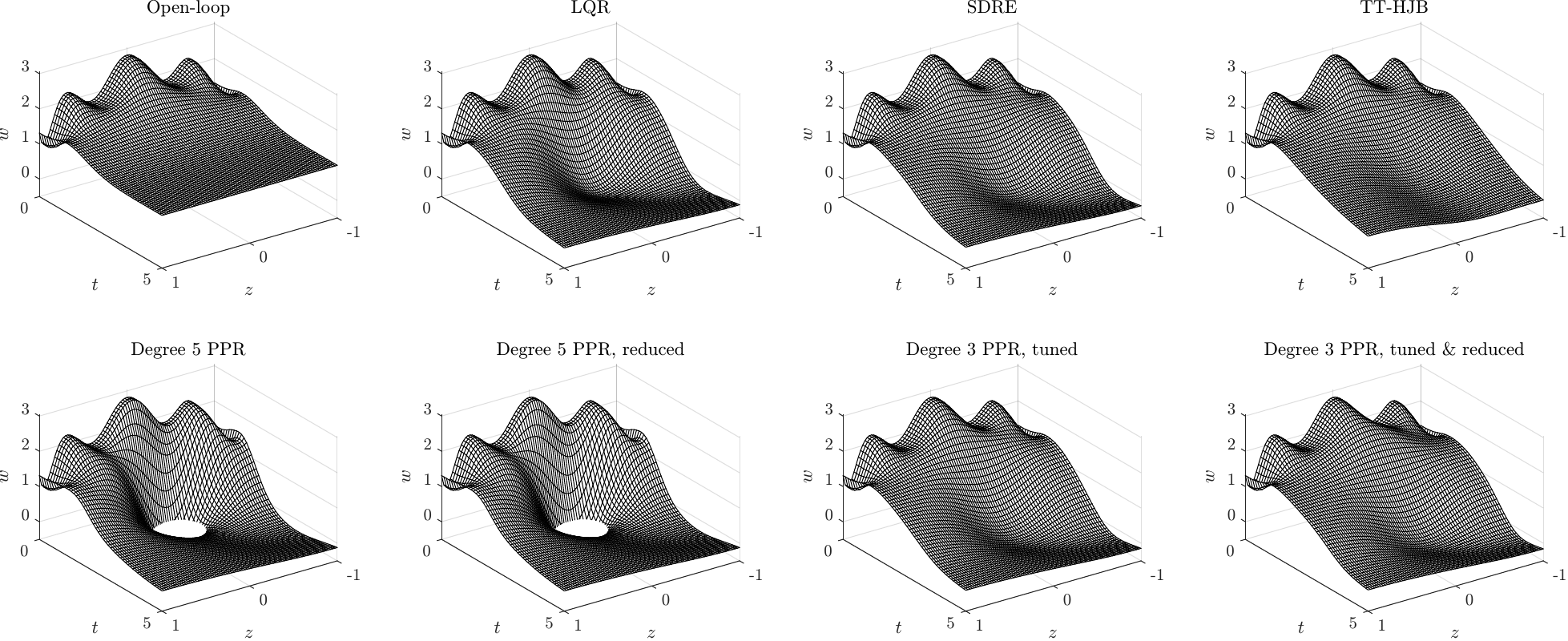}
    \caption{1D Allen-Cahn equation. Simulations for the initial condition corresponding to $w_\text{offset} = 2.0$.}
    \label{ch3:fig:example9_neumann_v02}
\end{figure}

\noindent \paragraph{Comparison of computational and implementation complexity}
We discuss the complexity of the various approaches from a computational and implementation standpoint.
In \cref{ch3:tab:cpu-time-ac1}, we present the offline time required to compute each controller, the time associated with evaluating each controller online, and an estimate of the RAM required to store the quantities necessary for online evaluation of each controller.
\textsc{Matlab}'s \texttt{lqr()} represents the standard, as it is straightforward to use and can be computed offline for the $n=14$ dimensional linearization in less than 0.01 seconds.
The online implementation of the LQR controller only requires storing the gain matrix of dimension $n \times m$, and online evaluation is performed very efficiently via matrix multiplication.
While the SDRE controller provides superior performance, it cannot be computed offline, instead requiring the solution of a Riccati equation at every time step.
This adds considerable online computational cost, although the online memory requirements are not prohibitively high.
While there are more sophisticated SDRE-based approaches which move some of the computations into the offline phase \cite{Alla2023,Heiland2023}, the implementation becomes more complex, and there does not appear to be a widely available package that gives a simple, scalable SDRE function for general purpose use.

The TT-HJB controller can be computed via the open-source repository provided in \cite{Dolgov2023}.
The computations behind the controller are complex, but the authors do provide high-level functions for computing and evaluating the control law.
Nonetheless, there are many initialization parameters, stopping criteria, and shifting parameters that need to be carefully chosen.
The control law is computed in an iterative fashion, and the code reaches its stopping criteria in about 36 seconds, though it does not appear to converge properly to a stabilizing controller (see \cref{ch3:fig:example9_runningCosts}).
Online evaluation of the controller is moderately expensive, though not nearly as expensive as the SDRE controller, and the memory required to store the controller is much higher than LQR or SDRE.

We have aimed to make the PPR controllers easy to compute and implement for general purpose use, with the \texttt{ppr()} function usage mirroring \textsc{Matlab}'s \texttt{lqr()} syntax, allowing the user to avoid the tedious computations behind the method.
The most challenging step is to derive the Kronecker-product form for the dynamics in \cref{ch3:eq:FOM-Poly}, which is admittedly not a commonly used form.
We provide a utility function \texttt{approxPolynomialDynamics()} that permits the user to supply a symbolic expression for analytic dynamics in the form of $\bf(\bx)$ and $\bg(\bx)$ to obtain the Kronecker coefficients for the dynamics.
We reiterate that this is simply a notational form for polynomials, so any model featuring analytic nonlinearities, or indeed any model that is already in polynomial form, can be written in this way.
For larger models though, where symbolic computations are inefficient, or for semi-discretized models, some additional effort is required to derive the explicit Kronecker coefficients.
The computation of the PPR controllers is fast: for this problem, the full-order degree~5 PPR controller is computed in less than 4 seconds, and the full-order tuned degree 3 PPR controller is computed in less than 0.1 seconds.
The reduced PPR controllers computed using dimensionality reduction are computed in less than 0.01 seconds.
Importantly, since the computations are not iterative, no initialization or stopping parameters are required.
The online computational time and storage requirements for the PPR controllers without reduction scale with the polynomial degree.
Thus, the degree~5 PPR controller requires the most memory for this problem at over 4 MB, and the controller evaluation time is almost as high as the SDRE method, due to the number of operations required to evaluate the degree 5 multivariate polynomial online.
By tuning the controller cost parameters though, the degree~3 tuned PPR controller can be made to stabilize the system with the desired performance and with online requirements more in line with the LQR controller.

\begin{table}[htb]
    \centering
    \caption{Comparison of the complexity of computing, evaluating, and storing each control law.}
    \begin{tabular}{lccc}
        \toprule
        Controller                     & \makecell{Controller                 \\ Computation Time (s)} & \makecell{Controller Evaluation Time            \\(s per $10{,}000$ evaluations)} & \makecell{Approximate Online \\ RAM Requirement (KB)} \\
        \midrule
        LQR                            & $< 0.01$             & 0.087  & 0.1  \\[1pt]
        SDRE                           & ---                  & 11.616 & 1.7  \\[1pt]
        TT-HJB                         & $36$                 & 4.271  & 166  \\[1pt]
        Degree 5 PPR                   & $4$                  & 10.694 & 4634 \\[1pt]
        Degree 5 PPR, reduced          & $< 0.01$             & 0.139  & 3.8  \\[2pt]
        Degree 3 PPR, tuned            & $< 0.1$              & 0.180  & 24   \\[2pt]
        Degree 3 PPR, tuned \& reduced & $< 0.01$             & 0.107  & 1.1  \\
        \bottomrule
    \end{tabular}
    \label{ch3:tab:cpu-time-ac1}
\end{table}

\subsection{Boundary control of the 2D Allen-Cahn equation}\label{ch3:sec:allen-cahn-2D}
\noindent \paragraph{Model description} We illustrate the scalability of the proposed approach on the 2D version of the Allen-Cahn problem.
The governing PDE and boundary conditions are
\begin{equation}
    \begin{split}
        \frac{\partial w(z_1,z_2,t)}{\partial t} & = \epsilon \left(\frac{\partial^2 w(z_1,z_2,t)}{\partial z_1^2} + \frac{\partial^2 w(z_1,z_2,t)}{\partial z_2^2}\right)  + w(z_1,z_2,t) - w(z_1,z_2,t)^3, \\
        \partial_{z_1} w(0,z_2,t)                & =\partial_{z_1} w(1,z_2,t)=\partial_{z_2} w(z_1,1,t)=0,                                                                                                   \\
        \partial_{z_2} w(z_1,0,t)                & =\sin(\pi z_1) u(t),
    \end{split}
\end{equation}
where $w:\real^2 \times \real^{+} \to \real$ is the dependent variable,
$t > 0$, and the square spatial domain is parametrized by $\bz \in \Omega = [0,1]^2$.
We select $\epsilon = 0.25$ as the diffusion coefficient.
The model is subject to homogeneous Neumann boundary conditions
on three edges,
and the control input is applied as a Neumann boundary control on the remaining edge.
Just as in the 1D case,
the system has two stable constant steady-state solutions at
$w = -1$ and $w = 1$, and an unstable constant steady-state solution $w = 0$ which we wish to stabilize.

\noindent \paragraph{Discretization} The spatial domain is discretized using bilinear rectangular finite elements \cite{Ragab2017,Reddy2004}.
Each edge is discretized with
320 elements,
leading to a model with $n=103{,}041$ degrees of freedom that can be written as
\begin{align}
    \bE \dot{\bx}(t) & = \bA \bx(t) + \bF_3 \kronF{\bx(t)}{3} + \bB u(t). \label{ch3:eq:allen-cahn-2}
\end{align}
Unlike the previous models, this model features a non-identity mass matrix $\bE$ due to the finite element discretization.
The mass matrix is invertible, but rather than multiplying by the inverse to put it in standard form, we deal with the model directly in the generalized form to exploit sparsity.
We opt for a simple quadratic performance index
\begin{align}
    J(\bx_0, \bu) = \frac{1}{2} \int_0^\infty \bx(t)^\top \bQ \bx(t)  +  \gamma u(t)^2 \,\rd t. \label{ch3:eq:allen-cahn-2-neumann-cost}
\end{align}
For such large state dimensions,
other HJB solvers are infeasible: standard grid-based techniques are applicable below $O(10)$ dimensions, and even the TT-HJB approach was only demonstrated up to $n=O(100)$ state dimensions.
Since the standard SDRE approach requires the solution of the Riccati equation at each time step, it also is prohibitively expensive for such large state dimensions.
There exist adaptations of the SDRE approach to speed up the online computation of the Riccati solutions \cite{Alla2023,Heiland2023}, but these introduce additional approximations and are beyond the scope of the present work, which is concerned with the HJB approach rather than the SDRE approach.
Even computing an LQR controller can be a significant challenge in such high dimensions without additional assumptions;
we therefore choose a low-rank state penalty $\bQ =  \bC^\top \bC$ and $\gamma=0.1$, where $\bC$ is a binary matrix that selects 10 nodes throughout the domain whose state deviations will be penalized, as shown in
\cref{ch3:fig:example29_domain}.
The low-rank state penalty permits using the LR-ADI solver \cite{Saak2022}, for which we use all default options,
to compute the Riccati solution and the linear feedback gain in Step 1 of \cref{ch3:alg:lr-PPR} in the full dimension $n=103{,}041$.
To compute approximations to the higher-order terms in the feedback law, we use the $r=10$ dominant eigenvectors in the computed Riccati solution as a reduction basis.
These are computed as the left singular vectors of the low-rank factors computed using the LR-ADI solver, as in Step~2 of \cref{ch3:alg:lr-PPR}.
Subsequently, we continue to compute a degree 4 value function approximation and corresponding cubic feedback law as in \cref{ch3:eq:vi-coeffs-reduced,ch3:eq:Ki-coeffs-reduced}.
Specifically, we compute $\bv_3^r$, $\bv_4^r$, $\bK_2^r$, and $\bK_3^r$ with the PPR solver using reduced coefficient matrices as in \cref{ch3:eq:ROM-operators} in Step 3 of \cref{ch3:alg:lr-PPR}.

\begin{figure}[htb]
    \centering
    \includegraphics[width=0.25\textwidth]{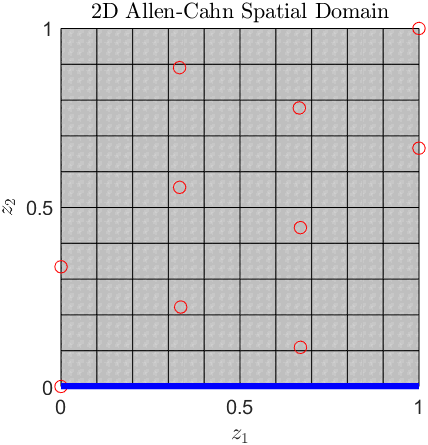}
    \caption{2D Allen-Cahn equation. FEM domain with a very fine mesh, where the grid lines are darkened every 32 nodes. The red circled nodes are those selected to be penalized in the state penalty portion of the cost function. The bottom boundary that is highlighted in blue is where the control input is applied. The model has $n=103{,}041$ degrees of freedom in total.}
    \label{ch3:fig:example29_domain}
\end{figure}

\noindent \paragraph{Results} The open- and closed-loop responses are shown in \cref{ch3:fig:example29} for an initial condition with large jumps, for which we see that the cubic controller successfully stabilizes the steady-state solution at $w=0$.

\begin{figure}[htb!]
    \centering
    \begin{subfigure}[h]{\textwidth}
        \centering
        \includegraphics[width=0.9\textwidth]{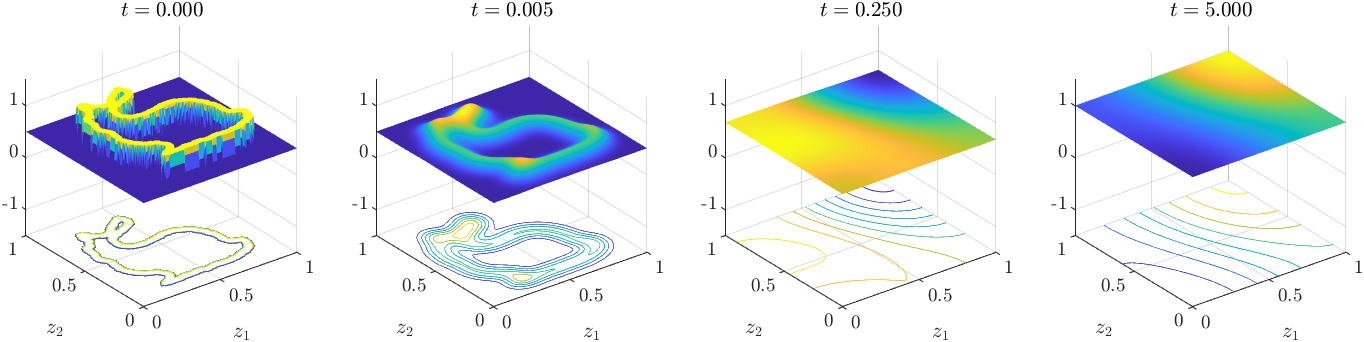}
        \caption{Open-loop.}
    \end{subfigure}

    \begin{subfigure}[h]{\textwidth}
        \centering
        \includegraphics[width=0.9\textwidth]{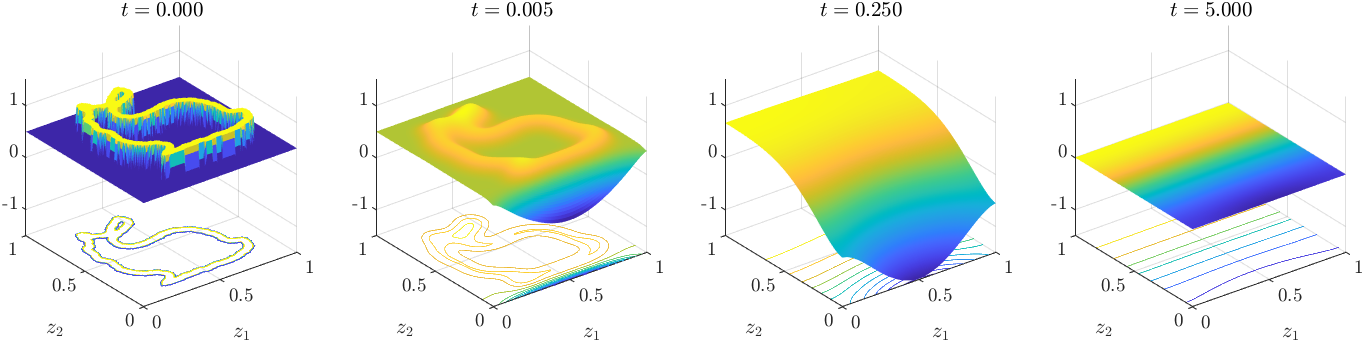}
        \caption{Closed-loop with an approximate cubic PPR controller.}
    \end{subfigure}
    \caption{2D Allen-Cahn equation. Simulation results for a finite element mesh with $n=103{,}041$ degrees of freedom. The uniform zero solution is unstable in open-loop, whereas it is stabilized in closed-loop using the computed PPR controller.}
    \label{ch3:fig:example29}
\end{figure}

While the previous examples focused on the performance of the PPR controllers, the purpose of this example is to demonstrate the scalability of the proposed approach.
Overall, the computation of this cubic PPR controller with our proposed method for a full-order dimension $n=103{,}041$ and reduced-order dimension $r=10$ can be completed in about 5 minutes on a Windows 11 laptop with an AMD Ryzen 9 HX 370 CPU with 32 GB of RAM.
The part of the controller computation that takes the most time is the projection of the model into the reduced dimension, which takes just under 4 minutes.
The second most intensive step is the computation of the Riccati solution, which takes about 45 seconds using the LR-ADI solver.
The computation of this PPR controller also succeeds on a slightly older Windows 10 laptop with an AMD Ryzen 9 5900HS CPU and 16 GB of RAM in about 10 minutes.
Since the Allen-Cahn equation is stiff, the simulation of the closed-loop response is more expensive than actually computing the controller, taking about a 1.5 hours on the Windows 11 laptop using \textsc{Matlab}'s \texttt{ode15s()} integrator.

Refining the mesh even further to a full-order dimension $n=263{,}169$ and retaining the reduced-order dimension $r=10$, the PPR controller can be computed in about 11 minutes on the Windows 11 laptop with the AMD Ryzen 9 HX 370 CPU with 32 GB of RAM.
However, we ran out of RAM during the simulating the closed-loop response of the model, so the simulations had to be performed on a Linux workstation with an Intel Xeon W-3175X CPU and 512 GB of RAM, which took about 8 hours.
The results are similar to those in \cref{ch3:fig:example29}, suggesting that further mesh refinement is unnecessary for this problem.
The closed-loop simulation could possibly be sped up using a more sophisticated integration scheme, but that is beyond the scope of the present work.

In \cref{ch3:tab:RAM}, we summarize the computational time and memory requirements for this example for the PPR controllers with and without the reduction proposed in this work.
Even the computation of an LQR controller is non-trivial for models of dimension $O(100{,}000)$.
The first row shows that theoretical RAM required to store the full Riccati solution $\bV_2$ for the two mesh sizes considered exceeds the typical amount of memory available on standard computer hardware.
Using the LR-ADI solver, the Riccati solutions can be computed in minutes, and the low-rank factors $\bZ$ can be stored using less than a gigabyte of memory.
To a much greater extent, the computation of a full-order PPR controller for such large models would be infeasible on even the highest-performance computer hardware available.
The RAM required to theoretically store the quartic coefficient $\bv_4$ for these high-dimensional models is in the exabyte (1 billion gigabytes) to zettabyte (1 trillion gigabytes) range.
For context, recent estimates place the size of the entire internet at about 100 zettabytes \cite{HealthIT}.
With the proposed reduction, low-dimensional approximations to these coefficients of the form $\bv_4 \approx {\kronF{\bT}{4}}  \bv_4^r$ can be computed in 5 to 10 minutes for a reduced-order dimension based on the dominant $r=10$ eigenvectors.
The reduced coefficients $\bv_4^r$ and the reduction basis $\bT$ together can be stored with less than a gigabyte of memory for both of the meshes considered.

\begin{table}[htb]
    \centering
    \caption{2D Allen-Cahn equation. Computational time and RAM required to store the highest-degree value function coefficient for an LQR controller and a degree 3 PPR controller. The LQR controller, which corresponds to a degree 1 PPR controller, requires computing the quadratic value function coefficient $\bV_2$ as the solution of the GARE. The degree 3 PPR controller requires computing the quartic value function coefficient $\bv_4$ as the solution of a tensor Lyapunov equation. }
    \begin{tabular}{lcccc}
        \toprule \toprule
                                                             & \multicolumn{2}{c}{$n=103{,}041$} & \multicolumn{2}{c}{$n=263{,}169$}                           \\
        \cmidrule(lr){2-3}
        \cmidrule(lr){4-5}
                                                             & CPU time                          & Required RAM                      & CPU time & Required RAM \\
        \midrule
        $\bV_2$                                              & --                                & 85 GB                             & --       & 554 GB       \\
        $\bV_2 \approx \bZ \bZ^\top$                         & 45 sec.                           & 256 MB                            & 3 min.   & 720 MB       \\
        $\bv_4$                                              & --                                & 902 EB                            & --       & 38 ZB        \\
        $\bv_4 \approx {\kronF{\bT}{4}}  \bv_4^r$ for $r=10$ & 5 min.                            & 8 MB                              & 11 min.  & 21 MB        \\
        \bottomrule \bottomrule
    \end{tabular}
    \label{ch3:tab:RAM}
\end{table}

The ability to compute a nonlinear feedback law as an approximate solution to the HJB PDE for models of dimension $O(100{,}000)$ represents a significant step forward from the previous state of the art of $O(10)$ to $O(100)$.
With our open-source and scalable implementation that runs quickly on standard computer hardware, we hope that this will enable the study of nonlinear feedback controllers for very high dimensional systems at the forefront of research applications.


\section{Conclusions}\label{ch3:sec:conclusions}
In this paper, we introduced a computational adaptation to Al'brekht's method for computing power series expansions of the solutions to Hamilton-Jacobi-Bellman PDEs arising in infinite horizon optimal control of control-affine nonlinear systems.
The method involves computing the higher-order nonlinear terms in the feedback law and associated value function in a reduced state dimension, which can be obtained using various model reduction techniques.
Specifically, we propose to use the dominant eigenvectors of the quadratic coefficient in the value function, which is the solution to the usual GARE.
This approach permits taking advantage of the very sophisticated modern solvers for GARE used in the computation of linear feedback laws.
The nonlinear terms are then computed as a perturbation of the linear feedback law, which enables scalability to very high state dimensions on the order of $n=O(100{,}000)$.

Our open-source \textsc{Matlab} implementation permits the design of multi-input controllers for control-affine systems with nonlinear analytic drift and input vector fields.
Furthermore, the cost function can include higher-order state and control penalties, which provide additional flexibility to the control designer.
Finally, our implementation supports models in generalized form with an invertible mass matrix.
When paired with the exploitation of sparsity, this permits control design for high-dimensional nonlinear systems with hundreds of thousands of states, such as those arising from semi-discretized partial differential equations.

While the proposed method scales well and performs better than linear controllers locally,
global performance remains an open challenge.
We see two possible directions regarding this: it is possible that an alternate set of basis functions--perhaps a variant of radial basis functions--could be used to replace the polynomial basis functions that are commonly used.
Another route could be to consider Hamilton-Jacobi-Bellman \textit{inequalities} rather than equality.
This provides a formal mechanism for computing suboptimal solutions that still satisfy the Lyapunov conditions for closed-loop stability, although the challenge becomes how to scalably solve the \textit{inequalities}.

\section*{Acknowledgements}
This work was supported by the Air Force Office of Scientific Research award number FA9550-24-1-0237.

\appendix
\section{Proof of \cref{ch3:thm:kiPoly}}\label{ch3:sec:thm-proof}
Inserting the polynomial forms of $\bf(\bx)$ and $\bg(\bx)$ from~\cref{ch3:eq:FOM-Poly}
into the HJB PDEs~\cref{ch3:eq:HJB-PDE,ch3:eq:HJB-PDE-u} gives
\begin{align}
     & 0 =  \frac{\partial \overline V^\top(\bz)}{\partial \bz} \left[\bA \bx + \sum_{p=2}^\infty \bF_p \kronF{\bx}{p}\right]
    + \frac{\partial \overline V^\top(\bz)}{\partial \bz}\left[\sum_{p=1}^\infty \bG_p \left(\kronF{\bx}{p} \otimes \bI_m \right) + \bB\right] \bu(\bx)  \label{ch3:eq:HJB-PDE-plugged-in}                                                                                   \\
     & \quad + \frac{1}{2} \sum_{p=2}^{\lambda} \bq_p^\top \kronF{\bx}{p}
    + \frac{1}{2} \sum_{p=0}^{\lambda}{\kronF{\bx}{p}}^\top \br_p^\top \kronF{\bu(\bx)}{2}
    \nonumber                                                                                                                                                                                                                                                                \\
     & \sum_{p=0}^\lambda \left( \bI_m \otimes \bu(\bx)^\top \right) \br_p \kronF{\bx}{p} =-  \left[\sum_{p=1}^\infty \bG_p \left(\kronF{\bx}{p} \otimes \bI_m \right) + \bB\right]^\top\frac{\partial \overline V(\bz)}{\partial \bz},  \label{ch3:eq:HJB-PDE-u-plugged-in}
\end{align}
where $\bz = \bE \bx$.
The degree~2 terms in \cref{ch3:eq:HJB-PDE-plugged-in} and degree~1 terms in \cref{ch3:eq:HJB-PDE-u-plugged-in}, respectively, are
\begin{align}
     & 0  = \bx^\top \bE^\top \overline\bV_2  \bA \bx  + \bx^\top \bE^\top \overline\bV_2  \bB  \bK_1 \bx + \frac{1}{2} \bx^\top \bQ_2 \bx + \frac{1}{2} \bx^\top \bK_1^\top  \bR_0  \bK_1 \bx, \label{ch3:eq:HJB-quadratic-terms-1} \\
     & \bR_0 \bK_1 \bx  = - \bB^\top \overline\bV_2 \bE \bx . \label{ch3:eq:HJB-quadratic-terms-2}
\end{align}
Under the stated assumptions, $\bR_0$ must be invertible, so
\cref{ch3:eq:HJB-quadratic-terms-2} can be rearranged for $\bK_1$, and inserting the resulting expression for $\bK_1$ into \cref{ch3:eq:HJB-quadratic-terms-1} and differentiating twice yields the GARE \cref{ch3:eq:ARE}.

Now assume we have $\bar\bv_{k-1}$ and $\bK_{k-2}$;
the next value function coefficient $\bar\bv_k$ is computed from the collection of degree~$k$ terms in \cref{ch3:eq:HJB-PDE-plugged-in}, and subsequently the next feedback law coefficient $\bK_{k-1}$ is computed from the collection of degree~$k-1$ terms in \cref{ch3:eq:HJB-PDE-u-plugged-in}.
We denote by $\textit{LHS}$ the terms in \cref{ch3:eq:HJB-PDE-plugged-in} of degree~$k$ containing the unknown $\bar\bv_k$, which can be written as
\begin{align*}
    \textit{LHS} & \coloneqq -\frac{1}{2} k\bar\bv_k^\top\left(   \bI_n \otimes  \kronF{\bz}{k-1}\right) (\bA + \bB \bK_1 ) \bx.
\end{align*}
These terms will form the left-hand-side of the linear algebraic system for $\bar\bv_k$.
After a series of algebraic manipulations and the substitution $\bz = \bE \bx$, the Kronecker products can be reordered to rewrite these terms using the generalized \textit{$k$-way Lyapunov matrix}:
\begin{align}
    \textit{LHS}
     & =  -\frac{1}{2} \bar\bv_k^\top \cL^\bE_{k}(\bA + \bB \bK_1 ) \kronF{\bx}{k} .      \label{ch3:eq:k-way-trick-2}
\end{align}

We denote by $\textit{RHS}$ the terms in \cref{ch3:eq:HJB-PDE-plugged-in} of degree~$k$ containing all terms other than the unknown $\bar\bv_k$, as they make up the right-hand-side of the linear algebraic equation for $\bar\bv_k$.
This collection of terms is
\begin{align}
    \textit{RHS} & \coloneqq  \frac{1}{2} \sum_{\mathclap{i+p=k+1}} i \bar\bv_i^\top \left(\bI_n \otimes \kronF{\bz}{i-1}\right) \bF_p \kronF{\bx}{p}
    + \frac{1}{2} \sum_{\mathclap{i+p+q=k+1}} i \bar\bv_i^\top \left(\bI_n \otimes \kronF{\bz}{i-1}\right)  \bG_p \left(\kronF{\bx}{p} \otimes \bI_m \right) \bK_q \kronF{\bx}{q}  \label{ch3:eq:HJB-PDE-plugged-in-symmetric-1}                                      \\
                 & \hspace{1cm} + \frac{1}{2}  \bq_k^\top \kronF{\bx}{k} + \frac{1}{2} \sum_{\mathclap{i+p+q=k}}{\kronF{\bx}{i}}^\top \br_i^\top \left( \bK_p \kronF{\bx}{p}   \otimes  \bK_q \kronF{\bx}{q} \right).   \label{ch3:eq:HJB-PDE-plugged-in-symmetric-3}
\end{align}
The terms from \cref{ch3:eq:HJB-PDE-plugged-in-symmetric-1} containing the drift nonlinearities $\bF_p$ can be rewritten using the generalized \textit{$k$-way Lyapunov matrix} as
\begin{align}
     & \frac{1}{2}i \bar\bv_i^\top \left(\bI_n \otimes \kronF{\bz}{i-1}\right) \bF_p \kronF{\bx}{p}
    = \frac{1}{2} \bar\bv_i^\top \cL^\bE_i(\bF_p) \kronF{\bx}{k}. \label{ch3:eq:N-terms}
\end{align}
The terms from \cref{ch3:eq:HJB-PDE-plugged-in-symmetric-1} containing the input map nonlinearities $\bG_p$ can be rewritten using Kronecker product identities as
\begin{align}
    \frac{1}{2}  i \bar\bv_i^\top \left(\bI_n \otimes \kronF{\bz}{i-1}\right) \bG_p \left(\kronF{\bx}{p} \otimes \bI_m \right)   \bK_q \kronF{\bx}{q}
     & = \frac{1}{2} i \Vec\left[\bK_q^\top\cG\cV^\bE_{p,i}\right]^\top {\kronF{\bx}{k}}, \label{ch3:eq:G-terms}
\end{align}
where $\cG\cV^\bE_{p,i} = {\text{reshape}\left(\bG_p^\top\overline\bV_i \kronF{\bE}{i-1},m,n^{k-q}\right)}$.
The state penalty terms containing  $\bq_p$ are already in the desired form $\frac{1}{2}  \bq_p^\top \kronF{\bx}{p}$.
The input penalty terms containing $\br_i$ are rewritten using Kronecker product identities as
\begin{align}
    \frac{1}{2} {\kronF{\bx}{i}}^\top \br_i^\top \left(\bK_p \kronF{\bx}{p} \otimes \bK_q \kronF{\bx}{q}\right)
     & =  \frac{1}{2} \Vec\left[\br_i^\top \left(\bK_p \otimes \bK_q \right)\right]^\top \kronF{\bx}{k}. \label{ch3:eq:R-terms}
\end{align}

Combining \cref{ch3:eq:k-way-trick-2,ch3:eq:N-terms,ch3:eq:G-terms,ch3:eq:R-terms},
multiplying by two, and summing over the appropriate sets of indices, the collection of degree $k$ terms is
\begin{align*}
    -\bar\bv_k^\top \cL^\bE_{k}(\bA + \bB\bK_1 ) \kronF{\bx}{k} & =
    \sum_{\mathclap{i+p=k+1}} \bar\bv_i^\top \cL^\bE_{i} (\bF_p) \kronF{\bx}{k}
    + \sum_{\mathclap{i+p+q=k+1}} i \Vec\left[\bK_q^\top\cG\cV^\bE_{p,i}\right]^\top {\kronF{\bx}{k}} \\
                                                                & +  \bq_k^\top \kronF{\bx}{k}
    + \sum_{\mathclap{i+p+q=k}} \Vec\left[\br_i^\top \left(\bK_p \otimes \bK_q \right)\right]^\top \kronF{\bx}{k}.
\end{align*}
Requiring this to hold for all $\bx$ leads to \cref{ch3:eq:value-function-linear-system}.
A symmetrization operation is required on the right-hand side due to the elimination of the $\bx$ variables to ensure that the computed coefficient is also symmetric, e.g. $\bx^\top \bM \bx = 0 \Leftrightarrow \bM = \bzero$ is only true if $\bM$ is symmetric.

The linear systems \labelcref{ch3:eq:value-function-linear-system} are uniquely solvable, as $\cL^\bE_{k}(\bA + \bB\bK_1 )^\top$ is a nonsingular matrix.
The proof follows easily by converting the dynamics to standard form by multiplying by $\bE^{-1}$, after which the standard results of Al'brekht apply.
For a more detailed discussion of the proof for the standard case, we refer to our related work \cite[Sec. III-A]{Corbin2025a}.

The gain coefficient $\bK_{k-1}$ is computed from the collection of degree $k-1$ terms in \cref{ch3:eq:HJB-PDE-u-plugged-in}:
\begin{align*}
    \sum_{\mathclap{p+q=k-1}}\left( \bI_m \otimes {\kronF{\bx}{p}}^\top \bK_p^\top \right) \br_q \kronF{\bx}{q} =- \frac{1}{2} \sum_{\mathclap{i+p=k}} \left({\kronF{\bx}{p}}^\top \otimes \bI_m \right)\bG_p^\top \left(\bI_n \otimes {\kronF{\bz}{i-1}}^\top\right) \bar\bv_i i.
\end{align*}
The left-hand-side can be rewritten using Kronecker product identities as
\begin{align}
    \left(\bI_m \otimes {\kronF{\bx}{p}}^\top \bK_p^\top\right) \br_q \kronF{\bx}{q} & =  \bR_q^\top \left(\bK_p \otimes \bI_{n^q}\right) \kronF{\bx}{k-1}, \label{ch3:eq:lhs-2}
\end{align}
where $\bR_q \coloneqq \text{reshape}\left(\br_q^\top, mn^q,m\right)$.
The right-hand-side is rewritten as
\begin{align}
    \frac{1}{2}\left({\kronF{\bx}{p}}^\top \otimes \bI_m \right)\bG_p^\top \left(\bI_n \otimes {\kronF{\bz}{i-1}}^\top\right) \bar\bv_i i
     & = \frac{1}{2}i
    \cG\cV^\bE_{p,i}
    \kronF{\bx}{k-1}, \label{ch3:eq:rhs-2}
\end{align}
where $\cG\cV^\bE_{p,i} \coloneqq \text{reshape}\left(\bG_p^\top \overline\bV_i \kronF{\bE}{i-1},m,n^{k-1}\right)$.
Combining \cref{ch3:eq:lhs-2,ch3:eq:rhs-2} yields
\begin{align*}
    \sum_{\mathclap{p+q=k-1}} \bR_q^\top \left(\bK_p \otimes \bI_{n^q}\right) \kronF{\bx}{k-1}
     & = -  \sum_{\mathclap{i+p=k}} \frac{1}{2}i\cG\cV^\bE_{p,i} \kronF{\bx}{k-1}.
\end{align*}
Note that the $q=0$ and $p=k-1$ case on the left side is the only term containing the unknown coefficient $\bK_{k-1}$, so that term can be extracted.
We can then rearrange the equation for the unknown coefficient $\bK_{k-1}$, completing the proof:
\begin{align*}
    \bK_{k-1} & =  \bR_0^{-1}  \left[- \sum_{\mathclap{\substack{p+q=k-1 \\q\neq0}}}  \bR_q^\top \left(\bK_p \otimes \bI_{n^q}\right)   -\sum_{\mathclap{i+p=k}} \frac{1}{2} i\cG\cV^\bE_{p,i}\right].
\end{align*}

\section{Conditions for the existence of a stabilizing solution to the LQR problem}\label{sec:are-solution}
One of the key assumptions in \cref{ch3:thm:albrekht} is that a stabilizing solution exists to the LQR problem associated with the linearized dynamics.
This is equivalent to assuming the GARE \cref{ch3:eq:ARE} has a unique stabilizing solution, which is satisfied when the transformed pair $(\bE^{-1}\bA,\bE^{-1}\bB)$ is stabilizable and $(\bE^{-1}\bA,\bC)$ is detectable, where $\bQ = \bC^\top \bC$.
\begin{theorem}[Popov-Belevitch-Hautus (PBH) tests for stabilizability and detectability \cite{Hespanha2018}]
    The continuous-time generalized linear time-invariant (LTI) system
    \begin{equation}\label{eq:LTI}
        \begin{split}
            \dot{\bx} & = \bA \bx + \bB \bu, \\
            \by       & = \bC \bx,
        \end{split}
    \end{equation}
    is stabilizable if and only if
    \begin{align*}
        \text{rank}\begin{bmatrix}
                       \bA-\lambda\bI & \bB
                   \end{bmatrix} = n, \qquad \forall \lambda \in \mathbb{C} : \mathfrak{R}[\lambda] \geq 0.
    \end{align*}
    The continuous-time generalized LTI system \cref{eq:LTI} is detectable if and only if
    \begin{align*}
        \text{rank}\begin{bmatrix}
                       \bA-\lambda\bI \\ \bC
                   \end{bmatrix} = n, \qquad \forall \lambda \in \mathbb{C} : \mathfrak{R}[\lambda] \geq 0.
    \end{align*}
\end{theorem}
It turns out that detectability of $(\bE^{-1}\bA,\bC)$ is equivalent to $(\bE^{-1}\bA,\bQ)$ \cite{Hespanha2018}.


\begin{thebibliography}{59}
    \expandafter\ifx\csname natexlab\endcsname\relax\def\natexlab#1{#1}\fi
    \providecommand{\url}[1]{\texttt{#1}}
    \providecommand{\href}[2]{#2}
    \providecommand{\path}[1]{#1}
    \providecommand{\DOIprefix}{doi:}
    \providecommand{\ArXivprefix}{arXiv:}
    \providecommand{\URLprefix}{URL: }
    \providecommand{\Pubmedprefix}{pmid:}
    \providecommand{\doi}[1]{\href{http://dx.doi.org/#1}{\path{#1}}}
    \providecommand{\Pubmed}[1]{\href{pmid:#1}{\path{#1}}}
    \providecommand{\bibinfo}[2]{#2}
    \ifx\xfnm\relax \def\xfnm[#1]{\unskip,\space#1}\fi
    \bibitem[{Bardi and Dolcetta(2013)}]{Bardi2013}
    \bibinfo{author}{M.~Bardi}, \bibinfo{author}{I.~C. Dolcetta},
    \bibinfo{title}{Optimal control and viscosity solutions of
        {H}amilton-{J}acobi-{B}ellman equations}, Systems \& Control,
    \bibinfo{publisher}{Springer}, \bibinfo{address}{New York, NY},
    \bibinfo{year}{2013}.
    \bibitem[{Crandall and Lions(1983)}]{Crandall1983}
    \bibinfo{author}{M.~G. Crandall}, \bibinfo{author}{P.-L. Lions},
    \newblock \bibinfo{title}{Viscosity solutions of {H}amilton-{J}acobi
    equations},
    \newblock \bibinfo{journal}{Transactions of the American Mathematical Society}
    \bibinfo{volume}{277} (\bibinfo{year}{1983}) \bibinfo{pages}{1--42}.
    \DOIprefix\doi{10.1090/s0002-9947-1983-0690039-8}.
    \bibitem[{Falcone and Ferretti(2014)}]{Falcone2014}
    \bibinfo{author}{M.~Falcone}, \bibinfo{author}{R.~Ferretti},
    \bibinfo{title}{Semi-{L}agrangian approximation schemes for linear and
        {H}amilton-{J}acobi equations}, Applied mathematics,
    \bibinfo{publisher}{SIAM}, \bibinfo{address}{Philadelphia, PA},
    \bibinfo{year}{2014}.
    \bibitem[{Falcone and Ferretti(2016)}]{Falcone2016}
    \bibinfo{author}{M.~Falcone}, \bibinfo{author}{R.~Ferretti},
    \newblock \bibinfo{title}{Numerical methods for {H}amilton-{J}acobi type
    equations},
    \newblock in: \bibinfo{booktitle}{Handbook of Numerical Analysis},
    \bibinfo{publisher}{Elsevier}, \bibinfo{year}{2016}, pp.
    \bibinfo{pages}{603--626}. \DOIprefix\doi{10.1016/bs.hna.2016.09.018}.
    \bibitem[{Kang and Wilcox(2017)}]{Kang2017}
    \bibinfo{author}{W.~Kang}, \bibinfo{author}{L.~C. Wilcox},
    \newblock \bibinfo{title}{Mitigating the curse of dimensionality: sparse grid
        characteristics method for optimal feedback control and {HJB} equations},
    \newblock \bibinfo{journal}{Computational Optimization and Applications}
    \bibinfo{volume}{68} (\bibinfo{year}{2017}) \bibinfo{pages}{289--315}.
    \DOIprefix\doi{10.1007/s10589-017-9910-0}.
    \bibitem[{Beard et~al.(1997)Beard, Saridis, and Wen}]{Beard1997}
    \bibinfo{author}{R.~W. Beard}, \bibinfo{author}{G.~N. Saridis},
    \bibinfo{author}{J.~T. Wen},
    \newblock \bibinfo{title}{Galerkin approximations of the generalized
        {H}amilton-{J}acobi-{B}ellman equation},
    \newblock \bibinfo{journal}{Automatica} \bibinfo{volume}{33}
    (\bibinfo{year}{1997}) \bibinfo{pages}{2159--2177}.
    \DOIprefix\doi{10.1016/s0005-1098(97)00128-3}.
    \bibitem[{Kalise and Kunisch(2018)}]{Kalise2018}
    \bibinfo{author}{D.~Kalise}, \bibinfo{author}{K.~Kunisch},
    \newblock \bibinfo{title}{Polynomial approximation of high-dimensional
    {H}amilton-{J}acobi-{B}ellman equations and applications to feedback control
    of semilinear parabolic {PDE}s},
    \newblock \bibinfo{journal}{SIAM Journal on Scientific Computing}
    \bibinfo{volume}{40} (\bibinfo{year}{2018}) \bibinfo{pages}{A629--A652}.
    \DOIprefix\doi{10.1137/17M1116635}.
    \bibitem[{Gorodetsky et~al.(2018)Gorodetsky, Karaman, and
                Marzouk}]{Gorodetsky2018}
    \bibinfo{author}{A.~Gorodetsky}, \bibinfo{author}{S.~Karaman},
    \bibinfo{author}{Y.~Marzouk},
    \newblock \bibinfo{title}{High-dimensional stochastic optimal control using
        continuous tensor decompositions},
    \newblock \bibinfo{journal}{The International Journal of Robotics Research}
    \bibinfo{volume}{37} (\bibinfo{year}{2018}) \bibinfo{pages}{340--377}.
    \DOIprefix\doi{10.1177/0278364917753994}.
    \bibitem[{Dolgov et~al.(2023)Dolgov, Kalise, and Saluzzi}]{Dolgov2023}
    \bibinfo{author}{S.~Dolgov}, \bibinfo{author}{D.~Kalise},
    \bibinfo{author}{L.~Saluzzi},
    \newblock \bibinfo{title}{Data-driven tensor train gradient cross approximation
    for {H}amilton-{J}acobi-{B}ellman equations},
    \newblock \bibinfo{journal}{SIAM Journal on Scientific Computing}
    \bibinfo{volume}{45} (\bibinfo{year}{2023}) \bibinfo{pages}{A2153--A2184}.
    \DOIprefix\doi{10.1137/22m1498401}.
    \bibitem[{Ehring and Haasdonk(2024)}]{Ehring2024}
    \bibinfo{author}{T.~Ehring}, \bibinfo{author}{B.~Haasdonk},
    \newblock \bibinfo{title}{Hermite kernel surrogates for the value function of
        high-dimensional nonlinear optimal control problems},
    \newblock \bibinfo{journal}{Advances in Computational Mathematics}
    \bibinfo{volume}{50} (\bibinfo{year}{2024}).
    \DOIprefix\doi{10.1007/s10444-024-10128-5}.
    \bibitem[{Nakamura-Zimmerer et~al.(2021)Nakamura-Zimmerer, Gong, and
                Kang}]{NakamuraZimmerer2021}
    \bibinfo{author}{T.~Nakamura-Zimmerer}, \bibinfo{author}{Q.~Gong},
    \bibinfo{author}{W.~Kang},
    \newblock \bibinfo{title}{Adaptive deep learning for high-dimensional
    {H}amilton-{J}acobi-{B}ellman equations},
    \newblock \bibinfo{journal}{SIAM Journal on Scientific Computing}
    \bibinfo{volume}{43} (\bibinfo{year}{2021}) \bibinfo{pages}{A1221--A1247}.
    \DOIprefix\doi{10.1137/19m1288802}.
    \bibitem[{Ito and Schroeter(2001)}]{Ito2001}
    \bibinfo{author}{K.~Ito}, \bibinfo{author}{J.~D. Schroeter},
    \newblock \bibinfo{title}{Reduced order feedback synthesis for viscous
        incompressible flows},
    \newblock \bibinfo{journal}{Mathematical and Computer Modelling}
    \bibinfo{volume}{33} (\bibinfo{year}{2001}) \bibinfo{pages}{173--192}.
    \DOIprefix\doi{10.1016/s0895-7177(00)00237-5}.
    \bibitem[{Rawlings(2020)}]{Rawlings2020}
    \bibinfo{author}{J.~B. Rawlings}, \bibinfo{title}{Model predictive control:
        theory, computation and design}, \bibinfo{edition}{second} ed.,
    \bibinfo{publisher}{Nob Hill Publishing, LLC}, \bibinfo{address}{Santa
        Barbara}, \bibinfo{year}{2020}.
    \bibitem[{Al'brekht(1961)}]{Albrekht1961}
    \bibinfo{author}{E.~G. Al'brekht},
    \newblock \bibinfo{title}{On the optimal stabilization of nonlinear systems},
    \newblock \bibinfo{journal}{Journal of Applied Mathematics and Mechanics}
    \bibinfo{volume}{25} (\bibinfo{year}{1961}) \bibinfo{pages}{1254--1266}.
    \DOIprefix\doi{10.1016/0021-8928(61)90005-3}.
    \bibitem[{Krener et~al.(2013)Krener, Aguilar, and Hunt}]{Krener2013}
    \bibinfo{author}{A.~Krener}, \bibinfo{author}{C.~Aguilar},
    \bibinfo{author}{T.~Hunt},
    \newblock \bibinfo{title}{Series solutions of {HJB} equations},
    \newblock \bibinfo{journal}{Mathematical System Theory---Festschrift in Honor
        of Uwe Helmke on the Occasion of his Sixtieth Birthday}
    (\bibinfo{year}{2013}) \bibinfo{pages}{247--260}.
    \bibitem[{Krener(2019)}]{Krener2019}
    \bibinfo{author}{A.~J. Krener}, \bibinfo{title}{Nonlinear {S}ystems {T}oolbox},
    \bibinfo{howpublished}{Available on request to ajkrener@nps.edu},
    \bibinfo{year}{2019}.
    \bibitem[{Grasedyck(2004)}]{Grasedyck2004}
    \bibinfo{author}{L.~Grasedyck},
    \newblock \bibinfo{title}{Existence and computation of low {K}ronecker-rank
        approximations for large linear systems of tensor product structure},
    \newblock \bibinfo{journal}{Computing} \bibinfo{volume}{72}
    (\bibinfo{year}{2004}). \DOIprefix\doi{10.1007/s00607-003-0037-z}.
    \bibitem[{Chen and Kressner(2019)}]{Chen2019}
    \bibinfo{author}{M.~Chen}, \bibinfo{author}{D.~Kressner},
    \newblock \bibinfo{title}{Recursive blocked algorithms for linear systems with
            {K}ronecker product structure},
    \newblock \bibinfo{journal}{Numerical Algorithms} \bibinfo{volume}{84}
    (\bibinfo{year}{2019}) \bibinfo{pages}{1199--1216}.
    \DOIprefix\doi{10.1007/s11075-019-00797-5}.
    \bibitem[{Breiten et~al.(2018)Breiten, Kunisch, and Pfeiffer}]{Breiten2018}
    \bibinfo{author}{T.~Breiten}, \bibinfo{author}{K.~Kunisch},
    \bibinfo{author}{L.~Pfeiffer},
    \newblock \bibinfo{title}{Numerical study of polynomial feedback laws for a
        bilinear control problem},
    \newblock \bibinfo{journal}{Mathematical Control \& Related Fields}
    \bibinfo{volume}{8} (\bibinfo{year}{2018}) \bibinfo{pages}{557--582}.
    \DOIprefix\doi{10.3934/mcrf.2018023}.
    \bibitem[{Breiten et~al.(2019)Breiten, Kunisch, and Pfeiffer}]{Breiten2019a}
    \bibinfo{author}{T.~Breiten}, \bibinfo{author}{K.~Kunisch},
    \bibinfo{author}{L.~Pfeiffer},
    \newblock \bibinfo{title}{Feedback stabilization of the two-dimensional
    {N}avier–{S}tokes equations by value function approximation},
    \newblock \bibinfo{journal}{Applied Mathematics \& Optimization}
    \bibinfo{volume}{80} (\bibinfo{year}{2019}) \bibinfo{pages}{599--641}.
    \DOIprefix\doi{10.1007/s00245-019-09586-x}.
    \bibitem[{Almubarak et~al.(2019)Almubarak, Sadegh, and Taylor}]{Almubarak2019}
    \bibinfo{author}{H.~Almubarak}, \bibinfo{author}{N.~Sadegh},
    \bibinfo{author}{D.~G. Taylor},
    \newblock \bibinfo{title}{Infinite horizon nonlinear quadratic cost regulator},
    \newblock in: \bibinfo{booktitle}{2019 American Control Conference ({ACC})},
    \bibinfo{publisher}{{IEEE}}, \bibinfo{year}{2019}, pp.
    \bibinfo{pages}{5570--5575}. \DOIprefix\doi{10.23919/acc.2019.8814783}.
    \bibitem[{Borggaard and Zietsman(2018)}]{Borggaard2018}
    \bibinfo{author}{J.~Borggaard}, \bibinfo{author}{L.~Zietsman},
    \newblock \bibinfo{title}{Computation of nonlinear feedback for flow control
        problems},
    \newblock in: \bibinfo{booktitle}{2018 American Control Conference ({ACC})},
    \bibinfo{publisher}{{IEEE}}, \bibinfo{year}{2018}, pp.
    \bibinfo{pages}{1726--1731}. \DOIprefix\doi{10.23919/acc.2018.8431410}.
    \bibitem[{Borggaard and Zietsman(2020)}]{Borggaard2020}
    \bibinfo{author}{J.~Borggaard}, \bibinfo{author}{L.~Zietsman},
    \newblock \bibinfo{title}{The quadratic-quadratic regulator problem:
        approximating feedback controls for quadratic-in-state nonlinear systems},
    \newblock in: \bibinfo{booktitle}{2020 American Control Conference ({ACC})},
    \bibinfo{publisher}{{IEEE}}, \bibinfo{year}{2020}, pp.
    \bibinfo{pages}{818--823}. \DOIprefix\doi{10.23919/ACC45564.2020.9147286}.
    \bibitem[{Borggaard and Zietsman(2021)}]{Borggaard2021}
    \bibinfo{author}{J.~Borggaard}, \bibinfo{author}{L.~Zietsman},
    \newblock \bibinfo{title}{On approximating polynomial-quadratic regulator
        problems},
    \newblock \bibinfo{journal}{{IFAC}-{PapersOnLine}} \bibinfo{volume}{54}
    (\bibinfo{year}{2021}) \bibinfo{pages}{329--334}.
    \DOIprefix\doi{10.1016/j.ifacol.2021.06.090}.
    \bibitem[{Kramer et~al.(2024)Kramer, Gugercin, Borggaard, and
                Balicki}]{Kramer2024}
    \bibinfo{author}{B.~Kramer}, \bibinfo{author}{S.~Gugercin},
    \bibinfo{author}{J.~Borggaard}, \bibinfo{author}{L.~Balicki},
    \newblock \bibinfo{title}{Scalable computation of energy functions for
        nonlinear balanced truncation},
    \newblock \bibinfo{journal}{Computer Methods in Applied Mechanics and
        Engineering} \bibinfo{volume}{427} (\bibinfo{year}{2024})
    \bibinfo{pages}{117011}. \DOIprefix\doi{10.1016/j.cma.2024.117011}.
    \bibitem[{Benner et~al.(2013)Benner, Saak, Stoll, and Weichelt}]{Benner2013b}
    \bibinfo{author}{P.~Benner}, \bibinfo{author}{J.~Saak},
    \bibinfo{author}{M.~Stoll}, \bibinfo{author}{H.~K. Weichelt},
    \newblock \bibinfo{title}{Efficient solution of large-scale saddle point
        systems arising in {R}iccati-based boundary feedback stabilization of
        incompressible {S}tokes flow},
    \newblock \bibinfo{journal}{SIAM Journal on Scientific Computing}
    \bibinfo{volume}{35} (\bibinfo{year}{2013}) \bibinfo{pages}{S150--S170}.
    \DOIprefix\doi{10.1137/120881312}.
    \bibitem[{Corbin and Kramer(2024)}]{Corbin2024e}
    \bibinfo{author}{N.~A. Corbin}, \bibinfo{author}{B.~Kramer},
    \newblock \bibinfo{title}{Computing solutions to the polynomial-polynomial
        regulator problem},
    \newblock in: \bibinfo{booktitle}{2024 IEEE 63rd Conference on Decision and
        Control (CDC)}, \bibinfo{publisher}{IEEE}, \bibinfo{year}{2024}, pp.
    \bibinfo{pages}{2689--2696}. \DOIprefix\doi{10.1109/cdc56724.2024.10885897}.
    \bibitem[{Lukes(1969)}]{Lukes1969}
    \bibinfo{author}{D.~L. Lukes},
    \newblock \bibinfo{title}{Optimal regulation of nonlinear dynamical systems},
    \newblock \bibinfo{journal}{{SIAM} Journal on Control} \bibinfo{volume}{7}
    (\bibinfo{year}{1969}) \bibinfo{pages}{75--100}.
    \DOIprefix\doi{10.1137/0307007}.
    \bibitem[{Corbin and Kramer(2025)}]{Corbin2025a}
    \bibinfo{author}{N.~A. Corbin}, \bibinfo{author}{B.~Kramer},
    \newblock \bibinfo{title}{Scalable computation of $\mathcal {H}_\infty$ energy
        functions for polynomial control-affine systems},
    \newblock \bibinfo{journal}{IEEE Transactions on Automatic Control}
    \bibinfo{volume}{70} (\bibinfo{year}{2025}) \bibinfo{pages}{3088--3100}.
    \DOIprefix\doi{10.1109/tac.2024.3494472}.
    \bibitem[{Kunisch et~al.(2004)Kunisch, Volkwein, and Xie}]{Kunisch2004}
    \bibinfo{author}{K.~Kunisch}, \bibinfo{author}{S.~Volkwein},
    \bibinfo{author}{L.~Xie},
    \newblock \bibinfo{title}{{HJB}-{POD}-based feedback design for the optimal
    control of evolution problems},
    \newblock \bibinfo{journal}{SIAM Journal on Applied Dynamical Systems}
    \bibinfo{volume}{3} (\bibinfo{year}{2004}) \bibinfo{pages}{701--722}.
    \DOIprefix\doi{10.1137/030600485}.
    \bibitem[{Alla et~al.(2017)Alla, Schmidt, and Haasdonk}]{Alla2017}
    \bibinfo{author}{A.~Alla}, \bibinfo{author}{A.~Schmidt},
    \bibinfo{author}{B.~Haasdonk},
    \newblock \bibinfo{title}{Model order reduction approaches for infinite horizon
        optimal control problems via the {HJB} equation},
    \newblock in: \bibinfo{booktitle}{Model Reduction of Parametrized Systems},
    \bibinfo{publisher}{Springer}, \bibinfo{year}{2017}, pp.
    \bibinfo{pages}{333--347}. \DOIprefix\doi{10.1007/978-3-319-58786-8_21}.
    \bibitem[{Benner and Damm(2011)}]{Benner2011}
    \bibinfo{author}{P.~Benner}, \bibinfo{author}{T.~Damm},
    \newblock \bibinfo{title}{{L}yapunov equations, energy functionals, and model
        order reduction of bilinear and stochastic systems},
    \newblock \bibinfo{journal}{SIAM Journal on Control and Optimization}
    \bibinfo{volume}{49} (\bibinfo{year}{2011}) \bibinfo{pages}{686--711}.
    \DOIprefix\doi{10.1137/09075041x}.
    \bibitem[{Jonckheere and Silverman(1983)}]{Jonckheere1983}
    \bibinfo{author}{E.~Jonckheere}, \bibinfo{author}{L.~Silverman},
    \newblock \bibinfo{title}{A new set of invariants for linear
        systems---application to reduced order compensator design},
    \newblock \bibinfo{journal}{{IEEE} Transactions on Automatic Control}
    \bibinfo{volume}{28} (\bibinfo{year}{1983}) \bibinfo{pages}{953--964}.
    \DOIprefix\doi{10.1109/tac.1983.1103159}.
    \bibitem[{Mustafa and Glover(1991)}]{Mustafa1991}
    \bibinfo{author}{D.~Mustafa}, \bibinfo{author}{K.~Glover},
    \newblock \bibinfo{title}{Controller reduction by
            {$\mathcal{H}_\infty$}-balanced truncation},
    \newblock \bibinfo{journal}{{IEEE} Transactions on Automatic Control}
    \bibinfo{volume}{36} (\bibinfo{year}{1991}) \bibinfo{pages}{668--682}.
    \DOIprefix\doi{10.1109/9.86941}.
    \bibitem[{Benner et~al.(2017)Benner, Bujanovi{\'{c}}, Kürschner, and
    Saak}]{Benner2017b}
    \bibinfo{author}{P.~Benner}, \bibinfo{author}{Z.~Bujanovi{\'{c}}},
    \bibinfo{author}{P.~Kürschner}, \bibinfo{author}{J.~Saak},
    \newblock \bibinfo{title}{{RADI}: a low-rank {ADI}-type algorithm for large
    scale algebraic {R}iccati equations},
    \newblock \bibinfo{journal}{Numerische Mathematik} \bibinfo{volume}{138}
    (\bibinfo{year}{2017}) \bibinfo{pages}{301--330}.
    \DOIprefix\doi{10.1007/s00211-017-0907-5}.
    \bibitem[{Saak et~al.(2025)Saak, Köhler, and Benner}]{Saak2022}
    \bibinfo{author}{J.~Saak}, \bibinfo{author}{M.~Köhler},
    \bibinfo{author}{P.~Benner}, \bibinfo{title}{{M-M.E.S.S.} - the matrix
        equation sparse solver library}, \bibinfo{year}{2025}. \URLprefix
    \url{https://github.com/mpimd-csc/mmess}.
    \DOIprefix\doi{10.5281/ZENODO.632897}.
    \bibitem[{Golub and Van~Loan(2013)}]{Golub2013}
    \bibinfo{author}{G.~H. Golub}, \bibinfo{author}{C.~Van~Loan},
    \bibinfo{title}{Matrix computations}, \bibinfo{edition}{fourth} ed.,
    \bibinfo{publisher}{Johns Hopkins University Press}, \bibinfo{year}{2013}.
    \bibitem[{Kalman(1963)}]{Kalman1963}
    \bibinfo{author}{R.~E. Kalman},
    \newblock \bibinfo{title}{The theory of optimal control and the calculus of
        variations},
    \newblock in: \bibinfo{editor}{R.~Bellman} (Ed.),
    \bibinfo{booktitle}{Mathematical Optimization Techniques},
    \bibinfo{publisher}{University of California Press}, \bibinfo{year}{1963},
    pp. \bibinfo{pages}{309--331}. \DOIprefix\doi{10.1525/9780520319875-018}.
    \bibitem[{Moler and Stewart(1973)}]{Moler1973}
    \bibinfo{author}{C.~B. Moler}, \bibinfo{author}{G.~W. Stewart},
    \newblock \bibinfo{title}{An algorithm for generalized matrix eigenvalue
        problems},
    \newblock \bibinfo{journal}{SIAM Journal on Numerical Analysis}
    \bibinfo{volume}{10} (\bibinfo{year}{1973}) \bibinfo{pages}{241--256}.
    \DOIprefix\doi{10.1137/0710024}.
    \bibitem[{Penzl(1998)}]{Penzl1998}
    \bibinfo{author}{T.~Penzl},
    \newblock \bibinfo{title}{Numerical solution of generalized {L}yapunov
        equations},
    \newblock \bibinfo{journal}{Advances in Computational Mathematics}
    \bibinfo{volume}{8} (\bibinfo{year}{1998}) \bibinfo{pages}{33--48}.
    \DOIprefix\doi{10.1023/a:1018979826766}.
    \bibitem[{Corbin(2025)}]{Corbin2024}
    \bibinfo{author}{N.~A. Corbin}, \bibinfo{title}{{PPR} repository},
    \bibinfo{howpublished}{Available online:
        \url{https://github.com/cnick1/PPR}}, \bibinfo{year}{2025}. \URLprefix
    \url{https://github.com/cnick1/PPR}.
    \bibitem[{Penzl(1999)}]{Penzl1999}
    \bibinfo{author}{T.~Penzl},
    \newblock \bibinfo{title}{A cyclic low-rank {S}mith method for large sparse
        {L}yapunov equations},
    \newblock \bibinfo{journal}{SIAM Journal on Scientific Computing}
    \bibinfo{volume}{21} (\bibinfo{year}{1999}) \bibinfo{pages}{1401--1418}.
    \DOIprefix\doi{10.1137/s1064827598347666}.
    \bibitem[{Penzl(2000)}]{Penzl2000}
    \bibinfo{author}{T.~Penzl},
    \newblock \bibinfo{title}{Eigenvalue decay bounds for solutions of {L}yapunov
        equations: the symmetric case},
    \newblock \bibinfo{journal}{Systems \& Control Letters} \bibinfo{volume}{40}
    (\bibinfo{year}{2000}) \bibinfo{pages}{139--144}.
    \DOIprefix\doi{10.1016/s0167-6911(00)00010-4}.
    \bibitem[{Opmeer(2010)}]{Opmeer2010}
    \bibinfo{author}{M.~R. Opmeer},
    \newblock \bibinfo{title}{Decay of {H}ankel singular values of analytic control
        systems},
    \newblock \bibinfo{journal}{Systems \& Control Letters} \bibinfo{volume}{59}
    (\bibinfo{year}{2010}) \bibinfo{pages}{635--638}.
    \DOIprefix\doi{10.1016/j.sysconle.2010.07.009}.
    \bibitem[{Baker et~al.(2015)Baker, Embree, and Sabino}]{Baker2015}
    \bibinfo{author}{J.~Baker}, \bibinfo{author}{M.~Embree},
    \bibinfo{author}{J.~Sabino},
    \newblock \bibinfo{title}{Fast singular value decay for {L}yapunov solutions
        with nonnormal coefficients},
    \newblock \bibinfo{journal}{SIAM Journal on Matrix Analysis and Applications}
    \bibinfo{volume}{36} (\bibinfo{year}{2015}) \bibinfo{pages}{656--668}.
    \DOIprefix\doi{10.1137/140993867}.
    \bibitem[{Benner and Saak(2013)}]{Benner2013}
    \bibinfo{author}{P.~Benner}, \bibinfo{author}{J.~Saak},
    \newblock \bibinfo{title}{Numerical solution of large and sparse continuous
        time algebraic matrix {R}iccati and {L}yapunov equations: a state of the art
        survey},
    \newblock \bibinfo{journal}{{GAMM}-Mitteilungen} \bibinfo{volume}{36}
    (\bibinfo{year}{2013}) \bibinfo{pages}{32--52}.
    \DOIprefix\doi{10.1002/gamm.201310003}.
    \bibitem[{Allen and Cahn(1979)}]{Allen1979}
    \bibinfo{author}{S.~M. Allen}, \bibinfo{author}{J.~W. Cahn},
    \newblock \bibinfo{title}{A microscopic theory for antiphase boundary motion
        and its application to antiphase domain coarsening},
    \newblock \bibinfo{journal}{Acta Metallurgica} \bibinfo{volume}{27}
    (\bibinfo{year}{1979}) \bibinfo{pages}{1085--1095}.
    \DOIprefix\doi{10.1016/0001-6160(79)90196-2}.
    \bibitem[{Yang et~al.(2019)Yang, Ragnvaldsen, Bai, Yi, and Xu}]{Yang2019}
    \bibinfo{author}{Y.~Yang}, \bibinfo{author}{O.~Ragnvaldsen},
    \bibinfo{author}{Y.~Bai}, \bibinfo{author}{M.~Yi}, \bibinfo{author}{B.-X.
        Xu},
    \newblock \bibinfo{title}{3{D} non-isothermal phase-field simulation of
        microstructure evolution during selective laser sintering},
    \newblock \bibinfo{journal}{npj Computational Materials} \bibinfo{volume}{5}
    (\bibinfo{year}{2019}). \DOIprefix\doi{10.1038/s41524-019-0219-7}.
    \bibitem[{Li et~al.(2018)Li, Fan, Taniguchi, and Zhang}]{Li2018}
    \bibinfo{author}{J.-Q. Li}, \bibinfo{author}{T.-H. Fan},
    \bibinfo{author}{T.~Taniguchi}, \bibinfo{author}{B.~Zhang},
    \newblock \bibinfo{title}{Phase-field modeling on laser melting of a metallic
        powder},
    \newblock \bibinfo{journal}{International Journal of Heat and Mass Transfer}
    \bibinfo{volume}{117} (\bibinfo{year}{2018}) \bibinfo{pages}{412--424}.
    \DOIprefix\doi{10.1016/j.ijheatmasstransfer.2017.10.001}.
    \bibitem[{Ali et~al.(2021)Ali, Heider, and Markert}]{Ali2021}
    \bibinfo{author}{B.~Ali}, \bibinfo{author}{Y.~Heider},
    \bibinfo{author}{B.~Markert},
    \newblock \bibinfo{title}{Residual stresses in gas tungsten arc welding: a
        novel phase-field thermo-elastoplasticity modeling and parameter treatment
        framework},
    \newblock \bibinfo{journal}{Computational Mechanics} \bibinfo{volume}{69}
    (\bibinfo{year}{2021}) \bibinfo{pages}{565--587}.
    \DOIprefix\doi{10.1007/s00466-021-02104-3}.
    \bibitem[{Darabi et~al.(2023)Darabi, Azinpour, Reis, and de~Sa}]{Darabi2023}
    \bibinfo{author}{R.~Darabi}, \bibinfo{author}{E.~Azinpour},
    \bibinfo{author}{A.~Reis}, \bibinfo{author}{J.~C. de~Sa},
    \newblock \bibinfo{title}{Multi-scale multi-physics phase-field coupled
        thermo-mechanical approach for modeling of powder bed fusion process},
    \newblock \bibinfo{journal}{Applied Mathematical Modelling}
    \bibinfo{volume}{122} (\bibinfo{year}{2023}) \bibinfo{pages}{572--597}.
    \DOIprefix\doi{10.1016/j.apm.2023.06.021}.
    \bibitem[{Alla and Saluzzi(2020)}]{Alla2020}
    \bibinfo{author}{A.~Alla}, \bibinfo{author}{L.~Saluzzi},
    \newblock \bibinfo{title}{A {HJB}-{POD} approach for the control of nonlinear
        {PDE}s on a tree structure},
    \newblock \bibinfo{journal}{Applied Numerical Mathematics}
    \bibinfo{volume}{155} (\bibinfo{year}{2020}) \bibinfo{pages}{192--207}.
    \DOIprefix\doi{10.1016/j.apnum.2019.11.023}.
    \bibitem[{Dolgov et~al.(2021)Dolgov, Kalise, and Kunisch}]{Dolgov2021}
    \bibinfo{author}{S.~Dolgov}, \bibinfo{author}{D.~Kalise},
    \bibinfo{author}{K.~Kunisch},
    \newblock \bibinfo{title}{Tensor decomposition methods for high-dimensional
    {H}amilton-{J}acobi-{B}ellman equations},
    \newblock \bibinfo{journal}{{SIAM} Journal on Scientific Computing}
    \bibinfo{volume}{43} (\bibinfo{year}{2021}) \bibinfo{pages}{A1625--A1650}.
    \DOIprefix\doi{10.1137/19m1305136}.
    \bibitem[{Alla et~al.(2023)Alla, Kalise, and Simoncini}]{Alla2023}
    \bibinfo{author}{A.~Alla}, \bibinfo{author}{D.~Kalise},
    \bibinfo{author}{V.~Simoncini},
    \newblock \bibinfo{title}{State-{D}ependent {R}iccati {E}quation feedback
    stabilization for nonlinear {PDE}s},
    \newblock \bibinfo{journal}{Advances in Computational Mathematics}
    \bibinfo{volume}{49} (\bibinfo{year}{2023}).
    \DOIprefix\doi{10.1007/s10444-022-09998-4}.
    \bibitem[{Heiland and Werner(2023)}]{Heiland2023}
    \bibinfo{author}{J.~Heiland}, \bibinfo{author}{S.~W.~R. Werner},
    \newblock \bibinfo{title}{Low-complexity linear parameter-varying
    approximations of incompressible {N}avier-{S}tokes equations for truncated
        {S}tate-{D}ependent {R}iccati feedback},
    \newblock \bibinfo{journal}{IEEE Control Systems Letters} \bibinfo{volume}{7}
    (\bibinfo{year}{2023}) \bibinfo{pages}{3012--3017}.
    \DOIprefix\doi{10.1109/lcsys.2023.3291231}.
    \bibitem[{Ragab and Fayed(2017)}]{Ragab2017}
    \bibinfo{author}{S.~A. Ragab}, \bibinfo{author}{H.~E. Fayed},
    \bibinfo{title}{Introduction to finite element analysis for engineers},
    \bibinfo{publisher}{Taylor \& Francis Group}, \bibinfo{year}{2017}.
    \bibitem[{Reddy(2004)}]{Reddy2004}
    \bibinfo{author}{J.~N. Reddy}, \bibinfo{title}{An introduction to nonlinear
        finite element analysis}, \bibinfo{publisher}{Oxford University Press},
    \bibinfo{year}{2004}.
    \DOIprefix\doi{10.1093/acprof:oso/9780198525295.001.0001}.
    \bibitem[{IT(2025)}]{HealthIT}
    \bibinfo{author}{H.~IT}, \bibinfo{title}{How big is the internet and how do we
        measure it?},
    \bibinfo{howpublished}{\url{https://healthit.com.au/how-big-is-the-internet-and-how-do-we-measure-it/}},
    \bibinfo{year}{2025}. \bibinfo{note}{Accessed: 2025-10-31}.
    \bibitem[{Hespanha(2018)}]{Hespanha2018}
    \bibinfo{author}{J.~Hespanha}, \bibinfo{title}{Linear systems theory},
    \bibinfo{edition}{second} ed., \bibinfo{publisher}{Princeton University
        Press}, \bibinfo{year}{2018}. \DOIprefix\doi{10.23943/9781400890088}.

\end{thebibliography}
\end{document}